\documentclass[11pt,a4paper,twoside,leqno]{amsart}

\usepackage[a4paper,textwidth=16cm,textheight=21cm,centering]{geometry}
\usepackage[colorlinks=true,citecolor=red,linkcolor=blue,urlcolor=blue,hypertexnames=false]{hyperref}
\usepackage{amssymb}
\usepackage{amsmath}
\usepackage{mathrsfs}
\usepackage{amsthm}
\usepackage{amsfonts}
\usepackage{mathtools}
\usepackage[T1]{fontenc}
\usepackage{newunicodechar}
\newunicodechar{‑}{-}
\IfFileExists{newtxtext.sty}{%
  \usepackage{newtxtext}
  \usepackage{newtxmath}
}{\usepackage{mathptmx}}
\usepackage{enumerate}
\usepackage{enumitem}
\usepackage{latexsym}
\usepackage{color}
\usepackage{graphicx}
\usepackage{microtype}
\usepackage{indentfirst}
\usepackage{tikz}
\usepackage{float}
\usepackage{booktabs}
\usepackage{tabularx}
\usepackage{array}
\usepackage{float}
\usetikzlibrary{
	arrows.meta,
	calc,
	angles,
	quotes,
	positioning,
	decorations.pathreplacing
}
\hypersetup{colorlinks=true,citecolor=red,linkcolor=blue,urlcolor=blue,pdfborder={0 0 0}}
\makeatletter
\def\subsection{\@startsection{subsection}{2}%
	\z@{.7\linespacing\@plus.2\linespacing}%
	{.5\linespacing}%
	{\normalfont\bfseries}}
\makeatother
\allowdisplaybreaks
\newtheoremstyle{boldplain}%
{3pt}{3pt}{\itshape}{}%
{\bfseries\upshape}{.}{0.5em}%
{\thmname{\bfseries #1}\thmnumber{\bfseries\ #2}%
 \thmnote{\bfseries\ (#3)}}

\newtheoremstyle{bolddefinition}%
{3pt}{3pt}{\normalfont}{}%
{\bfseries\upshape}{.}{0.5em}%
{\thmname{\bfseries #1}\thmnumber{\bfseries\ #2}%
 \thmnote{\bfseries\ (#3)}}

\theoremstyle{boldplain}
\newtheorem{theorem}{Theorem}[section]

\newtheorem{lemma}[theorem]{Lemma}

\theoremstyle{bolddefinition}

\numberwithin{equation}{section}
\newtagform{blue}[\textcolor{blue}]{\textcolor{blue}{(}}{\textcolor{blue}{)}}
\usetagform{blue}

\newcommand{\R}{\mathbb R}
\newcommand{\Sph}{\mathbb S}
\newcommand{\sig}{\sigma}
\newcommand{\1}{\mathbf 1}

\newcommand{\abs}[1]{\left|#1\right|}
\newcommand{\norm}[1]{\left\|#1\right\|}

\newcommand{\dist}{\operatorname{dist}}
\newcommand{\supp}{\operatorname{supp}}

\usepackage{needspace}
\hypersetup{
  pdftitle={Weak type (1,1) bounds for rough oscillatory singular integrals},
  pdfauthor={Yanping Chen and Zhengyang Ji},
  pdfsubject={Harmonic analysis},
  pdfkeywords={oscillatory singular integrals, rough kernels, weak type (1,1), polynomial phases}
}
\newcommand{\Z}{\mathbb Z}
\newcommand{\one}{\mathbf1}
\newcommand{\pv}{\operatorname{p.v.}}
\newcommand{\coef}{\mathrm{coef}}
\newcommand{\nc}{\mathrm{nc}}
\newcommand{\weak}[1]{\left\|#1\right\|_{L^{1,\infty}}}
\newcommand{\HH}{\mathcal H}
\theoremstyle{boldplain}

\newcommand{\submissionauthorinfo}{%
  \par\vspace{0.2\baselineskip}%
  \begin{center}
    \footnotesize
    \begin{minipage}{0.94\textwidth}
      \centering
      \textsuperscript{1}Department of Mathematics, Northeastern University,
      Shenyang 110004, China\\[0.3em]
    \end{minipage}
  \end{center}%
}
\makeatletter
\g@addto@macro\@setauthors{\submissionauthorinfo}
\let\@setaddresses\@empty
\makeatother

\begin{document}
\arraycolsep=1pt

\title{Weak type $(1,1)$ bounds for rough oscillatory singular integrals}

\author[Yanping Chen]{Yanping Chen$^{1,\ast}$[\href{https://orcid.org/0000-0002-8879-7776}{0000-0002-8879-7776}]}

\email{yanpingch@126.com}

\author[Zhengyang Ji]{Zhengyang Ji$^1$[\href{https://orcid.org/0009-0009-0505-7480}{0009-0009-0505-7480}]}
\address{Department of Mathematics,
	Northeastern University,
	Shenyang 110004, China}
\email{zhengyang.ji.zjut@gmail.com}

\thanks{\textsuperscript{$\ast$}Corresponding author: Yanping Chen. Email: \href{mailto:yanpingch@126.com}{\nolinkurl{yanpingch@126.com}}.}
\thanks{The research was partially supported by the National Natural Science Foundation of China (Grant Nos. 12525105 and 12371092).}
\thanks{The authors sincerely thank Professor Jiecheng Chen, Professor Yong Ding, Professor Shanzhen Lu, Professor Huoxiong Wu and Professor Dunyan Yan, who patiently endured my elementary questions regarding the rough oscillatory singular integral operators.}
\date{October 9, 2026}

\subjclass[2010]{42B20, 42B25}
\keywords{oscillatory singular integrals; rough kernels; weak type $(1,1)$ bounds; polynomial phases}

\begin{abstract}
In their seminal work, F. Ricci and E. M. Stein[Proc. Nat. Acad. Sci. USA 83, 1--3 (1986)] introduced the oscillatory singular integral
 \[
T_Pf(x)=\pv\int_{\R^n}e^{iP(x,y)}\frac{\Omega(x-y)}{|x-y|^n}f(y)\,dy,
\]
 and S. Chanillo and M. Christ [Duke Math. J. 55, 141--155 (1987)] established its weak-type \((1,1)\) inequality for  smooth  $\Omega$. For the rough variant, however, the weak‑type \((1,1)\) inequality remained a long-standing open question. In the present paper, we settle this problem: specifically, we 
prove that the rough oscillatory singular integral is of weak type
$(1,1)$ with $\Omega\in L\log^+\!L(\mathbb{S}^{n-1})$. Moreover, we establish weak‑type \((1,1)\) bounds for the rough maximal oscillatory operator associated with \(T_P\), a result of independent interest.
\end{abstract}

\maketitle

\renewcommand{\contentsname}{Contents}
\setcounter{tocdepth}{1}
\tableofcontents

\section{Introduction}\label{sec:intro}

Let \(\mathbb{R}^n\) be the $n$-dimensional Euclidean space with \(n\ge 2\), and let
\(\mathbb{S}^{n-1}=\{x\in\mathbb{R}^n: |x|=1\}\)
be the unit sphere in \(\mathbb{R}^n\). Suppose $\Omega$ is homogeneous of degree zero on $\mathbb{R}^n\setminus\{0\}$, i.e.
\begin{align}\label{eq:homogeneous}
	\Omega(tx)=\Omega(x) ~~\text{for~any}~t>0 ~\text{and}~ x \in \mathbb{R}^{n} \backslash\{0\},
\end{align}
and satisfies the cancellation condition
\begin{align}\label{eq:mean}
	\int_{\mathbb{S}^{n - 1}}\Omega(x')d \sigma(x') = 0.
\end{align}
\smallskip

Oscillatory integrals occupy a central position in harmonic analysis, featuring prominently in Fourier restriction theory, multiplier operators, square-function and decoupling theory, singular integrals, maximal operators, and time-frequency analysis. They provide an indispensable tool for investigating geometric properties rooted in curvature and orthogonality, exploiting oscillatory cancellation to translate the curvature, degeneracy order, and transversality of phases into decay estimates, operator boundedness, and regularity gains. Such estimates further forge connections with dispersive equations, spectral theory, fractal distance sets, and exponential-sum problems in number theory (see\cite{Bo, BG, Car, CKS1986, Fefferman,  FS, Guth, Hormander, KPV,  MVV, Sch, Sch0, Sogge,SW2001} ). Many important operators in harmonic analysis are some versions of oscillatory integrals, such as the Fourier transform, Bochner-Riesz means, Radon transform and so on.
\smallskip
\smallskip

F. Ricci and E. M. Stein \cite{RS1986, RS1987} introduced the following oscillatory singular integral, initially defined for smooth, compactly supported functions $f$:
\begin{align}\label{eq:orig11}
	T_Pf(x) = \text { p.v. } \int_{\mathbb{R}^n} e^{iP(x,y)} \frac{\Omega(x-y)}{|x-y|^n}f(y) dy,
\end{align}
where $P(x,y)$ is a real-valued polynomial on
$\mathbb{R}^n\times\mathbb{R}^n$.
This operator was also discussed in E.~M.~Stein's address
at the 1986 International Congress of Mathematicians,
in which he presented $L^p$ bounds for $1<p<\infty$;
see \cite[Section~8, Corollary~3]{Stein1987ICM}. It was precisely in the analysis of operators of that type that the class of
oscillatory singular integrals treated here arose; see \cite{ IKM2025, KL2018, KL2020, LS2017, Pan1991, PS1986, PS1998, RS1986,  Sato2000, Lie2009} and the references
listed there. 
\smallskip

By a rough integral, we mean that no regularity is assumed for the kernel \(\Omega\), only its size. In recent years a variety of rough operators related to the singular integrals for Calder\'{o}n and Zygmund
have been studied and  their boundedness properties have played an
important role in the subsequent development of harmonic analysis; see, for example, \cite{CDHL, CDH,  DR, CG, FS, FCS,GS, GHS, GHP, HLTX, HRT, Ho, SS, SS1, S1, ST}.
Using the method of rotations of Calder\'{o}n and Zygmund \cite{CZ1956} together with van der Corput's lemma, one can show that the rough oscillatory singular integral $T_P$ extends
to a bounded operator on $L^p(\mathbb{R}^n)$, for $1<p<\infty$, provided $\Omega\in L\log^+\!L(\mathbb{S}^{n-1})$ (see \cite{LZ1992, JL1995, FP1997, Ojanen2000}).
One might therefore hope that it is also of weak type $(1,1)$. However, the \(L^1\)-theory is far more difficult.

\smallskip
E. M. Stein \cite{Stein1993,Stein1987ICM} identified three central classes of operators in modern harmonic analysis:
	\begin{enumerate}
	\item {\emph{maximal averages}}; \item{\emph{singular integrals;}} \item { \emph{oscillatory integrals}}.\end{enumerate} For the first two classes, the weak-type \((1,1)\) boundedness of their rough variants was established under the $L\log^+\!L$ assumption (see  \cite{CR1988}, \cite{Seeger1996}).
However, the weak type (1,1) boundedness of rough oscillatory singular integral operators $T_P$ has remained an open problem
since then, even for $\Omega \in L^\infty (\mathbb{S}^{n-1})$. This longstanding problem was explicitly stated in \cite{DL2019}.
The long-standing open problem has remained unresolved mainly due to
three difficulties:\begin{enumerate}
	\item\emph{ the oscillatory nature}; \item\emph{the roughness of the kernel;} \item
\emph{the non-convolutional structure}.\end{enumerate}
 Even for the oscillatory integrals \(T_P\) with smooth kernel \(\Omega\), establishing their weak-type \((1,1)\) boundedness remains highly nontrivial. In 1987, Chanillo and Christ \cite{CC1987} employed a modified geometric argument and the van der Corput lemma to prove weak-type \((1,1)\) inequality for singular integral operators  \(T_P\) with sufficiently smooth kernel \(\Omega \in C^1(\Bbb S^{n-1})\).

\smallskip
The preceding results are summarized in Table~\ref{tab:weak11-rough}.
\begin{table}[H]
	\centering
	\caption{Progress on weak type $(1,1)$ estimates for rough operators.}
	\label{tab:weak11-rough}
	\renewcommand{\arraystretch}{1}
	\begin{tabularx}{0.95\textwidth}{
			>{\centering\arraybackslash}p{0.25\textwidth}
			>{\raggedright\arraybackslash}X}
		\toprule
		\textbf{Rough operators}
		&
		\centering\arraybackslash
		\textbf{Weak type $(1,1)$ results}
		\\
		\midrule
		
		\textbf{Rough maximal averages}
		&
		Christ~\cite{Christ1988} (1988):
		$n=2$, $\Omega\in L^q(\mathbb S^1)$, $q>1$;
		\newline
		Christ--Rubio de Francia~\cite{CR1988} (1988):
		$n\ge2$, $\Omega\in L\log^+\! L(\mathbb S^{n-1})$.
		\\
		
		\midrule
		
		\textbf{Rough singular integrals}
		&
		Hofmann~\cite{Hofmann1989} (1988):
		$n=2$, $\Omega\in L^q(\mathbb S^1)$, $q>1$;
		\newline
		Christ--Rubio de Francia~\cite{CR1988} (1988):
		$n=2$, $\Omega\in L\log^+\! L(\mathbb S^1)$;
		\newline
		Seeger~\cite{Seeger1996} (1996):
		$n\ge2$, $\Omega\in L\log^+\! L(\mathbb S^{n-1})$;
		\newline
		Tao~\cite{Tao1999} (1999):
		$\Omega\in L\log^+\! L$ on arbitrary homogeneous groups.
		\\
		
		\midrule
		
		\textbf{Rough oscillatory integrals}
		&
		
		$\Omega\in L^\infty (\mathbb S^{n-1})$?
		
		\\
		
		\bottomrule
	\end{tabularx}
\end{table}
In this paper, we resolve this open problem. Our main result can be stated as follows.

\begin{theorem}\label{thm:main}
	Let $\Omega \in L\log^{+}\!\!L(\mathbb{S}^{n-1})$ and satisfy \eqref{eq:homogeneous} and \eqref{eq:mean}. Suppose that $P(x,y)$ is a real valued polynomial defined on $\mathbb{R}^n\times \mathbb{R}^n$. Let the operator $T_P$ be defined as
	\begin{align*}
		T_Pf(x) = \hbox { p.v. } \int_{\mathbb{R}^n} e^{iP(x,y)}\frac{ \Omega(x-y)}{|x-y|^n}f(y) dy.
	\end{align*}
	Then for any $\lambda >0$, we have
	\begin{equation}\label{eq:target}
		\lambda |\{x\in\mathbb{R}^n:|T_P f(x)|>\lambda\}|
		\le C (1+\|\Omega\|_{L\log^+\!\!L})\|f\|_1,
	\end{equation}
	where $C$ depends only on the total degree of $P$ and $n$, but not on the coefficients of $P.$
\end{theorem}

When $P (x,y) = 0$, $T_P$ is the singular integral operator with homogeneous kernel $\Omega$ introduced by Calder\'on and Zygmund \cite{CZ1956} in 1956
\[
T_{\Omega} f(x) = \mathrm{p.v.}\int \frac{\Omega(x-y)}{|x-y|^n} f(y)\,dy
\]
and established its $L^p$-boundedness for $1<p<\infty$ under the assumptions that either $\Omega$ is odd and $\Omega\in L^1(\mathbb{S}^{n-1})$, or $\Omega$ is even and $\Omega\in L\log^+\!L(\mathbb{S}^{n-1})$ by the method of rotations. However, the rough-operator \(L^1\) theory  is considerably harder.
It was only in 1988 that Christ and Rubio de Francia \cite{CR1988} established the weak‑type \((1,1)\) inequality for this operator in two dimensions, assuming \(\Omega\in L\log^+\!L(\mathbb{S}^1)\). Hofmann \cite{Hofmann1989} independently established this result for \(\Omega\in L^q(\mathbb{S}^1)\), where \(1<q\le\infty\). Later, in 1996, Seeger \cite{Seeger1996} employed microlocal decomposition to establish weak-type \((1,1)\) inequality in all dimensions when \(\Omega\in L\log^+\!L(\mathbb{S}^{n-1})\). In 1999, Tao \cite{Tao1999} developed the \(TT^*\) method and established the weak-type \((1,1)\) inequality for rough singular integral operators on homogeneous groups.
In \cite{Christ1988} and \cite{CR1988}, a weak-type $(1,1)$ inequality was also established for the less singular maximal operator
\[
M_{\Omega} f(x)=\sup_{r>0} r^{-n}\int_{|y|\le r}\left|\Omega\!\left(\frac{y}{|y|}\right)\right|\,|f(x-y)|\,dy
\]
in all dimensions, once more under the assumption that $\Omega\in L\log^+\!L(\mathbb{S}^{n-1})$.

\smallskip

Nevertheless, several fundamental obstructions  prevent one from studying  rough oscillatory singular integrals.

\smallskip

\begin{enumerate}
	\item
	Although Christ-Rubio de Francia \cite{CR1988} and Seeger \cite{Seeger1996}  proved the weak $(1,1)$ boundedness of
	rough maximal operators and rough singular integral operators,
	respectively, substantial difficulties arise when adapting their arguments to rough oscillatory singular integrals.
	These difficulties arise because differentiating the \emph{oscillatory
	factor}  produces terms that grow with the scale and may lead to divergence in the high-scale part.
	\smallskip

	\item
	Moreover, weak \(L^1\) is \emph{not a normed linear space}, so the method of rotations is no longer applicable, and one cannot isolate the kernel \(\Omega\) via polar coordinates. The most difficult scenario is the following: the van der Corput lemma—the primary tool for oscillatory integrals—imposes certain regularity requirements on the amplitude, which the rough kernel \(\Omega\) fails to satisfy. By contrast, the kernel in the setting of Chanillo and Christ \cite{CC1987} is sufficiently smooth, whence the van der Corput lemma applies directly.
	\smallskip

	\item
	Finally, the mixed terms of $P(x,y)$ couple the variables $x$ and $y$
	and cannot be eliminated by multiplication with factors depending on
	a single variable. The resulting \emph{nonconvolutional structure}
	prevents the direct application of the estimates for rough
	convolution singular integrals.
\end{enumerate}

Accordingly, new methods and techniques are needed to overcome these  obstructions. In this paper, we develop \emph{two  new methods} and may prove useful in other problems.
\begin{enumerate}
\item The first method, \emph{two-center rough oscillatory estimates}, yields decay from oscillation while avoiding differentiation of the rough kernel factors. We construct coordinates adapted simultaneously to the two centers so that the two rough kernel factors depend on separate variables. This reduces the correlation integral to a pairing with a one-dimensional oscillatory operator. We then relate the mixed derivative of the transformed phase to the coefficients of the original polynomial phase. In regions where this derivative is small, we bound the contribution via polynomial sub-level set estimates. On the remaining domains, a $TT^*$ argument combined with van der Corput's lemma produces oscillatory decay, again without differentiating the rough kernel factors.
\smallskip

\item The second method, \emph{Fourier expansion of the mixed phase}, decomposes the exponential of the mixed phase into a sum of products of single‑variable functions and, combined with a suitable mean‑value correction, allows us to invoke convolution estimates.

\end{enumerate}
 \smallskip

Finally, we shall get a similar result for the maximal operator associated with $T_P.$
Additional new arguments are required to address maximal truncations.
\begin{theorem}\label{thm1}
	Let $\Omega \in L\log^{+}\!\!L(\mathbb{S}^{n-1})$ and satisfy \eqref{eq:homogeneous} and \eqref{eq:mean}. Suppose that $P(x,y)$ is a real valued polynomial defined on $\mathbb{R}^n\times \mathbb{R}^n$. Let the operator $T_P^*$ be defined as
	\begin{align*}
		T_P^*f(x) =\sup_{\varepsilon>0} \bigg|\int_{|x-y|>\varepsilon} e^{iP(x,y)}\frac{ \Omega(x-y)}{|x-y|^n}f(y) dy\bigg|.
	\end{align*}
	Then for any $\lambda >0$, we have
	\begin{equation}\label{eq:max-target}
		\lambda |\{x\in\mathbb{R}^n:|T_P^* f(x)|>\lambda\}|
		\le C (1+\|\Omega\|_{L\log^+\!\!L})\|f\|_1,
	\end{equation}
	where $C$ depends only on the total degree of $P$ and $n$, but not on the coefficients of $P.$
\end{theorem}
\smallskip

\subsection{Outline of the proof of Theorem 1.1}

The usual tools of harmonic analysis such as Calder\'{o}n-Zygmund decomposition,
Littlewood-Paley theory, $TT^*$ methods and  microlocal decomposition  have been developed in this context. e.g., \cite{CR1988, Seeger1996}. 
\smallskip

We begin by introducing the relevant notation.
Fix a nonnegative real-valued function $\phi\in C_c^\infty((1/2,2))$ such that
\begin{equation}\label{eq:partition}
	\sum_{j\in\Z}\phi(2^{-j}t)=1,\qquad t>0,
\end{equation}
and define the annular kernels and the corresponding operators by
\begin{equation}\label{eq:Tj}
	\begin{aligned}
		k_j(x)&=\frac{\phi(2^{-j}|x|)}{|x|^n},
		\qquad K_j(x)=\Omega(x')k_j(x),\\
		T_jf(x)&=\int_{\R^n}e^{iP(x,y)}K_j(x-y)f(y)\,dy.
	\end{aligned}
\end{equation}
\smallskip
We split  $T_P$ into its local-scale and positive-scale parts:
\[
T_{\le0}f=\sum_{j\le0}T_jf,
\qquad
T_{>0}f=\sum_{j>0}T_jf.
\]

For \(T_{\le 0}\), we invoke Seeger’s result \cite{Seeger1996}, together with the induction hypothesis and the Calder{\'{o}}n–Zygmund theory, to show that \(T_{\le 0}\) is of weak type \((1,1)\). To treat $T_P$ it suffices to treat $T_{>0}.$
Given $\lambda>0$, set $\tau=\lambda/(1+\|\Omega\|_{L\log^+\!\!L})$.
We shall perform a Calder\'on--Zygmund decomposition
$f=g+\sum_Qb_Q$ at height $\tau$ with phase-adapted cancellation, so that
\[
\operatorname{supp}b_Q\subset Q,\qquad
\norm{b_Q}_1\le C\tau|Q|,
\qquad
\int_Qe^{iP(c_Q,y)}b_Q(y)\,dy=0,
\]
where $c_Q$ denotes the center of $Q$. Grouping the bad functions by sidelength, we write
\(
B_k=\sum_{\ell(Q)=2^k}b_Q.
\)
In $T_jB_{j-s}$, the parameter $s$ measures the separation between
the annular scale and the scale of the bad cubes.
\smallskip

For small constants
$\varrho,\eta_0>0$ to be chosen below, we divide the positive scales
into $j>\varrho s$ and $1\le j\le\varrho s$, and set
\(
\Omega_s
=\Omega\one_{\{|\Omega|\le2^{\eta_0s}\norm\Omega_1\}}.
\)
Let $T_{j,\Omega_s}$ denote the operator obtained from \eqref{eq:Tj}
by replacing $\Omega$ with $\Omega_s$.

\smallskip

We first treat the local scales
by induction on the mixed bidegree. For positive-scale parts,  we perform a Calder\'on--Zygmund
decomposition with phase-adapted cancellation. The good function is controlled
by the known $L^2$ estimate. The main task is therefore to estimate the
bad functions at  positive scales.

\smallskip

At scale $R=2^j$ and near a local center $x_h$, define the mixed phase
\[
\Pi_h(x,y)
=P(x,y)-P(x,x_h)-P(x_h,y)+P(x_h,x_h).
\]
Write
\[
\Pi_h(x,y)
=
\sum_{\substack{|\alpha|\ge1\\|\beta|\ge1}}
p_{\alpha\beta}^h
\left(\frac{x-x_h}{R}\right)^\alpha
\left(\frac{y-x_h}{R}\right)^\beta,
\qquad
\mathfrak M_h
=
\sum_{\substack{|\alpha|\ge1\\|\beta|\ge1}}
|p_{\alpha\beta}^h|.
\]
Thus $\mathfrak M_h$ measures the size of the coefficients of the mixed phase at this scale.
We denote the Fourier frequency indices by
$\gamma=(\gamma_1,\gamma_2)\in\Z^{2n}$;
here $L_0$ is the fixed parameter specifying the periodic box, and $N_0=4n+6$ is the derivative order used below.
\smallskip

With this notation, we describe the two main steps for estimating the contribution of the bad functions
to the positive-scale part $T_{>0}$: two-center rough oscillatory estimates at high positive scales,
and Fourier expansion of the mixed phase and mean correction at low positive scales.
In the latter case, the dependence of the Fourier coefficient bounds on the
mixed phase leads to a division according to the size of its coefficients.
\smallskip

\noindent\textbf{1. Two-center rough oscillatory estimates.}
Expanding the square of the $L^2$ norm of an annular operator produces a correlation kernel
involving  the phase difference $P(x,y)-P(x,z)$ and two rough functions centered at $y$ and $z$, respectively:
\(
\Omega_s\!\left(\frac{x-y}{|x-y|}\right),
\overline{\Omega_s\!\left(\frac{x-z}{|x-z|}\right)}.
\)
Standard polar coordinates fix only one of the two directions,
while the other still varies with the oscillatory variable. The most difficult scenario is the following: 
one cannot directly apply van der Corput’s lemma, which imposes certain regularity assumptions on the amplitude.
We draw on the geometric change of variables of
Christ--Rubio de Francia~\cite[\S4, proof of Lemma~4.4]{CR1988}
to separate the two rough factors in coordinates adapted to both centers.
However,  this separation alone gives no quantitative lower bound
for the mixed derivative of the transformed phase.
We relate this derivative to the original polynomial coefficients through
an algebraic identity preserving the highest-degree terms.
Polynomial sublevel set estimates based on the
Carbery--Wright inequality~\cite[Theorem~8]{CW2001}
control the regions where the derivative is small,
while a $TT^*$ argument on their complement yields oscillatory decay
without differentiating either rough factor.
\smallskip

On a two-dimensional section passing through the two centers, we write
\[
x=z+r(pe+t\theta),\qquad
e=\frac{y-z}{|y-z|},\quad
H=\frac{|y-z|}{r},\quad
\theta\in e^\perp\cap\Sph^{n-1}.
\]
where $y\ne z$, $r>0$ is the annular scale under consideration,
and $p\in\R$, $t>0$. Set
\[
u=\frac pt,\qquad
v=\frac{p-H}{Lt},\qquad L=\max(1,H).
\]
Then
\[
\frac{x-z}{|x-z|}
=\frac{ue+\theta}{\sqrt{1+u^2}},
\qquad
\frac{x-y}{|x-y|}
=\frac{Lve+\theta}{\sqrt{1+L^2v^2}}.
\]
For fixed $\theta$, the two rough functions become
\[
\Omega_s\!\left(\frac{Lve+\theta}{\sqrt{1+L^2v^2}}\right),\qquad
\overline{\Omega_s\!\left(\frac{ue+\theta}{\sqrt{1+u^2}}\right)},
\]
and hence depend only on $v$ and $u$, respectively.
We apply a one-dimensional oscillatory operator to the former and integrate
its product with the latter. By the Cauchy--Schwarz inequality,
it suffices to bound the $L^2$ norm of this one-dimensional operator.
On the region where the mixed derivative is nondegenerate,
a $TT^*$ argument and van der Corput's lemma for integrals with smooth amplitudes
yield oscillatory decay without differentiating either rough function.
\smallskip

To obtain this nondegeneracy, we must relate the mixed derivative after the
change of variables to the coefficients of the original phase.
If $q_\theta(p,t)$ is the phase on the section and $\Phi_\theta$
is the transformed phase, then
\[
\partial_u\partial_v\Phi_\theta
=-\frac{Lt^2}{H^2}(E_H+1)E_0q_\theta,
\qquad
E_0=p\partial_p+t\partial_t,\quad
E_H=(p-H)\partial_p+t\partial_t.
\]
For the homogeneous component $q_k$ of degree $k$, we have
\[
(E_H+1)E_0q_k
=k(k+1)q_k-kH\partial_pq_k.
\]
This triangular relation preserves the highest-degree coefficients and
quantitatively relates the mixed derivative to the coefficients of the original phase.
Combining sublevel set estimates in the section direction with those for the mixed derivative,
we isolate the region where the derivative is sufficiently small and apply the preceding
one-dimensional $L^2$ estimate on its complement. This yields the two-center
rough oscillatory estimate and also accommodates distinct annular scales.
Since the bad functions satisfy mass bounds only on entire cubes,
the degenerate region must be estimated together with all the bad cubes that intersect it.
This allows us to convert decay of the correlation kernel into an $L^2$ estimate
for the bad functions.
\smallskip

Applied to the truncated kernel $\Omega_s$, the two-center estimate gives
\[
\norm{\sum_{j\in\mathcal J}T_{j,\Omega_s}B_{j-s}}_2^2
\le
C\tau\norm\Omega_1^2\,2^{-(\kappa_1-2\eta_0)s}
\sum_{j\in\mathcal J}\norm{B_{j-s}}_1,
\qquad
\mathcal J\subset\{j>\varrho s\},
\]
where $\mathcal J$ is any finite set and $\kappa_1>0$.
Choosing $\eta_0$ sufficiently small allows the scale decay to absorb the growth of the truncated kernel.

\medskip
\noindent\textbf{2. Fourier expansion of the mixed phase and mean correction.}
At low positive scales, the annular scale need not increase with $s$,
so the lower bound used in the high positive scale argument is no longer
available. We instead use the separation between the annular scale
$R=2^j$ and the bad-cube sidelength $R\cdot 2^{-s}$ through Seeger's convolution
estimates(\cite{Seeger1996}). To obtain a convolution structure, we retain the phases
depending on a single variable as unimodular multipliers and expand
the mixed phase. After inserting fixed smooth cutoffs, we obtain
\[
e^{i\Pi_h(x,y)}
=
\sum_{\gamma=(\gamma_1,\gamma_2)\in\Z^{2n}}
c_\gamma^h
e^{i\pi\gamma_1\cdot(x-x_h)/(L_0R)}
e^{i\pi\gamma_2\cdot(y-x_h)/(L_0R)}.
\]
This identity is used on the support of the localized operator, and the coefficients satisfy
\[
\sum_\gamma
|c_\gamma^h|(1+|\gamma_2|+\mathfrak M_h)
\le C(1+\mathfrak M_h)^{N_0},
\] where $\mathfrak M_h
=
\sum_{\substack{|\alpha|\ge1\\|\beta|\ge1}}
|p_{\alpha\beta}^h|.$
The two exponential factors in each mode depend only on $x$ and $y$, respectively;
the remaining kernel therefore possesses a convolution structure.
\smallskip

The expansion is valid for every $\mathfrak M_h$, but the coefficient
bound is not uniform in the coefficients of $P$. Seeger's estimates
provide a factor $2^{-s/4}$ in the $L^2$ norm of the main term and
the $L^1$ norm of the remainder. Thus Fourier summation preserves decay
when $\mathfrak M_h\le2^{\sigma s}$ and $\sigma>0$ is sufficiently small,
since
\[
2^{-s/4}(1+\mathfrak M_h)^{N_0}
\le C2^{-(1/4-N_0\sigma)s}.
\]
This is the reason for dividing the localized terms according to the
size of $\mathfrak M_h$. We use Fourier expansion and mean correction
when $\mathfrak M_h\le2^{\sigma s}$. When
$\mathfrak M_h>2^{\sigma s}$, we instead apply the equal-scale
two-center estimate to obtain decay from the large mixed coefficients.
\smallskip

In the small-phase region, the bad functions obtained after Fourier
expansion need not have ordinary mean zero.
However, the phase-adapted cancellation imposed earlier quantitatively controls their means.
If
\[
f_{Q,\gamma_2}(y)
=
e^{i[P(x_h,y)-P(x_h,x_h)]}
e^{i\pi\gamma_2\cdot(y-x_h)/(L_0R)}b_Q(y),
\]
then the cancellation condition
$\int_Qe^{iP(c_Q,y)}b_Q(y)\,dy=0$ gives
\[
\left|\int_Qf_{Q,\gamma_2}(y)\,dy\right|
\le
C2^{-s}(|\gamma_2|+\mathfrak M_h)\norm{b_Q}_1.
\]
This estimate involves only the variation of the mixed phase and the Fourier mode;
no derivative acts on the pure $y$-phase.
After subtracting the mean of each bad function, we may apply Seeger's microlocal estimates
for convolution operators. The mean-correction error is controlled by the preceding
bound and the weighted sum of the Fourier coefficients. Finally, we choose
$\sigma,\eta_0$ sufficiently small so that the resulting decay absorbs the growth
of the Fourier coefficients and the truncated kernel. The total contribution of
$\Omega-\Omega_s$ is controlled by the $L\log^+L$ assumption.

\subsection{Organization and Notation} The paper is organized as follows.
In Section \ref{sec:prelim}, we collect the necessary preliminaries.
Section \ref{sec:main-proof} proves Theorem \ref{thm:main}, subject to the
high and low positive scale estimates stated there.
Section \ref{sec:high} establishes Theorem \ref{prop:high}.
Section \ref{sec:reduction} constructs the low positive scale decomposition
and proves Theorem \ref{prop:lowfull}, using the directional projection
estimates established in Section \ref{sec:microlocal}.
Finally, Section \ref{sec:maximal} proves Theorem \ref{thm1}
by treating local maximal truncations and positive dyadic tails.
\smallskip

We now fix our notation.\begin{enumerate} \item For a multi-index $\alpha\in\Z_+^n$, write
$|\alpha|=\sum_i\alpha_i$, 
$x^\alpha=\prod_i x_i^{\alpha_i}$, 
$\partial^\alpha=\prod_i\partial_i^{\alpha_i}$.
Fix an upper bound $D\ge\deg P$ for the total degree.
The largest mixed bidegree $(m,l)$ in lexicographic order satisfies $m+l\le D$.

\item Set
\[
\norm\Omega_1
:=\int_{\Sph^{n-1}}|\Omega(\theta)|\,d\sigma(\theta),\,\,
\|\Omega\|_{L\log^+\! L}
 :=
 \int_{\Sph^{n-1}}|\Omega(\theta)|\log\!\bigl(2+|\Omega(\theta)|\bigr)d\sig(\theta)\]
and 
$$A:=1+\norm{\Omega}_{L\log^+L}.$$
\smallskip

\item The letter $C$ denotes a positive constant that may change from one occurrence to the next.
Unless otherwise stated, it depends only on $n,D$ and the fixed cutoff functions
and geometric parameters, but not on the individual coefficients of $P$.
Constants in the general lemmas may also depend on the fixed parameters specified there.
\item  For a polynomial $q(x)=\sum_\alpha c_\alpha x^\alpha$, write
\[
\norm q_{\coef}=\sum_\alpha|c_\alpha|,
\qquad
\norm q_{\nc}=\sum_{|\alpha|\ge1}|c_\alpha|.
\]
\end{enumerate}

\section{Preliminaries}\label{sec:prelim}

We use the following phase-adapted Calder\'on--Zygmund
decomposition, whose cancellation condition is tailored to
the phase in the oscillatory integral.

\begin{lemma}\label{lem:adapted-CZ}
Let $f\in L^1(\R^n)$ and $\tau>0$, and let $\vartheta_Q$ be a real-valued measurable function on each cube. There exists a family $\mathcal Q$ of pairwise disjoint dyadic cubes such that
\begin{equation}\label{eq:pre-CZ-cubes}
 \tau<\frac1{|Q|}\int_Q|f|\le2^n\tau,
 \qquad \sum_{Q\in\mathcal Q}|Q|\le\tau^{-1}\norm f_1,
 \qquad |f|\le\tau\quad\text{a.e. on }\Bigl(\bigcup_QQ\Bigr)^c.
\end{equation}
Define
\[
 g(y)=\begin{cases}
 e^{-i\vartheta_Q(y)}|Q|^{-1}\displaystyle\int_Qe^{i\vartheta_Q(z)}f(z)\,dz,&y\in Q,\\
 f(y),&y\notin\bigcup_QQ,
 \end{cases}
 \qquad b_Q=(f-g)\one_Q.
\]
Then
\begin{gather}\label{eq:pre-CZ-bounds}
 \norm g_\infty\le2^n\tau,\quad \norm g_1\le\norm f_1,
 \quad\norm g_2^2\le2^n\tau\norm f_1,\\
 \norm{b_Q}_1\le2\int_Q|f|\le2^{n+1}\tau|Q|,
 \quad\sum_Q\norm{b_Q}_1\le2\norm f_1,
 \quad\int_Qe^{i\vartheta_Q(y)}b_Q(y)\,dy=0.\notag
\end{gather}

\end{lemma}

\begin{proof}
	The cubes  $\{Q\} $ are obtained from the standard Calder\'on--Zygmund decomposition; see \cite{Grafakos2014, Stein1993}.
	Choose the maximal dyadic cubes satisfying
	\(
	\frac1{|Q|}\int_Q|f|>\tau.
	\)
	They are pairwise disjoint. If $Q^+$ denotes the dyadic
	parent of $Q$, then
	\[
	\tau<\frac1{|Q|}\int_Q|f|
	\le\frac1{|Q|}\int_{Q^+}|f|
	\le 2^n\tau.
	\]
	Consequently,
	\[
	\sum_{Q\in\mathcal Q}|Q|
	\le \tau^{-1}\|f\|_1.
	\]
	The dyadic differentiation theorem also gives
	$|f|\le\tau$ almost everywhere outside
	$\bigcup_{Q\in\mathcal Q}Q$.
	
	For each $Q\in\mathcal Q$, set
	\[
	c_Q=\frac1{|Q|}\int_Q e^{i\vartheta_Q(z)}f(z)\,dz.
	\]
	Then $g=e^{-i\vartheta_Q}c_Q$ on $Q$, and hence
	\[
	|g(y)|\le\frac1{|Q|}\int_Q|f|,
	\qquad
	\int_Q|g(y)|\,dy\le\int_Q|f(y)|\,dy.
	\]
	It follows that
	\[
	\|g\|_\infty\le 2^n\tau,\qquad
	\|g\|_1\le\|f\|_1,\qquad
	\|g\|_2^2\le 2^n\tau\|f\|_1.
	\]
	Moreover,
	\[
	\|b_Q\|_1
	\le\int_Q(|f|+|g|)
	\le2\int_Q|f|
	\le2^{n+1}\tau|Q|,
	\]
	so that $\sum_Q\|b_Q\|_1\le2\|f\|_1$.
	Finally, by the definition of $c_Q$,
	\[
	\begin{aligned}
		\int_Q e^{i\vartheta_Q(y)}b_Q(y)\,dy
		&=\int_Q e^{i\vartheta_Q(y)}f(y)\,dy
		-c_Q|Q|\\
		&=0.
	\end{aligned}
	\]
	This proves the lemma.
\end{proof}

\begin{lemma}[Local truncation]\label{lem:localtrunc}
Let $T$ be a weak type $(1,1)$ operator with $\weak{Tf}\le A_T\norm f_1$, whose kernel satisfies
\[
 |K(x,y)|\le\frac{|\Omega(x-y)|}{|x-y|^n},\qquad \Omega\in L^1(\Sph^{n-1}).
\]
Then, for $\varepsilon>0$,
\begin{equation}\label{eq:localtrunc-bound}
 T_\varepsilon f(x)=\pv\int_{|x-y|<\varepsilon}K(x,y)f(y)\,dy,
 \qquad\weak{T_\varepsilon f}\le C(A_T+\norm\Omega_1)\norm f_1.
\end{equation}
\end{lemma}

\begin{proof}
	Partition $\R^n$ into half-open cubes $I$ of sidelength
	$\varepsilon$. Choose a fixed concentric dilate $I^*$ such that
	\[
	x\in I\Longrightarrow
	B(x,\varepsilon)\subset I^*\subset B(x,C\varepsilon),
	\qquad
	\sum_I\one_{I^*}\le C.
	\]
	Set $f_I=f\one_{I^*}$. For $x\in I$, the kernel
	representation gives
	\[
	T_\varepsilon f(x)=Tf_I(x)-E_If(x),
	\qquad
	E_If(x)=
	\int_{I^*\cap\{|x-y|\ge\varepsilon\}}
	K(x,y)f(y)\,dy.
	\]
	By the bounded overlap of $\{I^*\}$,
	\[
	\sum_I\norm{f_I}_1\le C\norm f_1.
	\]
	Hence
	\[
	\begin{aligned}
		\sum_I\bigl|\{x\in I:|Tf_I(x)|>\lambda/2\}\bigr|
		&\le \frac{2A_T}{\lambda}\sum_I\norm{f_I}_1\\
		&\le \frac{CA_T}{\lambda}\norm f_1.
	\end{aligned}
	\]
	and
	\[
	\begin{aligned}
		\sum_I\int_I|E_If(x)|\,dx
		&\le \int_{\R^n}|f(y)|
		\int_{\varepsilon\le|x-y|\le C\varepsilon}
		\frac{|\Omega(x-y)|}{|x-y|^n}\,dx\,dy\\
		&\le C\norm\Omega_1\norm f_1.
	\end{aligned}
	\]
	Thus, by Chebyshev's inequality and the decomposition above,
	\[
	\bigl|\{x:|T_\varepsilon f(x)|>\lambda\}\bigr|
	\le \frac{C(A_T+\norm\Omega_1)}{\lambda}\norm f_1.
	\]
	This proves the lemma.
\end{proof}

\begin{lemma}[Schur's test \cite{Grafakos2014}]\label{lem:schur}
Let $Sf(u)=\int K(u,v)f(v)\,dv$, and suppose that
\[
 A_K=\sup_u\int|K(u,v)|\,dv<\infty,\qquad
 B_K=\sup_v\int|K(u,v)|\,du<\infty.
\]
Then $$\norm S_{2\to2}\le(A_KB_K)^{1/2}.$$
\end{lemma}

\begin{lemma}[Carbery--Wright {\cite[Theorem~8]{CW2001}}]
	\label{lem:CW}
Let $\mu$ be a log-concave probability measure on $\R^n$, that is,
$\mu(\R^n)=1$, and, for all Borel sets $A,B\subset\R^n$
and all $\lambda\in[0,1]$,
\[
\mu\bigl((1-\lambda)A+\lambda B\bigr)
\ge
\mu(A)^{1-\lambda}\mu(B)^\lambda,
\]
where
\(
(1-\lambda)A+\lambda B
=
\{(1-\lambda)x+\lambda y:x\in A,\ y\in B\}
\). Let $q$ be a real polynomial of degree $m\ge1$ such that $\int|q|\,d\mu>0$. Then
\begin{equation}\label{eq:CW-input}
 \mu\{|q|\le a\}\le C m\left(\frac{a}{\int|q|\,d\mu}\right)^{1/m},\qquad a>0.
\end{equation}
\end{lemma}

\section{Proof of Theorem \ref{thm:main}}\label{sec:main-proof}

\subsection{Normalization and local scales}\label{subsec:local}
Following Ricci--Stein~\cite{RS1987} and
Chanillo--Christ~\cite{CC1987}, we argue by induction
on the mixed bidegree of $P$.
Write
\[
P(x,y)=A_1(x)+\widetilde P(x,y)+C_1(y),
\]
where $\widetilde P$ contains only mixed terms. Since
\[
T_Pf=e^{iA_1}T_{\widetilde P}(e^{iC_1}f),
\]
we may discard the pure terms. If $\widetilde P=0$,
the result follows from Seeger's weak type $(1,1)$
estimate~\cite[\S1, Proposition]{Seeger1996}.

We may thus write
\[
P(x,y)
=\sum_{\substack{|\alpha|\ge1\\|\beta|\ge1}}
a_{\alpha\beta}x^\alpha y^\beta.
\]
Let
\[
m=\max\{|\alpha|:a_{\alpha\beta}\ne0\},
\qquad
l=\max\{|\beta|:|\alpha|=m,\ a_{\alpha\beta}\ne0\},
\]
and induct on $(m,l)$ in lexicographic order.
Write
\begin{equation}\label{eq:topblock}
	P(x,y)=\sum_{|\alpha|=m}x^\alpha q_\alpha(y)+R(x,y),
	\qquad
	q_\alpha(y)=\sum_\beta a_{\alpha\beta}y^\beta,
	\qquad
	\deg_xR<m.
\end{equation}
Then $\max_{|\alpha|=m}\deg q_\alpha=l$.

Set
\[
a=\sum_{\substack{|\alpha|=m\\|\beta|=l}}
|a_{\alpha\beta}|>0,
\qquad B=a^{1/(m+l)},
\]
and define
\[
P_B(x,y)=P(x/B,y/B),
\qquad f_B(y)=f(y/B).
\]
By a change of variables and the homogeneity of $\Omega$,
\[
T_Pf(x/B)=T_{P_B}f_B(x),
\qquad
\norm{f_B}_1=B^n\norm f_1,
\]
and
\[
|\{x:|T_Pf(x)|>\lambda\}|
=
B^{-n}|\{x:|T_{P_B}f_B(x)|>\lambda\}|.
\]
We may therefore assume that
\begin{equation}\label{eq:normal}
	\sum_{\substack{|\alpha|=m\\|\beta|=l}}
	|a_{\alpha\beta}|=1.
\end{equation}

Using the dyadic partition of unity $\phi$ in
\eqref{eq:partition}, write
\[
T_P=T_{\le0}+T_{>0},
\qquad
T_{\le0}=\sum_{j\le0}T_j,
\qquad
T_{>0}=\sum_{j>0}T_j,
\]
where
\[
T_jf(x)=\int_{\R^n}e^{iP(x,y)}
\frac{\Omega(x-y)}{|x-y|^n}
\phi(2^{-j}|x-y|)f(y)\,dy.
\]
 We first estimate $T_{\le0}$.

For any real polynomial $Q$, let
\[
S_Qf(x)=\pv\int_{|x-y|<1}
e^{iQ(x,y)}
\frac{\Omega(x-y)}{|x-y|^n}f(y)\,dy.
\]
The support and partition properties of $\phi$ give
\[
\left|
\sum_{j\le0}\phi(2^{-j}r)-\one_{(0,1)}(r)
\right|
\le C\one_{[1/2,2]}(r).
\]

Consequently, by Fubini's theorem and polar coordinates,
\begin{equation}\label{eq:local-smooth-error}
	\begin{aligned}
		\norm{(T_{\le0}-S_P)f}_1
		&\le C\int_{\R^n}|f(y)|
		\int_{1/2\le|x-y|\le2}
		\frac{|\Omega(x-y)|}{|x-y|^n}\,dx\,dy\\
		&=C\int_{1/2}^2\frac{dr}{r}
		\norm\Omega_1\norm f_1\\
		&\le C\norm\Omega_1\norm f_1.
	\end{aligned}
\end{equation}

To estimate $S_P$, choose balls $B(h_\nu,1/4)$
covering $\R^n$ such that
\[
\R^n=\bigcup_\nu B(h_\nu,1/4),
\qquad
\sum_\nu\one_{B(h_\nu,5/4)}\le C,
\]
and set $f_\nu=f\one_{B(h_\nu,5/4)}$.
Then
\[
S_Pf=S_Pf_\nu
\quad\text{on }B(h_\nu,1/4),
\qquad
\sum_\nu\norm{f_\nu}_1\le C\norm f_1.
\]

Fix $h=h_\nu$ and write
\[
P_{m,l}(u,v)=
\sum_{\substack{|\alpha|=m\\|\beta|=l}}
a_{\alpha\beta}u^\alpha v^\beta.
\]
Define
\begin{equation}\label{eq:phase-expand-h}
	R_h(x,y)=P(x,y)-P_{m,l}(x-h,y-h).
\end{equation}
Since translation preserves the highest mixed
bidegree part, the mixed bidegree of $R_h$,
after removing pure terms, is strictly smaller
than $(m,l)$. The induction hypothesis therefore
applies to $R_h$, uniformly in $h$.

For $x\in B(h,1/4)$ and $|x-y|<1$, we have
$|y-h|<5/4$. By the mean value theorem and
\eqref{eq:normal},
\begin{equation}\label{eq:phase-difference}
	\begin{aligned}
		&|P_{m,l}(x-h,y-h)-P_{m,l}(y-h,y-h)|\\
		&\qquad\le
		C|x-y|
		\sum_{\substack{|\alpha|=m\\|\beta|=l}}
		|a_{\alpha\beta}|
		\le C|x-y|.
	\end{aligned}
\end{equation}
Thus, using $|e^{ia}-e^{ib}|\le|a-b|$,
\[
\begin{aligned}
	&\left|S_Pf_\nu(x)-
	S_{R_h}\bigl(
	e^{iP_{m,l}(\cdot-h,\cdot-h)}f_\nu
	\bigr)(x)\right|\\
	&=\left|\int_{|x-y|<1}
	(e^{iP(x,y)}
	-e^{iR_h(x,y)}e^{iP_{m,l}(y-h,y-h)})
	\frac{\Omega(x-y)}{|x-y|^n}f_\nu(y)\,dy\right|\\
	&\qquad\le
	C\int_{|x-y|<1}
	\frac{|\Omega(x-y)|}{|x-y|^{n-1}}
	|f_\nu(y)|\,dy.
\end{aligned}
\]
Since
\[
\int_{|z|<1}
\frac{|\Omega(z)|}{|z|^{n-1}}\,dz
=\norm\Omega_1,
\]
Fubini's theorem gives
\[
\begin{aligned}
	&\int_{B(h,1/4)}
	\left|S_Pf_\nu(x)-
	S_{R_h}\bigl(
	e^{iP_{m,l}(\cdot-h,\cdot-h)}f_\nu
	\bigr)(x)\right|\,dx\\
	&\qquad\le C\norm\Omega_1\norm{f_\nu}_1.
\end{aligned}
\]

By the induction hypothesis and
Lemma~\ref{lem:localtrunc},
\[
\weak{S_{R_h}u}\le CA\norm u_1.
\]
Since multiplication by
$e^{iP_{m,l}(y-h,y-h)}$ preserves the
$L^1$ norm, Chebyshev's inequality yields
\[
\begin{aligned}
	&\left|
	\left\{
	x\in B(h_\nu,1/4):
	|S_Pf_\nu(x)|>\lambda
	\right\}
	\right|
	\\
	&\quad\le
	\left|
	\left\{
	x:
	\left|
	S_{R_h}\bigl(
	e^{iP_{m,l}(\cdot-h,\cdot-h)}f_\nu
	\bigr)(x)
	\right|>\frac{\lambda}{2}
	\right\}
	\right|
	\\
	&\qquad+
	\frac{2}{\lambda}
	\int_{B(h_\nu,1/4)}
	\left|
	S_Pf_\nu(x)
	-
	S_{R_h}\bigl(
	e^{iP_{m,l}(\cdot-h,\cdot-h)}f_\nu
	\bigr)(x)
	\right|\,dx
	\\
	&\quad\le
	\frac{C(A+\norm\Omega_1)}{\lambda}\norm{f_\nu}_1
	\\
	&\quad\le
	\frac{CA}{\lambda}\norm{f_\nu}_1.
\end{aligned}
\]
Summing over $\nu$ and using the bounded
overlap of the enlarged balls, we obtain
\[
\sup_{\lambda>0}
\lambda\left|\{x:|S_Pf(x)|>\lambda\}\right|
\le CA\norm f_1.
\]
Together with \eqref{eq:local-smooth-error},
this gives
\begin{equation}\label{eq:T0weak}
	\sup_{\lambda>0}
	\lambda\left|\{x:|T_{\le0}f(x)|>\lambda\}\right|
	\le CA\norm f_1.
\end{equation}

\subsection{Bad functions and truncation by kernel size}
By \eqref{eq:T0weak}, it remains to estimate $T_{>0}$.
Fix $\lambda>0$ and set $\tau=\lambda/A$.
Applying Lemma~\ref{lem:adapted-CZ} with
$\vartheta_Q(y)=P(c_Q,y)$, we write
$f=g+b$, where $b=\sum_Qb_Q$ and
\begin{equation}\label{eq:adaptg}
	g(y)=e^{-iP(c_Q,y)}
	\frac1{|Q|}\int_Qe^{iP(c_Q,z)}f(z)\,dz
	\quad(y\in Q),\qquad b_Q=(f-g)\one_Q.
\end{equation}
Outside the bad cubes, set $g=f$. Then
\begin{equation}\label{eq:CZ}
	\begin{gathered}
		\norm g_\infty\le C\tau,\qquad
		\norm g_1\le\norm f_1,\qquad
		\norm{b_Q}_1\le C\tau|Q|,\\
		\sum_Q|Q|\le\tau^{-1}\norm f_1,\qquad
		\sum_Q\norm{b_Q}_1\le2\norm f_1.
	\end{gathered}
\end{equation}
Moreover,
\begin{equation}\label{eq:cancel}
	\int_Qe^{iP(c_Q,y)}b_Q(y)\,dy=0.
\end{equation}

By the $L^2$ boundedness theorem of
Jiang--Lu~\cite{JL1995} and Chebyshev's inequality,
\begin{equation}\label{eq:good}
	\begin{aligned}
		|\{x:|T_{>0}g(x)|>\lambda/4\}|
		&\le\frac{C}{\lambda^2}\norm{T_{>0}g}_2^2\\
		&\le\frac{CA^2}{\lambda^2}\norm g_2^2\\
		&\le\frac{CA^2\tau}{\lambda^2}\norm f_1
		=\frac{CA}{\lambda}\norm f_1.
	\end{aligned}
\end{equation}
Thus it remains to prove
\begin{equation}\label{eq:bad-reduction}
	|\{x:|T_{>0}b(x)|>\lambda\}|
	\le CA\lambda^{-1}\norm f_1.
\end{equation}

For $k\in\Z$, set
\[
B_k=\sum_{\ell(Q)=2^k}b_Q.
\]
Fix a sufficiently large $s_*>3$, and let $E$
be the union of the $C2^{s_*+10}$ dilates
of all bad cubes. By \eqref{eq:CZ},
\begin{equation}\label{eq:Esize}
	|E|\le C\sum_Q|Q|
	\le CA\lambda^{-1}\norm f_1.
\end{equation}
If $\ell(Q)=2^{j-s}$ with $s<s_*$, then
the annular support of $T_j$ implies
\[
|x-y|\le2^{j+1}
=2^{s+1}\ell(Q)
\le2^{s_*+1}\ell(Q),
\qquad y\in Q.
\]
Thus the support of $T_jb_Q$ is contained in the
chosen dilate of $Q$, and $T_jb_Q$ vanishes on $E^c$. 
Hence
\[
T_{>0}b
=\sum_{s\ge s_*}\sum_{j>0}T_jB_{j-s}
\qquad\text{on }E^c.
\]

We divide the positive scales into $j>\varrho s$
and $1\le j\le\varrho s$. Set
\begin{equation}\label{eq:parameters}
	\varrho=\frac1{800D(n+2)^3},
	\qquad N_0=4n+6,
\end{equation}
and choose
\begin{equation}\label{eq:sigma}
	0<\sigma<\frac1{800N_0(n+2)^3}.
\end{equation}
Let $\kappa_1,\kappa_2,\kappa_3,\kappa_4>0$
be the decay exponents obtained below, independently
of $\Omega$. We then choose
\begin{equation}\label{eq:eta}
	0<\eta_0<
	\frac18\min\{\kappa_1,\kappa_2,\kappa_3,\kappa_4\}.
\end{equation}

We now truncate $\Omega$ according to its size.
The case $\norm\Omega_1=0$ is trivial.
For $s\ge s_*$, set
\begin{equation}\label{eq:Omega-s}
	\Omega_s
	=\Omega\one_{\{|\Omega|\le2^{\eta_0s}\norm\Omega_1\}},
	\qquad
	\nu_s=\Omega-\Omega_s.
\end{equation}
Then
\begin{equation}\label{eq:Omega-s-infty}
	\norm{\Omega_s}_\infty
	\le2^{\eta_0s}\norm\Omega_1.
\end{equation}
Writing $T_j=T_{j,\Omega_s}+T_{j,\nu_s}$,
we first estimate the terms involving $\nu_s$.

By Fubini's theorem and polar coordinates,
\begin{equation}\label{eq:ring-L1}
	\norm{T_{j,\nu_s}h}_1
	\le C\norm{\nu_s}_1\norm h_1.
\end{equation}
Moreover, Tonelli's theorem gives
\begin{equation}\label{eq:angular-tail}
	\begin{aligned}
		\sum_{s\ge s_*}\norm{\nu_s}_1
		&=\int_{\Sph^{n-1}}|\Omega(\theta)|
		\operatorname{card}\left\{s\ge s_*:
		2^{\eta_0s}\norm\Omega_1<|\Omega(\theta)|
		\right\}d\sigma(\theta)\\
		&\le C\int_{\Sph^{n-1}}|\Omega(\theta)|
		\log^+\frac{|\Omega(\theta)|}{\norm\Omega_1}
		\,d\sigma(\theta)\\
		&\le CA.
	\end{aligned}
\end{equation}
Here the last inequality follows from the definition
of $A$ and
$\norm\Omega_1\log^+(1/\norm\Omega_1)\le C$.

Since
\[
\sum_{j>0}\norm{B_{j-s}}_1
\le\sum_Q\norm{b_Q}_1
\le2\norm f_1,
\]
we obtain
\begin{equation}\label{eq:large-angular}
	\begin{aligned}
		\sum_{s\ge s_*}\sum_{j>0}
		\norm{T_{j,\nu_s}B_{j-s}}_1
		&\le C\sum_{s\ge s_*}\norm{\nu_s}_1
		\sum_{j>0}\norm{B_{j-s}}_1\\
		&\le CA\norm f_1.
	\end{aligned}
\end{equation}

\subsection{Estimates at high and low positive scales and completion of the main theorem}

By \eqref{eq:large-angular}, it remains to estimate
the terms involving $\Omega_s$. We first establish
the following $L^2$ estimate for $j>\varrho s$.

\begin{theorem}\label{prop:high}
There exists $\kappa_1>0$, depending only on the fixed parameters, such that, for every finite set
	$\mathcal J\subset\{j>\varrho s\}$,
	\begin{equation}\label{eq:prop-high}
		\norm{\sum_{j\in\mathcal J}T_{j,\Omega_s}B_{j-s}}_2^2
		\le C\tau\norm\Omega_1^2\,2^{-(\kappa_1-2\eta_0)s}
		\sum_{j\in\mathcal J}\norm{B_{j-s}}_1.
	\end{equation}

\end{theorem}

For $1\le j\le\varrho s$, we use the following
decomposition, in which the main term is controlled
in $L^2$ and the two remainders in $L^1$.

\begin{theorem}\label{prop:lowfull}
Let $s\ge s_*$ and $1\le j\le\varrho s$. There exist functions
$G_{j,s}$, $E_{j,s}$, and $\mathcal R_{j,s}$ such that
	\begin{equation}\label{eq:prop-reduction}
		T_{j,\Omega_s}B_{j-s}=G_{j,s}+E_{j,s}+\mathcal R_{j,s}.
	\end{equation}
Moreover, for positive constants $\kappa_3,\kappa_4$ depending only on the
fixed parameters,
	\begin{align}
		\norm{G_{j,s}}_2^2
		&\le C\tau\norm\Omega_1^2\,2^{-(\kappa_3-2\eta_0)s}\norm{B_{j-s}}_1,
		\label{eq:prop-main}\\
		\norm{E_{j,s}}_1
		&\le C\norm\Omega_1\,2^{-(\kappa_4-\eta_0)s}\norm{B_{j-s}}_1,
		\label{eq:prop-rem}
	\end{align}
and
	\begin{equation}\label{eq:prop-reduction-error}
		\norm{\mathcal R_{j,s}}_1
		\le C\norm\Omega_1\,2^{-(\kappa_2-\eta_0)s}\norm{B_{j-s}}_1,
		\qquad\kappa_2=1-N_0\sigma>0.
	\end{equation}
\end{theorem}

Theorem~\ref{prop:high} is proved in
Section~\ref{sec:high}, while Theorem~\ref{prop:lowfull}
is proved in Section~\ref{sec:reduction}, using the
estimates established in Section~\ref{sec:microlocal}.

We now complete the proof of Theorem~\ref{thm:main}.
By \eqref{eq:prop-reduction}, on $E^c$ we have

\begin{align}\label{u}
	T_{>0}b
	={}&\sum_{s\ge s_*}\sum_{j>\varrho s}
	T_{j,\Omega_s}B_{j-s}\\
	&+\sum_{s\ge s_*}\sum_{1\le j\le\varrho s}
	(G_{j,s}+E_{j,s}+\mathcal R_{j,s})\notag\\
	&+\sum_{s\ge s_*}\sum_{j>0}
	T_{j,\nu_s}B_{j-s}.\notag
\end{align}

By \eqref{eq:eta} , \eqref{eq:prop-high} , and
Minkowski's inequality,
\begin{equation}\label{eq:high-final}
	\begin{aligned}
		\norm{\sum_{s\ge s_*}\sum_{j>\varrho s}
			T_{j,\Omega_s}B_{j-s}}_2
		&\qquad\le
		C\sqrt{\tau\norm f_1}\norm\Omega_1
		\sum_{s\ge s_*}2^{-(\kappa_1/2-\eta_0)s}\\
		&\qquad\le C\sqrt{\lambda A\norm f_1}.
	\end{aligned}
\end{equation}
For the low positive scales, Cauchy--Schwarz
and \eqref{eq:prop-main} give
\begin{equation}\label{eq:low-main-total}
	\begin{aligned}
		\norm{\sum_{s\ge s_*}
			\sum_{1\le j\le\varrho s}G_{j,s}}_2
		&\qquad\le
		C\sqrt{\tau\norm f_1}\norm\Omega_1
		\sum_{s\ge s_*}\sqrt{1+s}\,
		2^{-(\kappa_3/2-\eta_0)s}\\
		&\qquad\le C\sqrt{\lambda A\norm f_1}.
	\end{aligned}
\end{equation}
Here we used $\tau=\lambda/A$,
$\norm\Omega_1\le CA$, and
$\sum_j\norm{B_{j-s}}_1\le2\norm f_1$.
Similarly, \eqref{eq:prop-rem} and
\eqref{eq:prop-reduction-error} yield
\begin{equation}\label{eq:R-total}
	\begin{aligned}
		&\sum_{s\ge s_*}\sum_{1\le j\le\varrho s}
		\bigl(\norm{E_{j,s}}_1+
		\norm{\mathcal R_{j,s}}_1\bigr)\\
		&\qquad\le C\norm\Omega_1\norm f_1
		\sum_{s\ge s_*}
		\left(
		2^{-(\kappa_4-\eta_0)s}
		+2^{-(\kappa_2-\eta_0)s}
		\right)\\
		&\qquad\le CA\norm f_1.
	\end{aligned}
\end{equation}
Hence, by \eqref{eq:large-angular},\eqref{u},
\eqref{eq:high-final}--\eqref{eq:R-total},
and Chebyshev's inequality,
\[
|\{x\in E^c:|T_{>0}b(x)|>\lambda/2\}|
\le CA\lambda^{-1}\norm f_1.
\]
Together with \eqref{eq:T0weak},
\eqref{eq:good}, and \eqref{eq:Esize},
this gives
\begin{equation}\label{eq:target-final}
	\lambda|\{x:|T_Pf(x)|>\lambda\}|
	\le CA\norm f_1.
\end{equation}
This completes the proof of Theorem~\ref{thm:main}.

\section{Proof of Theorem \ref{prop:high}}\label{sec:high}

\subsection{Polynomial sublevel sets and entire cubes}
\label{subsec:high-sublevel}

The following lemma gives a polynomial sublevel set
estimate and extends it to the union of small cubes
intersecting the sublevel set. The latter is needed
since the mass bounds for bad functions are available
only on entire cubes. We also record the comparability
of the non-constant coefficient seminorm under bounded
translations.

\begin{lemma}\label{lem:sublevel}
	Let $U\subset\R^n$ be a fixed bounded box with nonempty
	interior, let $q\ne0$ be a real polynomial of degree
	at most $D\ge1$, and set $B=\norm q_\coef$.
	Then, for $0<t<1$,
	\begin{equation}\label{eq:sublevel}
		\abs{\{u\in U:|q(u)|\le tB\}}\le Ct^{1/D}.
	\end{equation}
	If a family of axis-parallel cubes has sidelengths
	at most $\ell\le1$, each cube intersects
	$\{u\in U:|q(u)|\le tB_0\}$, and $0<B_0\le CB$,
	then
	\begin{equation}\label{eq:thick}
		\abs{\bigcup_Q Q}\le C(t+\ell)^{1/D}.
	\end{equation}
	Moreover, for any polynomial $q$ of degree at most
	$D$ and any $|z|\le C$,
	\begin{equation}\label{eq:translation}
		\norm{q(z+\cdot)-q(z)}_\nc
		\asymp\norm q_\nc.
	\end{equation}
\end{lemma}

\begin{proof}
	By equivalence of norms on the finite-dimensional
	space of polynomials of degree at most $D$,
	\[
	\int_U|q(u)|\,du\asymp B.
	\]
	The normalized Lebesgue measure $\mu_U=|U|^{-1}\one_Udu$ on $U$ is a log-concave probability measure.
	If $d_q=\deg q\ge1$, Lemma~\ref{lem:CW}, with
	$\alpha=tB$, gives
	\[
	\begin{aligned}
		|\{u\in U:|q(u)|\le tB\}|
		&=|U|\mu_U\{|q|\le tB\}\\
		&\le Cd_q\left(
		\frac{tB}{\int|q|\,d\mu_U}
		\right)^{1/d_q}\\
		&\le Ct^{1/D},
	\end{aligned}
	\]
	where we used $d_q\le D$ and $0<t<1$.
	For a nonzero constant polynomial, the sublevel set
	is empty when $t<1$. This proves \eqref{eq:sublevel}.
	
	To prove \eqref{eq:thick}, let
	\[
	U^*=\{u:\dist_\infty(u,U)\le2\}.
	\]
	Every cube $Q$ in the statement is contained in
	$U^*$. Since $U^*$ is fixed,
	\[
	\norm{\nabla q}_{L^\infty(U^*)}\le CB.
	\]
	Choose
	$u_Q\in Q\cap\{u\in U:|q(u)|\le tB_0\}$.
	For every $u\in Q$,
	\[
	\begin{aligned}
		|q(u)|
		&\le |q(u_Q)|
		+|u-u_Q|\norm{\nabla q}_{L^\infty(U^*)}\\
		&\le C(t+\ell)B.
	\end{aligned}
	\]
	Thus
	\[
	\bigcup_QQ
	\subset U^*\cap\{u:|q(u)|\le C(t+\ell)B\}.
	\]
	Applying \eqref{eq:sublevel} on the fixed box $U^*$
	when $C(t+\ell)<1$, and using the volume of $U^*$
	otherwise, gives \eqref{eq:thick}.
	
	For \eqref{eq:translation}, the binomial formula
	\[
	(x+z)^\alpha=\sum_{\beta\le\alpha}\binom\alpha\beta z^{\alpha-\beta}x^\beta
	\]
	shows that $\tau_zq=q(z+x)-q(z)$ satisfies
	$\norm{\tau_zq}_\nc\le C\norm q_\nc$. Applying the identity
	$\tau_{-z}\tau_zq=q-q(0)$ gives the reverse inequality.
\end{proof}

To control degenerate directions in the subsequent analysis of two-dimensional sections, we extend the sublevel set estimate of Lemma \ref{lem:sublevel} to the sphere.

\begin{lemma}[Sublevel sets on the sphere]\label{lem:sphere}
Let $F$ be a nonzero restriction to $\Sph^{n-1}$ of a real polynomial of degree at most $D$, where $n\ge2$. Then
\begin{equation}\label{eq:sphere}
 \sigma\{\theta:|F(\theta)|\le t\norm F_\infty\}
 \le Ct^{1/(2D)},\qquad 0<t<1.
\end{equation}
\end{lemma}

\begin{proof}
	Let $U=[-1,1]^{n-1}$ and consider the two
	stereographic parametrizations
	\[
	\Psi_\pm(u)=\pm\frac{(2u,1-|u|^2)}{1+|u|^2},
	\qquad u\in U.
	\]
	Since $\Psi_+$ and $\Psi_-$ map the unit ball in
	$\R^{n-1}$ onto the closed upper and lower
	hemispheres, respectively,
	\[
	\Sph^{n-1}=\Psi_+(U)\cup\Psi_-(U).
	\]

Direct differentiation gives
\[
\left(
\partial_i\Psi_\pm(u)\cdot\partial_j\Psi_\pm(u)
\right)_{1\le i,j\le n-1}
=\frac{4}{(1+|u|^2)^2}I_{n-1}.
\]
Hence the surface Jacobian is
\begin{equation}\label{eq:sphere-density}
	J(u)=
	\sqrt{\det\left(
		\left(\partial_i\Psi_\pm(u)\cdot
		\partial_j\Psi_\pm(u)\right)_{i,j=1}^{n-1}
		\right)}
	=\left(\frac{2}{1+|u|^2}\right)^{n-1}
	\le 2^{n-1}.
\end{equation}

	Write
	\[
	F(\theta)=\sum_{|\alpha|\le D}a_\alpha\theta^\alpha,
	\qquad \alpha=(\alpha',\alpha_n),
	\]
	and define
	\begin{equation}\label{eq:sphere-numerator}
		G_\pm(u)=(1+|u|^2)^DF(\Psi_\pm(u)).
	\end{equation}
	For each monomial in $F$, the corresponding term
	in $G_\pm$ is
	\[
	a_\alpha(\pm1)^{|\alpha|}2^{|\alpha'|}
	u^{\alpha'}(1-|u|^2)^{\alpha_n}
	(1+|u|^2)^{D-|\alpha|},
	\]
	whose degree is at most
	\[
	|\alpha'|+2\alpha_n+2(D-|\alpha|)
	=2D-|\alpha'|\le2D.
	\]
	Thus $G_\pm$ is a polynomial of degree at most $2D$.
	
	The map $F\mapsto G_\pm$ is linear and injective:
	if $G_\pm$ vanishes on $U$, it vanishes identically
	as a polynomial, and hence $F$ vanishes on
	$\Sph^{n-1}$. By equivalence of norms on the
	finite-dimensional space of polynomial restrictions
	of degree at most $D$,
	\begin{equation}\label{eq:sphere-norms}
		\norm{G_\pm}_\coef\asymp\norm F_\infty.
	\end{equation}
	
	Set
	\[
	E_\pm(t)=
	\{u\in U:|F(\Psi_\pm(u))|
	\le t\norm F_\infty\}.
	\]
	Since $U$ is bounded, \eqref{eq:sphere-numerator}
	and \eqref{eq:sphere-norms} imply
	\[
	E_\pm(t)\subset
	\{u\in U:|G_\pm(u)|
	\le Ct\norm{G_\pm}_\coef\}.
	\]
	Applying Lemma~\ref{lem:sublevel} with degree
	bound $2D$ when $Ct<1$, and using the trivial
	bound $|E_\pm(t)|\le |U|$ otherwise, gives
	\[
	|E_\pm(t)|\le Ct^{1/(2D)}.
	\]
	Finally, by \eqref{eq:sphere-density},
	\[
	\begin{aligned}
		\sigma\{\theta:|F(\theta)|
		\le t\norm F_\infty\}
		&\le\sum_{\pm}\int_{E_\pm(t)}J(u)\,du\\
		&\le C\sum_{\pm}|E_\pm(t)|\\
		&\le Ct^{1/(2D)}.
	\end{aligned}
	\]
\end{proof}

We next apply the spherical sublevel set estimate to the coefficients of a homogeneous polynomial on two-dimensional sections. This bounds the measure of  directions for which the section coefficients are too small relative to those of the original polynomial.

\begin{lemma}\label{lem:direction}
Fix $e\in\Sph^{n-1}$ and set $\Sigma_e=\{\theta\in e^\perp:|\theta|=1\}$. Equip $\Sigma_e$ with surface measure $\sigma_e$ when $n\ge3$, and with counting measure on its two points when $n=2$. For a nonzero homogeneous polynomial $Q_m$ of degree $m$, where $1\le m\le D$, write
\begin{equation}\label{eq:direction-coeff}
 Q_m(pe+t\theta)=\sum_{a=0}^mp^at^{m-a}C_a^{e,Q_m}(\theta).
\end{equation}
Then
\begin{gather}
 \sup_{\theta\in\Sigma_e}\sum_{a=0}^m|C_a^{e,Q_m}(\theta)|
 \ge c\norm{Q_m}_\coef,\label{eq:direction-sup}\\
 \sigma_e\left\{\theta:\sum_{a=0}^m|C_a^{e,Q_m}(\theta)|
 \le t\norm{Q_m}_\coef\right\}
 \le Ct^\alpha,
 \qquad\alpha=\frac1{2D},\quad0<t<1.\label{eq:direction-bad}
\end{gather}
\end{lemma}
\begin{proof}
	For $|x|\le1$, write $x=pe+t\theta$ with
	$|p|,t\le1$. By equivalence of norms,
	\[
	\begin{aligned}
		c\norm{Q_m}_\coef
		&\le\sup_{|x|\le1}|Q_m(x)|\\
		&\le\sup_{\theta\in\Sigma_e}
		\sum_{a=0}^m|C_a^{e,Q_m}(\theta)|\\
		&\le(m+1)\max_a\norm{C_a^{e,Q_m}}_\infty.
	\end{aligned}
	\]
	Choose $a_*$ attaining the last maximum. Then
	\[
	\norm{C_{a_*}^{e,Q_m}}_\infty
	\ge c\norm{Q_m}_\coef,
	\]
	and hence
	\[
	\begin{aligned}
		&\left\{\theta:
		\sum_{a=0}^m|C_a^{e,Q_m}(\theta)|
		\le t\norm{Q_m}_\coef\right\}\\
		&\qquad\subset
		\left\{\theta:
		|C_{a_*}^{e,Q_m}(\theta)|
		\le Ct\norm{C_{a_*}^{e,Q_m}}_\infty\right\}.
	\end{aligned}
	\]
	If $n\ge3$, then $C_{a_*}^{e,Q_m}$ is the restriction
	of a polynomial of degree at most $D$ to
	$\Sigma_e\cong\Sph^{n-2}$. Lemma~\ref{lem:sphere}
	therefore gives \eqref{eq:direction-bad} when
	$Ct<1$; the remaining case follows from
	$\sigma_e(\Sigma_e)\le C$.
	
	If $n=2$, then $\Sigma_e=\{\theta,-\theta\}$ and
	\[
	C_a^{e,Q_m}(-\theta)
	=(-1)^{m-a}C_a^{e,Q_m}(\theta).
	\]
	Thus the coefficient sum has the same value at
	both points and is bounded below by
	$c\norm{Q_m}_\coef$. The sublevel set is therefore
	empty for $t<c$, while for $t\ge c$ its measure
	is at most $2\le Ct^\alpha$.
\end{proof}

\subsection{Two-center coordinates and mixed derivatives}
\label{subsec:two-center-coordinates}

Fix $y\ne z$ and set
\[
d=|y-z|,\qquad e=(y-z)/d.
\]
By polar coordinates in $e^\perp$, every $x$ outside
the line through $y$ and $z$ has a unique representation
\begin{equation}\label{eq:cyl}
	x=z+r(pe+t\theta),\qquad
	p\in\R,\quad t>0,\quad\theta\in\Sigma_e,
	\qquad
	dx=r^nt^{n-2}\,dp\,dt\,d\sigma_e(\theta).
\end{equation}
Set $H=d/r$, $L=\max(1,H)$, and define
\begin{equation}\label{eq:uv}
	u=\frac pt,\qquad v=\frac{p-H}{Lt}.
\end{equation}
The inverse transformation is
\begin{equation}\label{eq:inverse}
	t=\frac H{u-Lv},\qquad
	p=\frac{Hu}{u-Lv},\qquad
	\abs{\frac{\partial(p,t)}{\partial(u,v)}}
	=\frac{Lt^3}{H},
\end{equation}
which maps $t>0$ onto $u-Lv>0$.
Since
\[
x-z=rt(ue+\theta),\qquad
x-y=rt(Lve+\theta),
\]
we obtain
\begin{equation}\label{eq:directions}
	\frac{x-z}{|x-z|}
	=\frac{ue+\theta}{\sqrt{1+u^2}},
	\qquad
	\frac{x-y}{|x-y|}
	=\frac{Lve+\theta}{\sqrt{1+L^2v^2}}.
\end{equation}
Thus the two rough factors depend on $u$ and $v$
separately.

For a real polynomial $Q$, write
\[
q_\theta(p,t)=Q(z+r(pe+t\theta)),
\]
and define
\[
E_0=p\partial_p+t\partial_t,\qquad
E_H=(p-H)\partial_p+t\partial_t,\qquad
S_\theta=(E_H+1)E_0q_\theta.
\]

\begin{lemma}[Mixed derivative identity]\label{lem:mixed}
	Let $\Phi_\theta(u,v)=q_\theta(p(u,v),t(u,v))$.
	Then
	\begin{equation}\label{eq:mixed}
		\partial_u\partial_v\Phi_\theta(u,v)
		=-\frac{Lt^2}{H^2}S_\theta(p,t).
	\end{equation}
	If $q_k$ is the homogeneous component of degree
	$k$ of $q_\theta$, then
	\begin{equation}\label{eq:TH}
		(E_H+1)E_0q_k
		=k(k+1)q_k-kH\partial_pq_k.
	\end{equation}
	In particular, the highest-degree component
	does not cancel. Moreover, when $0\le H\le C$,
	\begin{equation}\label{eq:triangular}
		\norm{q_\theta}_\nc
		\le C\norm{S_\theta}_\coef.
	\end{equation}
	At $H=0$, \eqref{eq:triangular} is understood
	algebraically, without using the coordinate
	transformation.
\end{lemma}

\begin{proof}
	Differentiating \eqref{eq:inverse} gives
	\[
	\frac{\partial(p,t)}{\partial(u,v)}
	=\begin{pmatrix}
		t-pt/H & Lpt/H\\ -t^2/H & Lt^2/H
	\end{pmatrix},\qquad
	\det\frac{\partial(p,t)}{\partial(u,v)}=\frac{Lt^3}{H}.
	\]
	By the chain rule,
	\[
	\partial_v=\frac{Lt}{H}E_0,
	\qquad
	\partial_u=-\frac tH E_H.
	\]
	Since $E_Ht=t$,
	\[
	\begin{aligned}
		\partial_u\partial_v\Phi_\theta
		&=-\frac{Lt}{H^2}E_H(tE_0q_\theta)\\
		&=-\frac{Lt^2}{H^2}(E_H+1)E_0q_\theta,
	\end{aligned}
	\]
	which proves \eqref{eq:mixed}. The identities
	$E_0q_k=kq_k$ and $E_H=E_0-H\partial_p$
	give \eqref{eq:TH}.

Write $q_\theta=\sum_{k=0}^Dq_k$ and
$S_\theta=\sum_{k=0}^DS_k$, with $q_{D+1}=0$.
Comparing homogeneous components yields
\[
S_k=k(k+1)q_k-(k+1)H\partial_pq_{k+1},
\qquad 1\le k\le D.
\]
Since
\[
\norm{\partial_pq_{k+1}}_\coef
\le(k+1)\norm{q_{k+1}}_\coef
\]
and $0\le H\le C$, we obtain
\[
\norm{q_k}_\coef
\le C\left(
\norm{S_k}_\coef+\norm{q_{k+1}}_\coef
\right).
\]
Descending induction on $k$ gives
\[
\norm{q_k}_\coef
\le C\sum_{j=k}^D\norm{S_j}_\coef,
\qquad 1\le k\le D.
\]
Consequently,
\[
\norm{q_\theta}_\nc
=\sum_{k=1}^D\norm{q_k}_\coef
\le C\sum_{k=1}^D\norm{S_k}_\coef
\le C\norm{S_\theta}_\coef,
\]
which proves \eqref{eq:triangular}.

\end{proof}

\subsection{Oscillatory operator estimates on rectangles}\label{subsec:rectangle-estimate}
\begin{lemma}[Mixed derivative estimate on a rectangle]\label{lem:rect}
Let $U,V$ be bounded intervals of length at most $1$, let $\Phi\in C^3(U\times V)$ be real-valued, and let $a\in C_c^1(U\times V)$. Suppose that $\Phi_{uv}$ has constant sign and
\[
 |\Phi_{uv}|\ge\lambda\ge2,\qquad |\Phi_{uvv}|\le M_\Phi\lambda.
\]
Set  $$\mathcal V f(u)=\int_Ve^{i\Phi(u,v)}a(u,v)f(v)\,dv$$ and  $M_a=\norm a_\infty+\norm{\partial_va}_\infty.$ Then
\begin{equation}\label{eq:rect-bound}
 \norm{\mathcal V}_{2\to2}\le CM_a(1+M_\Phi)^{1/2}\lambda^{-1/4}.
\end{equation}

\end{lemma}

\begin{proof}
	The adjoint and the correlation kernel are given by
	\begin{align*}
		\mathcal V^*g(v)=\int_Ue^{-i\Phi(u',v)}\overline{a(u',v)}g(u')\,du',\end{align*} and 
	\begin{align*} \mathcal V\mathcal V^*g(u)=\int_UK_{\mathcal V}(u,u')g(u')\,du',\end{align*}
	where  $$K_{\mathcal V}(u,u')=\int_Ve^{i\Psi(v)}b(v)\,dv$$
	and \begin{align*}\Psi(v)&=\Phi(u,v)-\Phi(u',v),\qquad
		b(v)=a(u,v)\overline{a(u',v)}.
	\end{align*}
	For $h=|u-u'|>0$, the constant sign of
	$\Phi_{uv}$ gives
	\[
	\begin{aligned}
		|\Psi'(v)|
		&=\left|\int_{u'}^u\Phi_{uv}(w,v)\,dw\right|
		\ge\lambda h,\\
		|\Psi''(v)|
		&\le\int_{\min(u,u')}^{\max(u,u')}
		|\Phi_{uvv}(w,v)|\,dw
		\le M_\Phi\lambda h.
	\end{aligned}
	\]
	Since $b$ is compactly supported in $V$,
	integration by parts yields
	\[
	K_{\mathcal V}(u,u')
	=-\int_V e^{i\Psi(v)}
	\left(
	\frac{b'(v)}{i\Psi'(v)}
	-\frac{b(v)\Psi''(v)}{i(\Psi'(v))^2}
	\right)dv.
	\]
	Using
	\[
	\norm b_\infty\le M_a^2,
	\qquad
	\norm{b'}_\infty\le2M_a^2,
	\]
	and $|V|\le1$, we obtain
	\[
	|K_{\mathcal V}(u,u')|
	\le\frac{CM_a^2(1+M_\Phi)}
	{\lambda|u-u'|}.
	\]
	Together with the trivial bound
	$|K_{\mathcal V}(u,u')|\le M_a^2$, this gives
	\begin{equation}\label{eq:rect-kernel}
		|K_{\mathcal V}(u,u')|
		\le CM_a^2(1+M_\Phi)
		\min\left\{1,\frac1{\lambda|u-u'|}\right\}.
	\end{equation}
	
	Since $|U|\le1$ and $\lambda\ge2$,
	\begin{align}
		\sup_{u\in U}\int_U\min\left\{1,\frac1{\lambda|u-u'|}\right\}\,du'
		&\le\int_{|u-u'|\le\lambda^{-1}}1\,du'
		+\frac1\lambda\int_{\lambda^{-1}<|u-u'|\le1}\frac{du'}{|u-u'|}\notag\\
		&\le\frac2\lambda+\frac2\lambda\int_{\lambda^{-1}}^1\frac{dt}{t}
		=\frac{2(1+\log\lambda)}\lambda.\label{eq:rect-row-integral}
	\end{align}
	The same bound holds for the column integral.
	Lemma \ref{lem:schur} therefore yields
	\begin{align}
		\norm{\mathcal V}_{2\to2}^2
		&=\norm{\mathcal V\mathcal V^*}_{2\to2}\notag\\
		&\le\left(\sup_u\int_U|K_{\mathcal V}(u,u')|\,du'\right)^{1/2}
		\left(\sup_{u'}\int_U|K_{\mathcal V}(u,u')|\,du\right)^{1/2}\notag\\
		&\le CM_a^2(1+M_\Phi)\frac{1+\log\lambda}{\lambda}.\label{eq:rect-schur}
	\end{align}

	Since $1+\log\lambda\le C\lambda^{1/2}$
	for $\lambda\ge2$, we conclude that
	\[
	\norm{\mathcal V}_{2\to2}
	\le CM_a(1+M_\Phi)^{1/2}\lambda^{-1/4}.
	\]
\end{proof}

\subsection{Two-center rough oscillatory kernel estimates}\label{subsec:two-center-estimate}
\begin{lemma}[Two-center rough oscillatory kernel estimates]\label{lem:twocenter}
Let $R\ge r>0$, $y\ne z$, and $d=|y-z|$. Suppose that $a\in C_c^2(\R^n)$ and
\[
 \supp a\subset
 \{x:c_*r\le|x-z|\le Cr,\ c_*R\le|x-y|\le CR\}.
\]
In the coordinates of \eqref{eq:cyl}, set
\[
 a_\theta(p,t):=a(z+r(pe+t\theta)).
\]
Assume that
\[
 \supp a_\theta\subset\{|p|\le C,\ 0<t\le C\},\qquad
 \sup_\theta\max_{i+j\le2}\norm{\partial_p^i\partial_t^ja_\theta}_\infty\le C.
\]
Fix the scale separation $s$, use the truncated kernel $\Omega_s$ with the bound $|\Omega_s|\le2^{\eta_0s}\norm\Omega_1$, and write
\begin{equation}\label{eq:two-integral}
 I=R^{-n}r^{-n}\int_{\R^n}e^{iQ(x)}a(x)
 \Omega_s\!\left(\frac{x-y}{|x-y|}\right)
 \overline{\Omega_s\!\left(\frac{x-z}{|x-z|}\right)}dx.
\end{equation}
If $Q$ is a nonconstant real polynomial of degree at most $D$, with highest homogeneous component $Q_m$ and $A_m=\norm{Q_m}_\coef>0$, then
\begin{equation}\label{eq:two-high}
 |I|\le C2^{2\eta_0s}\norm\Omega_1^2 R^{-n}\min\left\{1,
 \left(1+\frac rd\right)^N
 \left(\frac d{r^{m+1}A_m}\right)^c\right\}.
\end{equation}
If $R=r$, $d\le CR$, and $A_Q=\norm{Q(z+R\cdot)}_\nc>0$, then
\begin{equation}\label{eq:two-equal}
 |I|\le C2^{2\eta_0s}\norm\Omega_1^2 R^{-n}\min\{1,(1+R/d)^N A_Q^{-c}\}.
\end{equation}
One may take $c=1/(258D)$ and $N=16$. In the equal-scale estimate, the constant may also depend on the fixed upper bound for $d/R$. When $Q$ is constant or $d=0$, we use only the direct bound $C2^{2\eta_0s}\norm\Omega_1^2 R^{-n}$.
\end{lemma}
\begin{proof}
	Using \eqref{eq:cyl}--\eqref{eq:directions}, set
	\[
	F_\theta(v)=\Omega_s\left(
	\frac{Lve+\theta}{\sqrt{1+L^2v^2}}\right),\qquad
	G_\theta(u)=\overline{\Omega_s\left(
		\frac{ue+\theta}{\sqrt{1+u^2}}\right)}.
	\]
	Then $|F_\theta|,|G_\theta|\le2^{\eta_0s}\norm\Omega_1$ and
	\begin{equation}\label{eq:section-integral}
		\begin{aligned}
			I&=\int_{\Sigma_e}I_\theta\,d\sigma_e(\theta),\\
			I_\theta&=R^{-n}\int_{-C}^C\int_0^C
			e^{iq_\theta(p,t)}a_\theta(p,t)t^{n-2}
			F_\theta(v)G_\theta(u)\,dt\,dp.
		\end{aligned}
	\end{equation}
	In particular,
	\begin{equation}\label{eq:section-direct}
		|I_\theta|\le C2^{2\eta_0s}\norm\Omega_1^2R^{-n}.
	\end{equation}
	Write
	\[
	K_H=1+H^{-1},\qquad
	B_\theta=\norm{S_\theta}_\coef,\qquad
	\Lambda_\theta=\frac L{H^2}B_\theta.
	\]
	Fix $\theta$ with $B_\theta>0$; otherwise we use
	\eqref{eq:section-direct}.
	
	By \eqref{eq:mixed}, the mixed derivative can be small
	when $t$ or $|S_\theta|/B_\theta$ is small. Choose an even
	function $\chi\in C^\infty(\R)$ with $0\le\chi\le1$,
	$\chi=0$ on $[-1,1]$, and $\chi=1$ outside $(-2,2)$.
	For $0<\varepsilon,\eta\le1$, set
	\[
	\chi_t=\chi(t/\varepsilon),\qquad
	\chi_S=\chi(S_\theta/(\eta B_\theta)).
	\]
	Split $I_\theta$ by inserting
	$1=(1-\chi_t\chi_S)+\chi_t\chi_S$.
	Since
	\[
	0\le1-\chi_t\chi_S
	\le\one_{\{t<2\varepsilon\}}
	+\one_{\{|S_\theta|<2\eta B_\theta\}},
	\]
	Lemma~\ref{lem:sublevel} gives, on a fixed box $U_*$,
	\begin{equation}\label{eq:bad-section}
		\begin{aligned}
			|I_{\theta,\mathrm{bad}}|
			&\le C2^{2\eta_0s}\norm\Omega_1^2R^{-n}
			\left[\int_{-C}^C\int_0^{2\varepsilon}t^{n-2}\,dt\,dp
			+\abs{\{|S_\theta|<2\eta B_\theta\}\cap U_*}\right]\\
			&\le C2^{2\eta_0s}\norm\Omega_1^2R^{-n}
			(\varepsilon^{n-1}+\eta^{1/D}).
		\end{aligned}
	\end{equation}
	When $2\eta\ge1$, the second term is estimated by $|U_*|$.
	
	On the support of the remaining amplitude,
	\[
	t\ge\varepsilon,\qquad |S_\theta|\ge\eta B_\theta,
	\qquad |u|+|v|\le C\varepsilon^{-1}.
	\]
	To apply Lemma~\ref{lem:rect}, these lower bounds must
	hold on entire rectangles. Take a square grid of sidelength
	\begin{equation}\label{eq:grid}
		\Delta=c_1\varepsilon\eta K_H^{-1},
	\end{equation}
	with $c_1>0$ sufficiently small, and a smooth partition
	of unity satisfying
	\begin{equation}\label{eq:partition-good}
		\sum_\nu\psi_\nu=1,\qquad
		0\le\psi_\nu\le1,\qquad
		\psi_\nu\in C_c^\infty(2Q_\nu),\qquad
		\norm{\nabla\psi_\nu}_\infty\le C\Delta^{-1}.
	\end{equation}
	Retain only those $2Q_\nu$ meeting the support of
	$a_\theta\chi_t\chi_S$ in the $(u,v)$ coordinates. Then
	\begin{equation}\label{eq:rect-count}
		\operatorname{card}\{\nu\}
		\le C\varepsilon^{-2}\Delta^{-2}
		\le CK_H^2\varepsilon^{-4}\eta^{-2}.
	\end{equation}
	
	Fix a retained square and choose a point $(u_0,v_0)$
	in this intersection, with corresponding coordinates
	$(p_0,t_0)$. For $(u,v)\in4Q_\nu$,
	\[
	\left|\frac{u-Lv}{H}-\frac1{t_0}\right|
	\le CK_H\Delta\le\frac1{2t_0},
	\qquad \frac23t_0\le t\le2t_0.
	\]
	Moreover,
	\[
	t-t_0=-tt_0\left(\frac1t-\frac1{t_0}\right),\qquad
	p-p_0=t(u-u_0)-tp_0\left(\frac1t-\frac1{t_0}\right),
	\]
	so $|p-p_0|+|t-t_0|\le CK_H\Delta$.
	Since $\norm{\nabla S_\theta}_\infty\le CB_\theta$
	on a fixed enlarged box,
	\[
	|S_\theta(p,t)-S_\theta(p_0,t_0)|
	\le CB_\theta K_H\Delta\le\eta B_\theta/2.
	\]
	Thus, on the whole of $4Q_\nu$,
	\begin{equation}\label{eq:expanded-good}
		t\ge\varepsilon/2,\qquad |S_\theta|\ge\eta B_\theta/2,
	\end{equation}
	and $S_\theta$ has the same sign as at $(p_0,t_0)$.
	
	Write the interior of $4Q_\nu$ as $U_\nu\times V_\nu$.
	Its sidelengths are at most $1$ after decreasing $c_1$.
	Since $t^{n-2}\,dp\,dt=(L/H)t^{n+1}\,du\,dv$, define
	\[
	A_{\theta,\nu}
	=\frac LHt^{n+1}a_\theta(p,t)\chi_t\chi_S\psi_\nu,
	\qquad
	\mathcal V_{\theta,\nu}f(u)
	=\int_{V_\nu}e^{i\Phi_\theta(u,v)}
	A_{\theta,\nu}(u,v)f(v)\,dv.
	\]
	The amplitude belongs to $C_c^1(U_\nu\times V_\nu)$.
	On this rectangle, $|p|+t\le C$ and
	\[
	p_v=\frac LHpt,\qquad t_v=\frac LHt^2,\qquad
	|p_v|+|t_v|\le CK_H.
	\]
	Consequently,
	\[
	|\partial_v\chi_t|\le CK_H\varepsilon^{-1},\qquad
	|\partial_v\chi_S|\le CK_H\eta^{-1},\qquad
	|\partial_v\psi_\nu|\le C\Delta^{-1}.
	\]
	The product rule and the bounds for $a_\theta$ give
	\begin{equation}\label{eq:amp-good}
		\begin{aligned}
			\norm{A_{\theta,\nu}}_\infty
			+\norm{\partial_vA_{\theta,\nu}}_\infty
			&\le C\bigl[K_H^2(1+\varepsilon^{-1}+\eta^{-1})
			+K_H\Delta^{-1}\bigr]\\
			&\le CK_H^2\varepsilon^{-2}\eta^{-2}.
		\end{aligned}
	\end{equation}
	By \eqref{eq:mixed}, \eqref{eq:expanded-good}, and
	$\partial_v=(Lt/H)E_0$,
	\begin{equation}\label{eq:mixed-good}
		\begin{aligned}
			|\Phi_{\theta,uv}|&\ge c\varepsilon^2\eta\Lambda_\theta,\\
			\Phi_{\theta,uvv}
			&=-\frac{L^2t^3}{H^3}(E_0+2)S_\theta,\qquad
			|\Phi_{\theta,uvv}|\le CK_H\Lambda_\theta.
		\end{aligned}
	\end{equation}
	The mixed derivative has constant sign on the rectangle.
	For $\lambda=c\varepsilon^2\eta\Lambda_\theta\ge2$,
	Lemma~\ref{lem:rect} applies with
	$M_\Phi\le CK_H\varepsilon^{-2}\eta^{-1}$, and yields
	\begin{equation}\label{eq:operator-good}
		\begin{aligned}
			\norm{\mathcal V_{\theta,\nu}}_{2\to2}
			&\le CK_H^{5/2}\varepsilon^{-7/2}\eta^{-11/4}
			\Lambda_\theta^{-1/4}\\
			&\le CK_H^4\varepsilon^{-4}\eta^{-4}
			\Lambda_\theta^{-1/4}.
		\end{aligned}
	\end{equation}
	The rough factors are estimated in $L^2$, without
	being differentiated:
	\[
	\begin{aligned}
		&R^{-n}\left|\int_{U_\nu}G_\theta(u)
		\mathcal V_{\theta,\nu}(F_\theta\one_{V_\nu})(u)\,du\right|\\
		&\quad\le R^{-n}\norm{G_\theta}_{L^2(U_\nu)}
		\norm{\mathcal V_{\theta,\nu}}_{2\to2}
		\norm{F_\theta}_{L^2(V_\nu)}\\
		&\quad\le C2^{2\eta_0s}\norm\Omega_1^2R^{-n}
		K_H^4\varepsilon^{-4}\eta^{-4}\Lambda_\theta^{-1/4}.
	\end{aligned}
	\]
	Summing over the retained squares by \eqref{eq:rect-count},
	and adding \eqref{eq:bad-section}, gives
	\begin{equation}\label{eq:pre-opt}
		|I_\theta|\le C2^{2\eta_0s}\norm\Omega_1^2R^{-n}
		\left[\varepsilon^{n-1}+\eta^{1/D}
		+K_H^8\varepsilon^{-8}\eta^{-8}\Lambda_\theta^{-1/4}\right].
	\end{equation}
	For $\Lambda_\theta$ sufficiently large, take
	$\varepsilon=\eta=\Lambda_\theta^{-1/256}$.
	Then $\lambda=c\Lambda_\theta^{253/256}\ge2$ and
	\[
	\varepsilon^{n-1}=\Lambda_\theta^{-(n-1)/256},\qquad
	\eta^{1/D}=\Lambda_\theta^{-1/(256D)},\qquad
	\varepsilon^{-8}\eta^{-8}\Lambda_\theta^{-1/4}
	=\Lambda_\theta^{-3/16}.
	\]
	Together with \eqref{eq:section-direct}, this proves
	\begin{equation}\label{eq:section-decay}
		|I_\theta|\le C2^{2\eta_0s}\norm\Omega_1^2R^{-n}
		\min\{1,K_H^8\Lambda_\theta^{-\delta}\},
		\qquad \delta=\frac1{256D}.
	\end{equation}
	
	It remains to integrate over $\theta$. The highest
	homogeneous component of $S_\theta$ is
	\[
	(S_\theta)_m(p,t)
	=m(m+1)r^m\sum_{a=0}^m p^at^{m-a}C_a^{e,Q_m}(\theta).
	\]
	Hence
	\[
	B_\theta\ge m(m+1)r^m
	\sum_{a=0}^m|C_a^{e,Q_m}(\theta)|.
	\]
	Lemma~\ref{lem:direction} gives
	\[
	\sigma_e\{\theta:B_\theta\le t r^mA_m\}
	\le Ct^\alpha,\qquad 0<t<1,\qquad \alpha=\frac1{2D}.
	\]

	 Since $\delta<\alpha$, we use a dyadic decomposition
to estimate the integral. Set
\[
E_k=\left\{\theta\in\Sigma_e:
2^{-(k+1)}<\frac{B_\theta}{r^mA_m}\le2^{-k}
\right\},
\qquad k\ge0.
\]
On $E_k$, we have
\[
\left(\frac{B_\theta}{r^mA_m}\right)^{-\delta}
\le2^{(k+1)\delta},
\qquad
\sigma_e(E_k)\le C2^{-k\alpha}.
\]
Since $B_\theta>0$ almost everywhere and
$\sigma_e(\Sigma_e)\le C$, it follows that
\[
\begin{aligned}
	\int_{\Sigma_e}
	\left(\frac{B_\theta}{r^mA_m}\right)^{-\delta}
	\,d\sigma_e(\theta)
	&\quad=
	\int_{\{B_\theta>r^mA_m\}}
	\left(\frac{B_\theta}{r^mA_m}\right)^{-\delta}
	\,d\sigma_e(\theta)\\
	&\qquad+
	\sum_{k\ge0}\int_{E_k}
	\left(\frac{B_\theta}{r^mA_m}\right)^{-\delta}
	\,d\sigma_e(\theta)\\
	&\quad\le C+\sum_{k\ge0}
	2^{(k+1)\delta}\sigma_e(E_k)\\
	&\quad\le C+C\sum_{k\ge0}2^{-k(\alpha-\delta)}
	\le C.
\end{aligned}
\]

	Since $L\ge H=d/r$, we have
	\[
	\Lambda_\theta
	=\frac{L}{H^2}B_\theta
	\ge\frac{r}{d}B_\theta
	=\frac{r^{m+1}A_m}{d}
	\frac{B_\theta}{r^mA_m}.
	\]
	Consequently,
	\[
	\begin{aligned}
		\int_{\Sigma_e}\Lambda_\theta^{-\delta}
		\,d\sigma_e(\theta)
		&\le
		\left(\frac{d}{r^{m+1}A_m}\right)^\delta
		\int_{\Sigma_e}
		\left(\frac{B_\theta}{r^mA_m}\right)^{-\delta}
		\,d\sigma_e(\theta)\\
		&\le C\left(\frac{d}{r^{m+1}A_m}\right)^\delta.
	\end{aligned}
	\]
	Integrating \eqref{eq:section-decay} over $\theta$,
	we obtain
	\[
	\begin{aligned}
		|I|
		&\le\int_{\Sigma_e}|I_\theta|\,d\sigma_e(\theta)\\
		&\le C2^{2\eta_0s}\norm\Omega_1^2R^{-n}
		K_H^8\int_{\Sigma_e}\Lambda_\theta^{-\delta}
		\,d\sigma_e(\theta)\\
		&\le C2^{2\eta_0s}\norm\Omega_1^2R^{-n}
		K_H^8\left(\frac{d}{r^{m+1}A_m}\right)^\delta.
	\end{aligned}
	\]
	Finally, since $c=1/(258D)<\delta$ and
	$N=16>8$, the second term in the minimum
	is bounded by
	\[
	K_H^N\left(\frac{d}{r^{m+1}A_m}\right)^c
	\]
	whenever $r^{m+1}A_m/d\ge1$.
	When $r^{m+1}A_m/d<1$, we use the direct
	bound. This proves \eqref{eq:two-high}.

	For the equal-scale estimate, set
	$\widetilde Q(w)=Q(z+Rw)-Q(z)$, so that
	$\norm{\widetilde Q}_\coef=A_Q$.
	Since $H\le C$, \eqref{eq:triangular} gives
	\[
	\begin{aligned}
		cA_Q
		&\le\sup_{|w|\le1}|\widetilde Q(w)|
		\le\sup_{\theta\in\Sigma_e}\norm{q_\theta}_\nc
		\le C\sup_{\theta\in\Sigma_e}B_\theta.
	\end{aligned}
	\]
	Write $S_\theta(p,t)=\sum_{a+b\le D}s_{ab}(\theta)p^at^b$.
	Each $s_{ab}$ is a polynomial in $\theta$ of degree at most
	$D$, and at least one satisfies
	$\norm{s_{ab}}_{L^\infty(\Sigma_e)}\ge cA_Q$.
	For $n\ge3$, Lemma~\ref{lem:sphere} therefore gives
	\[
	\sigma_e\{\theta:B_\theta\le tA_Q\}\le Ct^\alpha,
	\qquad 0<t<1.
	\]
	For $n=2$, the same estimate follows from
	$S_{-\theta}(p,t)=S_\theta(p,-t)$,
	$B_{-\theta}=B_\theta$, and the preceding lower bound.
	The same dyadic summation now yields
	\[
	\int_{\Sigma_e}(B_\theta/A_Q)^{-\delta}\,d\sigma_e(\theta)
	\le C.
	\]
	Since $L/H^2\ge c$ for bounded $H$, integrating
	\eqref{eq:section-decay} gives
	\[
	|I|\le C2^{2\eta_0s}\norm\Omega_1^2R^{-n}
	\min\{1,K_H^8A_Q^{-\delta}\}.
	\]
	Together with the direct bound when $A_Q\le1$, this
	implies \eqref{eq:two-equal} with the same $c$ and $N$.
\end{proof}

\smallskip

\emph{Proof of Theorem 3.1.}

We first consider correlation kernel estimates for the high positive scale part.

\emph{{1. The adjoint and the kernel of $T_{j,\Omega_s}^*T_{k,\Omega_s}$.}}
Fix $s\ge s_*$ and a finite set $\mathcal J\subset\{j>\varrho s\}$. Write
\[
 L_j(x,y)=e^{iP(x,y)}\Omega_s(x-y)k_j(x-y),\qquad
 T_{j,\Omega_s}f(x)=\int L_j(x,y)f(y)\,dy.
\]
The adjoint and the two compositions are given by
\begin{align}
 T_{j,\Omega_s}^*g(y)&=\int\overline{L_j(x,y)}g(x)\,dx,\notag\\
 (T_{k,\Omega_s}T_{j,\Omega_s}^*g)(x)
 &=\int\left[\int L_k(x,y)\overline{L_j(x',y)}\,dy\right]g(x')\,dx',\label{eq:actual-TTstar}\\
 (T_{j,\Omega_s}^*T_{k,\Omega_s}f)(z)
 &=\int\left[\int L_k(x,y)\overline{L_j(x,z)}\,dx\right]f(y)\,dy.\label{eq:actual-TstarT}
\end{align}
When expanding the square of the $L^2$ norm below, we use the kernel in $y,z$ obtained by integration in $x$ in \eqref{eq:actual-TstarT}. More precisely,
\begin{align}
 G_{jk}(y,z)
 &:=\int L_k(x,y)\overline{L_j(x,z)}\,dx\notag\\
 &=\int e^{i[P(x,y)-P(x,z)]}\Omega_s(x-y)\overline{\Omega_s(x-z)}
 k_k(x-y)\overline{k_j(x-z)}\,dx.\label{eq:correlation}
\end{align}
Thus, with the inner product $\langle f,g\rangle=\int f\overline g$,
\begin{align}
 \norm{\sum_{j\in\mathcal J}T_{j,\Omega_s}B_{j-s}}_2^2
 &=\sum_{j,k\in\mathcal J}
 \langle T_{k,\Omega_s}B_{k-s},T_{j,\Omega_s}B_{j-s}\rangle\notag\\
 &=\sum_{j,k\in\mathcal J}\langle T_{j,\Omega_s}^*T_{k,\Omega_s}B_{k-s},B_{j-s}\rangle\notag\\
 &=\sum_{j,k\in\mathcal J}\iint G_{jk}(y,z)
 B_{k-s}(y)\overline{B_{j-s}(z)}\,dy\,dz.\label{eq:TT-full}
\end{align}
If $j\ge k$, $R=2^j$, and $r=2^k$, then
\begin{align}
 \supp G_{jk}&\subset\{|y-z|\le CR\},\notag\end{align} and 
 \begin{align}|G_{jk}(y,z)|&\le \int|L_k(x,y)L_j(x,z)|\,dx\notag\\
 &\le C2^{2\eta_0s}\norm\Omega_1^2 r^{-n}R^{-n}\abs{B(y,Cr)}\notag\\&\le C2^{2\eta_0s}\norm\Omega_1^2 R^{-n}.\label{eq:high-direct}
\end{align}

\medskip
\noindent{\emph{2. Entire bad cubes intersecting the small sets.}}
By \eqref{eq:normal}, choose $q=q_{\alpha_*}$ so that the sum of the absolute values of its degree-$l$ coefficients is at least $c$. Fix $j,z$ and set $R=2^j,$
\begin{align*}
 \mathcal N&=\{y:|y-z|\le R2^{-\beta_Hs}\},\\
 \mathcal Z&=\{y:|y-z|\le CR,\ |q(y)-q(z)|\le R2^{-\beta_Hs}\},\\
 \mathcal G&=B(z,CR)\setminus(\mathcal N\cup\mathcal Z),
\end{align*} where $\beta_H>0$ will be chosen below.

When $\varrho s<k\le j$, the relevant bad cubes satisfy $\ell(Q)=2^{k-s}\le R2^{-s}$, and the cubes $Q$ are pairwise disjoint across all scales. Hence for $ \mathcal N$
\begin{align}
 \sum_{\varrho s<k\le j}\int_{\mathcal N}|B_{k-s}(y)|\,dy
 &\le\sum_{\substack{L(Q)\le j-s\\Q\cap\mathcal N\ne\varnothing}}\norm{b_Q}_1\notag\\
 &\le C\tau\abs{\bigcup_{Q\cap\mathcal N\ne\varnothing}Q}
 \le C\tau R^n(2^{-\beta_Hs}+2^{-s})^n.\label{eq:high-near}
\end{align}
For $\mathcal Z$, set $y=z+Ru$ and $F_z(u)=q(z+Ru)-q(z)$. Its highest-degree coefficients satisfy
\[
 \norm{F_z}_\coef\ge R^l\sum_{|\beta|=l}|[y^\beta]q(y)|
 \ge cR^l\ge cR,
\]
where $R=2^j>1$ and $l\ge1$. Thus
\[
 0<B_0:=R\le C\norm{F_z}_\coef,
\]
satisfies the hypothesis $B_0\le CB$ of Lemma \ref{lem:sublevel}, with $B=\norm{F_z}_\coef$.
If $L(Q)=k-s$ and $\varrho s<k\le j$, then under the rescaling $u=R^{-1}(y-z)$,
\[
 \ell(R^{-1}(Q-z))=R^{-1}2^{k-s}
 =2^{k-j-s}\le2^{-s}.
\]
and every normalized bad cube contributing to $\mathcal Z$ intersects
\[
 \{u\in B(0,C):|F_z(u)|\le R2^{-\beta_Hs}\}
 =\{u\in B(0,C):|F_z(u)|\le tB_0\},
 \qquad t=2^{-\beta_Hs},
\]
Thus, taking in \eqref{eq:thick} of Lemma \ref{lem:sublevel}
\[
 B_0=R,\qquad t=2^{-\beta_Hs},\qquad \ell=2^{-s},
\]
we obtain
\begin{equation}\label{eq:high-zero}
 \sum_{\varrho s<k\le j}\int_{\mathcal Z}|B_{k-s}(y)|\,dy
 \le C\tau R^n(2^{-\beta_Hs}+2^{-s})^{1/D}.
\end{equation}

\medskip
\noindent{\emph{3. Oscillatory decay on the good region and computation of the exponents.}}
For $y\in\mathcal G$, the phase $Q(x)=P(x,y)-P(x,z)$ has highest $x$-degree $m$, and
\[
 A_m=\sum_{|\alpha|=m}|q_\alpha(y)-q_\alpha(z)|
 \ge|q(y)-q(z)|>R2^{-\beta_Hs},\qquad
 R2^{-\beta_Hs}<d=|y-z|\le CR.
\]
In the correlation kernel, the smaller-scale center is $y$ and the larger-scale center is $z$. Apply Lemma \ref{lem:twocenter} to $\overline{G_{jk}(y,z)}$, interchanging the two centers in the lemma and using phase $-Q$ and amplitude $\overline{a_{jk}}$. The two rough kernel factors then agree with those in \eqref{eq:two-integral}. Complex conjugation does not change the absolute value of the correlation kernel or the amplitude bounds below.
Define
\[
 a_{jk}(x)=\left(\frac r{|x-y|}\right)^n
           \left(\frac R{|x-z|}\right)^n
 \phi(r^{-1}|x-y|)\overline{\phi(R^{-1}|x-z|)},
\] where $R=2^j$ and $r=2^k.$
Then
\[
 G_{jk}(y,z)=R^{-n}r^{-n}\int e^{iQ(x)}a_{jk}(x)
 \overline{\Omega_s(x-z)}\Omega_s(x-y)dx.
\]
By the annular support,
\[
 |x-y|\simeq r,
 \qquad |x-z|\simeq R.
\]
In the section coordinates $x=y+r(pe+t\theta)$ used in the lemma,
$|x-y|\simeq r$ implies $p^2+t^2\simeq1$, so that $|p|\le C$ and $0<t\le C$. The smaller-scale factor is a fixed $C^\infty$ annular function. If
\[
 a_R(x)=\left(\frac R{|x-z|}\right)^n\overline{\phi(R^{-1}|x-z|)},
\]
then, for $a+b\le2$,
\[
 |\partial_p^a\partial_t^b[a_R(y+r(pe+t\theta))]|
 \le C\left(\frac rR\right)^{a+b}\le C.
\]
The product rule therefore gives
\[
 \sup_\theta\max_{a+b\le2}
 \norm{\partial_p^a\partial_t^b
       [a_{jk}(y+r(pe+t\theta))]}_\infty\le C,
\]
Thus all the support and $C^2$ amplitude hypotheses of Lemma \ref{lem:twocenter} hold, with constants independent of $j,k,y,z$. By \eqref{eq:two-high},
\[
 |G_{jk}(y,z)|\le C2^{2\eta_0s}\norm\Omega_1^2 R^{-n}
 \left(1+\frac rd\right)^N
 \left(\frac d{r^{m+1}A_m}\right)^c.
\]
The individual factors satisfy
\begin{align*}
 1+\frac rd
 \le1+\frac rR2^{\beta_Hs}\le2\cdot2^{\beta_Hs},\end{align*} and \begin{align*}
 \frac d{r^{m+1}A_m}
 &\le\frac{CR}{r^{m+1}R2^{-\beta_Hs}}\\&
 =Cr^{-(m+1)}2^{\beta_Hs}\le C2^{[\beta_H-(m+1)\varrho]s},
 \qquad r=2^k>2^{\varrho s}.
\end{align*}
Hence
\begin{align}
 |G_{jk}(y,z)|
 &\le C2^{2\eta_0s}\norm\Omega_1^2 R^{-n}2^{N\beta_Hs}
                      2^{c[\beta_H-(m+1)\varrho]s}\notag\\
 &=C2^{2\eta_0s}\norm\Omega_1^2 R^{-n}2^{[(N+c)\beta_H-c(m+1)\varrho]s}.\label{eq:high-good}
\end{align}
Set
\[
 \beta_H=\frac{c\varrho}{4(N+c)}<1,\qquad
 \kappa_1=\frac{\beta_H}{D}>0.
\]
Since $m\ge1$,
\[
 (N+c)\beta_H-c(m+1)\varrho
 =c\varrho\left(\tfrac14-m-1\right)
 \le-\tfrac74c\varrho<-\kappa_1.
\]
The union of all relevant bad cubes intersecting $B(z,CR)$ is contained in $B(z,C'R)$. Hence
\begin{align}\label{Ch}
 \sum_{\varrho s<k\le j}\int_{B(z,CR)}|B_{k-s}(y)|\,dy
 \le C\tau\abs{B(z,C'R)}\le C\tau R^n.
\end{align}

For the sets \(\mathcal{N}\) and \(\mathcal{Z}\), we apply \eqref{eq:high-direct}, \eqref{eq:high-near}, and \eqref{eq:high-zero}, whereas for \(\mathcal{G}\) we use \eqref{eq:high-good} and \eqref{Ch}. Combining these estimates, we obtain
\begin{align}
 \sum_{\varrho s<k\le j}\int|G_{jk}(y,z)||B_{k-s}(y)|\,dy
 &\le C\tau 2^{2\eta_0s}\norm\Omega_1^2\left[(2^{-\beta_Hs}+2^{-s})^n
 +(2^{-\beta_Hs}+2^{-s})^{1/D}+2^{-7c\varrho s/4}\right]\notag\\
 &\le C\tau 2^{2\eta_0s}\norm\Omega_1^2\,2^{-\kappa_1s}.\label{eq:high-schur}
\end{align}

\medskip
\noindent{\emph{4. Completion of the correlation sum.}}
Using $G_{kj}(z,y)=\overline{G_{jk}(y,z)}$, \eqref{eq:TT-full}, and \eqref{eq:high-schur}, we obtain
\begin{align*}
 \norm{\sum_{j\in\mathcal J}T_{j,\Omega_s}B_{j-s}}_2^2
 &\le2\sum_{j\in\mathcal J}\int|B_{j-s}(z)|
 \sum_{\substack{k\in\mathcal J\\k\le j}}
 \int|G_{jk}(y,z)||B_{k-s}(y)|\,dy\,dz\\
 &\le C\tau 2^{2\eta_0s}\norm\Omega_1^2\,2^{-\kappa_1s}\sum_{j\in\mathcal J}\norm{B_{j-s}}_1.
\end{align*}
This is \eqref{eq:prop-high}. The same estimate holds for finite tails, and $\sum_k\norm{B_k}_1<\infty$. Hence the corresponding high-scale series converges in $L^2$. 
This completes the proof of Theorem 3.1.
\Needspace{8\baselineskip}
\section{Proof of Theorem \ref{prop:lowfull}}\label{sec:reduction}
The local-scale part $T_{\le0}$, the good function, and the contribution
of $\nu_s$ have been controlled in \eqref{eq:T0weak}, \eqref{eq:good},
and \eqref{eq:large-angular}, respectively.
Section \ref{sec:high} proves the estimate for $j>\varrho s$.
It therefore remains to estimate
\[
 \sum_{s\ge s_*}\sum_{1\le j\le\varrho s}
 T_{j,\Omega_s}B_{j-s}.
\]
By the summation in Section \ref{sec:main-proof}, it suffices to establish
the decomposition in Theorem \ref{prop:lowfull}, with $L^2$ decay
for its main term and $L^1$ decay for its two remainders.

Fix $s\ge s_*$ and $1\le j\le\varrho s$, and set $R=2^j$.
In the proof of \eqref{eq:high-good}, the annular scales are larger
than $2^{\varrho s}$, and this lower bound yields decay in $s$.
In the present range, $j$ may remain fixed as $s$ increases, so we
cannot use the same lower bound. On the other hand, every bad cube
occurring in $B_{j-s}$ has sidelength
\[
 \ell(Q)=2^{j-s}=R\cdot 2^{-s}.
\]
We therefore seek to use the separation between the annular scale and
the scale of the bad cubes. Seeger's convolution estimates provide
decay from this separation, but the mixed phase prevents us from
applying them directly to $T_{j,\Omega_s}$.
We first localize the operator and expand the mixed phase by Fourier series.
The dependence of the Fourier coefficient bounds on the mixed phase
will then determine which localized terms can be treated in this way.

\subsection{Localization at a fixed positive annular scale}\label{subsec:localization}
The annular support restricts the interaction to $|x-y|\simeq R$.
Localizing both variables to boxes of this size allows us to measure the
mixed coefficients on a fixed box after rescaling. The localization must
also preserve the total $L^1$ mass of the bad functions up to a fixed
factor, so that the local estimates can later be summed over $h$.
Let $I_h=R\cdot(h+[0,1)^n)$, $h\in\Z^n$. Set
\begin{equation}\label{eq:parent-cubes}
 \mathcal Q_h=\{Q:L(Q)=j-s,\ \dist(Q,I_h)\le CR\},\qquad
 b_h=\sum_{Q\in\mathcal Q_h}b_Q,
\end{equation}
Here $C$ exceeds the fixed constant required by the annular support. Each fixed $Q$ belongs to only a bounded number of the families $\mathcal Q_h$: after division by $R$, the admissible $h$ lie among the lattice points in a bounded region. Thus
\begin{equation}\label{eq:parent-overlap}
 \sum_h\sum_{Q\in\mathcal Q_h}\norm{b_Q}_1\le C\norm{B_{j-s}}_1.
\end{equation}
Choose $\chi_h^{x},\chi_h^{y}\in C_c^\infty$ taking values in $[0,1]$, equal to $1$ on $I_h$ and $\bigcup_{Q\in\mathcal Q_h}Q$, respectively, and supported in a fixed dilate of $I_h$, with derivatives of order $a$ bounded by $O_a(R^{-a})$. Define
\begin{equation}\label{eq:local-operator}
 T_{j,\Omega_s}^{(h)}u=\chi_h^{x}T_{j,\Omega_s}(\chi_h^{y}u).
\end{equation}
The support of the kernel gives
\begin{equation}\label{eq:parent-identity}
 \one_{I_h}T_{j,\Omega_s}B_{j-s}=\one_{I_h}T_{j,\Omega_s}^{(h)}b_h.
\end{equation}
The derivative estimates required for localization apply only to the smooth cutoffs. At the end, we restrict each localized term to $I_h$ and sum; no derivative is applied to $\one_{I_h}$.

The factors depending on only one variable can be kept as unimodular
multipliers. We therefore isolate the part that couples $x$ and $y$.
Set $x_h=R\cdot h$ and define directly in the original variables
\begin{align}\label{eq:local-phase}
 \Pi_h(x,y)&=P(x,y)-P(x,x_h)-P(x_h,y)+P(x_h,x_h),\end{align} and \begin{align}
 A_h(x)&=P(x,x_h)-P(x_h,x_h),\qquad
 C_h(y)=P(x_h,y)-P(x_h,x_h).\notag
\end{align}
Then
\begin{equation}\label{eq:phase-split}
 P(x,y)=P(x_h,x_h)+A_h(x)+C_h(y)+\Pi_h(x,y).
\end{equation}
Expand the mixed part at scale $R$:
\[
 \Pi_h(x,y)=\sum_{|\alpha|,|\beta|\ge1}
 p_{\alpha\beta}^h
 \left(\frac{x-x_h}{R}\right)^\alpha
 \left(\frac{y-x_h}{R}\right)^\beta,
 \qquad\mathfrak M_h=\sum_{\alpha,\beta}|p_{\alpha\beta}^h|.
\]
Every nonzero term satisfies $|\alpha|+|\beta|\le D$. On the localized support,
$|x-x_h|+|y-x_h|\le CR$, so derivatives of order $a$ of the mixed phase
are bounded by $C\mathfrak M_hR^{-a}$.
We next keep track of the dependence on $\mathfrak M_h$ when separating
the mixed phase by Fourier expansion.

\subsection{Fourier expansion of the mixed phase}\label{subsec:fourier}
To obtain convolution operators, we expand the smooth mixed phase
factor into products of functions of $x$ and functions of $y$.
After estimating each Fourier mode, we must sum over all frequencies.
Thus we need an absolutely summable bound for the Fourier coefficients,
with its dependence on the polynomial coefficients made explicit.
We also include the weight needed for the mean correction below.
The following lemma gives this bound without assuming that the
coefficients are small.

\begin{lemma}\label{lem:fourier}
Fix $L_0>0$ and $\chi_1,\chi_2\in C_c^\infty((-L_0,L_0)^n)$. Let $\Pi$ be a real polynomial of degree at most $D$, and set $\mathfrak M=\norm\Pi_\coef$ and $\gamma=(\gamma_1,\gamma_2)\in\Z^{2n}$. On $\mathbb T_{L_0}^{2n}=(\R/(2L_0\Z))^{2n}$, write
\[
 F(x,y)=\chi_1(x)\chi_2(y)e^{i\Pi(x,y)}
 =\sum_{\gamma\in\Z^{2n}}c_{\gamma}e^{i(\pi/L_0)(\gamma_1\cdot x+\gamma_2\cdot y)}.
\]
Then
\begin{align}
 |c_{\gamma}|\le C(1+\mathfrak M)^{4n+4}
 (1+|\gamma_1|+|\gamma_2|)^{-4n-4},\label{eq:Fourier-coefficient-lemma}
 \end{align} and \begin{align}\sum_{\gamma}|c_{\gamma}|(1+|\gamma_2|+\mathfrak M)
 \le C(1+\mathfrak M)^{4n+5}
 \le C(1+\mathfrak M)^{N_0},\qquad N_0=4n+6.\label{eq:Fourier-weight-lemma}
\end{align}
The constant may also depend on the fixed torus and cutoff functions.
\end{lemma}
\begin{proof}

We first bound the derivatives needed for Fourier coefficient decay in
terms of $\mathfrak M$. The rough function $\Omega_s$ is absent from
$F$, so all differentiations here act on smooth functions.
For $1\le|\mu|\le4n+4$, differentiation of the polynomial and the product rule give
\begin{align*}
 \norm{\partial^\mu\Pi}_{L^\infty([-L_0,L_0]^{2n})}&\le C\mathfrak M,\end{align*} and
 \begin{align*}\partial^\mu e^{i\Pi}
 &=e^{i\Pi}\sum_{r=1}^{|\mu|}
 \sum_{\substack{\mu_1+\cdots+\mu_r=\mu\\|\mu_j|\ge1}}
 a_{\mu_1,\ldots,\mu_r}i^r\prod_{j=1}^r\partial^{\mu_j}\Pi.\end{align*} Then
 \begin{align*}\norm{\partial^\mu e^{i\Pi}}_\infty
 &\le C\sum_{r=1}^{|\mu|}\mathfrak M^r
 \le C(1+\mathfrak M)^{|\mu|},\end{align*}and \begin{align*}
 \max_{|\mu|\le4n+4}\norm{\partial^\mu F}_\infty
 &\le C(1+\mathfrak M)^{4n+4}.
\end{align*}
\medskip

Applying $(1-\Delta_{x,y})^{2n+2}$ produces decay of order $4n+4$
in the $2n$ Fourier variables. This order leaves enough room both to
sum all the frequencies and to include the additional weight needed
for the mean correction. Using
\[
 c_{\gamma}=(2L_0)^{-2n}\int_{[-L_0,L_0]^{2n}}
 F(x,y)e^{-i(\pi/L_0)(\gamma_1\cdot x+\gamma_2\cdot y)}\,dx\,dy
\]
and integration by parts on the torus, we obtain
\begin{align*}
 &\bigl(1+(\pi/L_0)^2(|\gamma_1|^2+|\gamma_2|^2)\bigr)^{2n+2}c_{\gamma}\\
 &\qquad=(2L_0)^{-2n}\int
 (1-\Delta_{x,y})^{2n+2}F(x,y)
 e^{-i(\pi/L_0)(\gamma_1\cdot x+\gamma_2\cdot y)}\,dx\,dy.
\end{align*}
Therefore,
\begin{align*}
 |c_{\gamma}|&\le C(1+\mathfrak M)^{4n+4}
 \bigl(1+(\pi/L_0)^2(|\gamma_1|^2+|\gamma_2|^2)\bigr)^{-2n-2}\\
 &\le C(1+\mathfrak M)^{4n+4}(1+|\gamma_1|+|\gamma_2|)^{-4n-4}.
\end{align*}
Using also $1+|\gamma_2|+\mathfrak M\le(1+\mathfrak M)(1+|\gamma_1|+|\gamma_2|)$, we obtain
\begin{align*}
 \sum_{\gamma}|c_{\gamma}|(1+|\gamma_2|+\mathfrak M)
 &\le C(1+\mathfrak M)^{4n+5}
 \sum_{\gamma}(1+|\gamma_1|+|\gamma_2|)^{-4n-3}.\end{align*}Note that \begin{align*}
 \sum_{\gamma}(1+|\gamma_1|+|\gamma_2|)^{-4n-3}
 &\le1+C\sum_{\ell\ge0}
 \operatorname{card}\{(\gamma_1,\gamma_2):2^\ell\le1+|\gamma_1|+|\gamma_2|<2^{\ell+1}\}\,2^{-(4n+3)\ell}\\
 &\le1+C\sum_{\ell\ge0}2^{2n\ell}2^{-(4n+3)\ell}\\&
 =1+C\sum_{\ell\ge0}2^{-(2n+3)\ell}<\infty.
\end{align*}
This proves \eqref{eq:Fourier-weight-lemma} and also establishes absolute and uniform convergence of the Fourier series. This completes the proof of Lemma \ref{lem:fourier}.
\end{proof}

Choose fixed $L_0$ and smooth cutoffs $\chi_1,\chi_2$ such that
$\chi_1((x-x_h)/R)\chi_2((y-x_h)/R)=1$ on the support of the localized
operator. Applying Lemma \ref{lem:fourier} after rescaling by $R$ gives
\begin{equation}\label{eq:mixed-fourier}
 \chi_1\!\left(\frac{x-x_h}{R}\right)
 \chi_2\!\left(\frac{y-x_h}{R}\right)e^{i\Pi_h(x,y)}
 =\sum_{\gamma\in\Z^{2n}}c_{\gamma}^h
 e^{i\pi\gamma_1\cdot(x-x_h)/(L_0R)}
 e^{i\pi\gamma_2\cdot(y-x_h)/(L_0R)},
\end{equation}
where $\gamma=(\gamma_1,\gamma_2)$, and
\begin{equation}\label{eq:fourier-coeff-local}
 c_{\gamma}^h=(2L_0R)^{-2n}\iint
 \chi_1\!\left(\frac{x-x_h}{R}\right)
 \chi_2\!\left(\frac{y-x_h}{R}\right)e^{i\Pi_h(x,y)}
 e^{-i\pi[\gamma_1\cdot(x-x_h)+\gamma_2\cdot(y-x_h)]/(L_0R)}\,dx\,dy.
\end{equation}
The coefficient norm of the rescaled mixed phase is $\mathfrak M_h$.
Hence
\begin{equation}\label{eq:fourier-weight-general}
 \sum_{\gamma}|c_{\gamma}^h|(1+|\gamma_2|+\mathfrak M_h)
 \le C(1+\mathfrak M_h)^{N_0}.
\end{equation}
In each Fourier mode, the factor depending on $x$ is taken outside the
integral and the factor depending on $y$ is multiplied into $b_Q$.
The remaining kernel is $\Omega_s(x-y)k_j(x-y)$ and therefore has a
convolution structure. The pure phases $A_h,C_h$ are kept as
unimodular multipliers, so their coefficients do not enter
\eqref{eq:fourier-weight-general}.

The Fourier expansion is valid for every $\mathfrak M_h$.
However, its coefficient bound is not uniform in $\mathfrak M_h$.
The normalization \eqref{eq:normal} controls only the highest mixed
bidegree; it does not bound all the lower mixed coefficients or the
coefficients obtained by translation to $x_h$.
Consequently, applying \eqref{eq:fourier-weight-general} on every
localized region would leave a factor depending on the individual
coefficients of $P$, which is not allowed in Theorem
\ref{prop:lowfull}.

We compare this factor with the decay supplied by Seeger's estimates (\cite{Seeger1996}).
With the parameters used in Section \ref{sec:microlocal}, those
estimates give a factor $2^{-s/4}$ in the $L^2$ norm of the main term
and in the $L^1$ norm of the remainder, apart from the kernel size
and the mass factors; see \eqref{eq:mode-G-seeger-root} and
\eqref{eq:mode-rem-seeger}.
After Fourier summation, the corresponding factor is bounded by
\[
 C2^{-s/4}(1+\mathfrak M_h)^{N_0}.
\]
To retain decay in $s$, it is therefore sufficient to assume
$\mathfrak M_h\le2^{\sigma s}$, where $\sigma$ is fixed by
\eqref{eq:sigma}. Under this assumption,
\begin{equation}\label{eq:fourier-weight}
 \sum_{\gamma}|c_{\gamma}^h|(1+|\gamma_2|+\mathfrak M_h)
 \le C(1+\mathfrak M_h)^{N_0}\le C2^{N_0\sigma s},
\end{equation}
and
\[
 2^{-s/4}(1+\mathfrak M_h)^{N_0}
 \le C2^{-(1/4-N_0\sigma)s},
 \qquad \frac14-N_0\sigma>0.
\]
The additional factor $2^{\eta_0s}$ from the truncated kernel is
absorbed by the subsequent choice of $\eta_0$ in \eqref{eq:eta}.

This leads to the division according to the size of the local mixed
coefficients. When $\mathfrak M_h\le2^{\sigma s}$, we use the Fourier
expansion, correct the means of the bad functions, and apply Seeger's
convolution estimates. When $\mathfrak M_h>2^{\sigma s}$, the preceding
Fourier coefficient bound does not give the required uniform decay.
In this case, we instead use the size of the mixed coefficients in
the equal-scale two-center estimate \eqref{eq:two-equal}.
We first prove the estimate for these large local mixed phases and
then return to the mean correction for small mixed phases.

\subsection{Large local mixed phases}\label{subsec:large-local}
First, we concentrate on the case $\mathfrak M_h>2^{\sigma s}$.
The condition $\mathfrak M_h>2^{\sigma s}$ gives a lower bound for the
mixed coefficient norm. It does not imply a lower bound for every
phase difference $\Pi_h(x,y)-\Pi_h(x,z)$.
As in Section 4, we first estimate the
near-diagonal region and the region where the phase difference has
small nonconstant coefficients. On the complement, Lemma
\ref{lem:twocenter} gives oscillatory decay.
The following lemma controls the entire localized term in $L^2$
without using mean-zero cancellation.

 \begin{lemma}[Large local mixed phases]\label{lem:large}
Let $x_0\in\R^n$, $R>0$, and $s\ge s_*$. Suppose that the real polynomial
\[
 \Pi(x,y)=\sum_{|\alpha|,|\beta|\ge1}p_{\alpha\beta}
 \left(\frac{x-x_0}{R}\right)^\alpha
 \left(\frac{y-x_0}{R}\right)^\beta
\]
has degree at most $D$ and $\mathfrak M=\sum_{\alpha,\beta}|p_{\alpha\beta}|>2^{\sigma s}$.
Let $\chi^x,\chi^y\in C_c^\infty$ take values in $[0,1]$, be supported in $x_0+RU_0$, where $U_0$ is a fixed box, and satisfy
$$\norm{\partial^\alpha\chi^x}_\infty+\norm{\partial^\alpha\chi^y}_\infty\le CR^{-|\alpha|},\,\,|\alpha|\le2.$$
For arbitrary real-valued measurable functions $a_0,c_0$, define
\begin{equation}\label{eq:large-model}
 \mathcal T_Rb(x)=\chi^x(x)\int
 e^{i[a_0(x)+c_0(y)+\Pi(x,y)]}\Omega_s(x-y)
 \frac{\phi(R^{-1}|x-y|)}{|x-y|^n}\chi^y(y)b(y)\,dy.
\end{equation}
If $b=\sum_{Q\in\mathcal Q}b_Q$, where the cubes $Q$ are pairwise disjoint, have sidelength $R2^{-s}$, and are contained in $x_0+RU_0$, and
$\supp b_Q\subset Q$, $\norm{b_Q}_1\le C\tau|Q|$, then
\begin{equation}\label{eq:large-bound}
 \norm{\mathcal T_Rb}_2^2
 \le C\tau\norm\Omega_1^2\,2^{-(\delta_0-2\eta_0)s}
 \sum_{Q\in\mathcal Q}\norm{b_Q}_1.
\end{equation}
One may take $\delta_0=c\sigma/([4(N+c)]D)$, where $c=1/(258D)$ and $N=16$.
\end{lemma}
\begin{proof}
\noindent{ \emph{The correlation kernel}.}
To use the two-center estimate, we expand the square of the $L^2$ norm.
This removes the pure $x$-phase and turns the mixed phase into a
polynomial difference in $y,z$.
Set $Q_{y,z}(x)=\Pi(x,y)-\Pi(x,z)$ and
\[
 K(y,z)=\chi^y(y)\overline{\chi^y(z)}
 \int e^{iQ_{y,z}(x)}|\chi^x(x)|^2\Omega_s(x-y)\overline{\Omega_s(x-z)}
 \frac{\phi(R^{-1}|x-y|)\overline{\phi(R^{-1}|x-z|)}}{|x-y|^n|x-z|^n}\,dx.
\]
When the square is expanded, the pure $x$-phase cancels. Hence
\begin{equation}\label{eq:large-correlation}
 \norm{\mathcal T_Rb}_2^2
 =\iint e^{i[c_0(y)-c_0(z)]}K(y,z)b(y)\overline{b(z)}\,dy\,dz.
\end{equation}
By the annular support and \eqref{eq:Omega-s-infty},
$|K(y,z)|\le C2^{2\eta_0s}\norm\Omega_1^2R^{-n}$.
After removing the rough kernel factors and multiplying the amplitude by $R^{2n}$, its derivatives up to order two in the section coordinates are uniformly bounded. We may therefore apply the equal-scale estimate in Lemma \ref{lem:twocenter}.

\medskip
\noindent{ \emph{The two exceptional regions}.}
The equal-scale estimate deteriorates when $y$ and $z$ are too close or
when the nonconstant coefficients of $Q_{y,z}$ are too small.
We isolate these two possibilities. Since the coefficients of $\Pi$ are
large, at least one of its coefficient polynomials in $y$ is large;
a sublevel set estimate then controls where its difference is small.
Write
\[
 \Pi(x,y)=\sum_{\alpha\ne0}
 \left(\frac{x-x_0}{R}\right)^\alpha q_\alpha\!\left(\frac{y-x_0}{R}\right),
 \qquad q_\alpha(0)=0,
\]
and choose $q=q_{\alpha_*}$ so that $\norm q_\nc=\norm q_\coef\ge c\mathfrak M$.
Fix $z\in x_0+RU_0$ and set
\begin{align*}
 \mathcal N_z&=\{y\in x_0+RU_0:|y-z|\le R2^{-\beta_Ls}\},\\
 \mathcal Z_z&=\left\{y\in x_0+RU_0:
 \left|q\!\left(\frac{y-x_0}{R}\right)-q\!\left(\frac{z-x_0}{R}\right)\right|
 \le\mathfrak M2^{-\beta_Ls}\right\},
\end{align*}
Here $\beta_L>0$ will be chosen below to balance the small-set bounds
against the loss in the oscillatory estimate.

Each bad cube has sidelength $R\cdot2^{-s}$. The bound
$\norm{b_Q}_1\le C\tau|Q|$ controls the mass on an entire cube, so we
include every cube intersecting either exceptional region in its entirety.
Applying Lemma \ref{lem:sublevel} in the variable $(y-x_0)/R$, we obtain
\begin{align}\label{eq:large-smallsets}
 \int_{\mathcal N_z}|b(y)|\,dy
 &\le C\tau R^n(2^{-\beta_Ls}+2^{-s})^n,\notag\\
 \int_{\mathcal Z_z}|b(y)|\,dy
 &\le C\tau R^n(2^{-\beta_Ls}+2^{-s})^{1/D}.
\end{align}
Here we have used $\norm{q-q((z-x_0)/R)}_\coef\ge\norm q_\nc\ge c\mathfrak M$.

\medskip
\noindent{\emph{Oscillation on the complement and the final bound}.}

Outside $\mathcal N_z\cup\mathcal Z_z$, we have
$d=|y-z|>R2^{-\beta_Ls}$. 
we obtain
\[
\begin{aligned}
	\norm{Q_{y,z}(x_0+R\cdot)}_\nc
	&=\sum_{\alpha\ne0}
	|q_\alpha(\frac{y-x_0}{R})-q_\alpha(\frac{z-x_0}{R})|\\
	&\ge |q(\frac{y-x_0}{R})-q(\frac{z-x_0}{R})|\\
	&>\mathfrak M2^{-\beta_Ls}.
\end{aligned}
\]
Since $|\frac{z-x_0}{R}|\le C$, the translation estimate
\eqref{eq:translation} gives
\[
\norm{Q_{y,z}(z+R\cdot)}_\nc
\ge c\norm{Q_{y,z}(x_0+R\cdot)}_\nc
\ge c\mathfrak M2^{-\beta_Ls}.
\]

Lemma \ref{lem:twocenter} therefore gives
\begin{equation}\label{eq:large-good}
 |K(y,z)|\le C2^{2\eta_0s}\norm\Omega_1^2R^{-n}
 2^{-[c\sigma-(N+c)\beta_L]s}.
\end{equation}
Choose $$\beta_L=c\sigma/[4(N+c)]\,\,\, \hbox{and}\,\,\,\delta_0=\beta_L/D.$$ Then
$c\sigma-(N+c)\beta_L=3c\sigma/4>\delta_0$. Combining the preceding small-set estimates, the direct kernel bound, and $\int|b|\le C\tau R^n$, we obtain
\[
 \int|K(y,z)||b(y)|\,dy
 \le C\tau\norm\Omega_1^2\,2^{-(\delta_0-2\eta_0)s}.
\]
Substituting this into \eqref{eq:large-correlation} and using $\norm b_1\le\sum_{Q\in\mathcal Q}\norm{b_Q}_1$ proves \eqref{eq:large-bound}. The proof of Lemma \ref{lem:large} is finished.
\end{proof}

When $\mathfrak M_h>2^{\sigma s}$, take $x_0=x_h$, $\Pi=\Pi_h$,
$a_0(x)=P(x_h,x_h)+A_h(x)$, $c_0(y)=C_h(y)$, and $\mathcal Q=\mathcal Q_h$. The preceding lemma gives
\begin{equation}\label{eq:large-application}
 \norm{T_{j,\Omega_s}^{(h)}b_h}_2^2
 \le C\tau\norm\Omega_1^2\,2^{-(\delta_0-2\eta_0)s}
 \sum_{Q\in\mathcal Q_h}\norm{b_Q}_1.
\end{equation}

\subsection{Mean correction for small mixed phases}\label{subsec:mean-correction}
We now assume that $\mathfrak M_h\le2^{\sigma s}$.
The Fourier expansion \eqref{eq:mixed-fourier} reduces the localized
operator to convolution terms, and \eqref{eq:fourier-weight} controls
their coefficient sum. To apply Seeger's remainder estimate, however,
the functions in these convolutions must have ordinary mean zero.
The phase-adapted cancellation \eqref{eq:cancel} does not immediately
give this property after multiplication by a Fourier mode.
We therefore estimate their means and subtract them on each bad cube.
For $Q\in\mathcal Q_h$ and $\gamma_2\in\Z^n$, define
\begin{equation}\label{eq:mode-input}
 f_{Q,\gamma_2}(y)=e^{iC_h(y)}e^{i\pi\gamma_2\cdot(y-x_h)/(L_0R)}b_Q(y),
 \qquad m_{Q,\gamma_2}=\int_Qf_{Q,\gamma_2}(y)\,dy.
\end{equation}
The cancellation available from Section \ref{sec:main-proof} involves
$P(c_Q,y)$. Using \eqref{eq:phase-split} to factor out its constant part,
\eqref{eq:cancel} becomes
\[
 \int_Qe^{iC_h(y)}e^{i\Pi_h(c_Q,y)}b_Q(y)\,dy=0.
\]
We compare the Fourier factor with this cancelling phase, choosing the
constant multiplier so that the two agree at $y=c_Q$.
Taking $\zeta=\exp\{i\pi\gamma_2\cdot(c_Q-x_h)/(L_0R)-i\Pi_h(c_Q,c_Q)\}$, we obtain
\[
 m_{Q,\gamma_2}=\int_Qe^{iC_h(y)}
 \left[e^{i\pi\gamma_2\cdot(y-x_h)/(L_0R)}
 -\zeta e^{i\Pi_h(c_Q,y)}\right]b_Q(y)\,dy.
\]
The difference in brackets vanishes at $y=c_Q$, and its gradient is bounded by
$CR^{-1}\cdot(|\gamma_2|+\mathfrak M_h)$. Since $|y-c_Q|\le CR\cdot 2^{-s}$, we obtain
\begin{equation}\label{eq:mode-smallmean}
 |m_{Q,\gamma_2}|\le C2^{-s}(|\gamma_2|+\mathfrak M_h)\norm{b_Q}_1.
\end{equation}
The factor $2^{-s}$ comes from $\ell(Q)/R$.
The pure phase $e^{iC_h(y)}$ remains outside the difference, so its
derivatives do not enter the bound. We also have
$|m_{Q,\gamma_2}|\le\norm{b_Q}_1$, which will keep the mass of the
corrected function uniformly bounded for all frequencies.
Subtract the average on the same cube and set
\begin{gather}\label{eq:mode-correct}
 \widetilde f_{Q,\gamma_2}=f_{Q,\gamma_2}-m_{Q,\gamma_2}|Q|^{-1}\one_Q,
 \qquad F_{\gamma_2}=\sum_{Q\in\mathcal Q_h}f_{Q,\gamma_2},\,
 \widetilde F_{\gamma_2}=\sum_{Q\in\mathcal Q_h}\widetilde f_{Q,\gamma_2},\end{gather}and \begin{gather}
 F_{\gamma_2}-\widetilde F_{\gamma_2}
 =\sum_{Q\in\mathcal Q_h}m_{Q,\gamma_2}|Q|^{-1}\one_Q.\notag
\end{gather}
Then
\begin{equation}\label{eq:mode-mass}
 \int_Q\widetilde f_{Q,\gamma_2}=0,\qquad
 \norm{\widetilde f_{Q,\gamma_2}}_1\le2\norm{b_Q}_1\le C\tau|Q|.
\end{equation}
Thus the corrected functions have ordinary mean zero and satisfy the
same support and mass conditions up to a fixed constant.
It remains to estimate the contribution of the subtracted averages.

\subsection{Radial estimates for the convolution kernels and the mean-correction error}\label{subsec:radial}
We now have a convolution representation with corrected bad functions.
Two facts are needed: the kernel must satisfy the radial derivative
bounds in Seeger's lemma (\cite[Lemma 2.1-2.2]{Seeger1996}), and the discarded averages must produce a
summable $L^1$ error. For fixed $s$, define the convolution kernel
without the phase factor at each annular scale by
\begin{equation}\label{eq:new-reg}
 \HH_{j,s}(x)=\Omega_s(x)k_j(x)
 =\Omega_s(x)\frac{\phi(2^{-j}|x|)}{|x|^n}.
\end{equation}
It satisfies
\begin{equation}\label{eq:reg-bounds}
 \supp\HH_{j,s}\subset\{2^{j-1}\le|x|\le2^{j+1}\},
 \qquad\norm{\HH_{j,s}}_1\le C 2^{\eta_0s}\norm\Omega_1.
\end{equation}
For fixed $\theta$ and any nonnegative integer $a$, the product rule gives
\[
 \partial_r^a\HH_{j,s}(r\theta)
 =\Omega_s(\theta)\sum_{b=0}^a\binom ab
 2^{-jb}\phi^{(b)}(2^{-j}r)
 \frac{d^{a-b}}{dr^{a-b}}r^{-n}.
\]
Every term on the right-hand side is supported where $r\simeq2^j$. Hence
\[
 r^{n+a}|\partial_r^a\HH_{j,s}(r\theta)|
 \le C|\Omega_s(\theta)|.
\]
The remainder estimate in Section \ref{sec:microlocal} uses finitely many
radial derivatives. We choose enough derivatives to make its second
error term decay, namely
\begin{equation}\label{eq:NS}
 N_S:=5(n+1).
\end{equation}
Then
\begin{equation}\label{eq:reg-radial}
 \sup_{0\le a\le N_S}\sup_j
 \operatorname*{ess\,sup}_{\theta\in\Sph^{n-1}}\sup_{r>0}
 r^{n+a}|\partial_r^a\HH_{j,s}(r\theta)|\le C2^{\eta_0s}\norm\Omega_1.
\end{equation}
The constant is uniform in $j$. Radial differentiation leaves $\theta$
fixed and does not differentiate $\Omega_s(\theta)$. Separating the phase
before applying this estimate is essential: these kernel bounds contain
no phase coefficients. The calculation explaining the choice of $N_S$
is given in \eqref{eq:seeger-exp-calc}.

By \eqref{eq:mode-correct}, each Fourier mode admits the exact decomposition
\[
 \HH_{j,s}*F_{\gamma_2}
 =\HH_{j,s}*\widetilde F_{\gamma_2}
 +\HH_{j,s}*(F_{\gamma_2}-\widetilde F_{\gamma_2}).
\]
In a region with $\mathfrak M_h\le2^{\sigma s}$, retain the corrected
convolution terms in $U_h$ and collect the subtracted averages in
$R_h^{\mathrm{err}}$. Thus define
\begin{align}\label{eq:Uh}
 U_h(x)&=\chi_h^{x}(x)
 e^{iP(x_h,x_h)+iA_h(x)}
 \sum_{\gamma}c_{\gamma}^he^{i\pi\gamma_1\cdot(x-x_h)/(L_0R)}
 (\HH_{j,s}*\widetilde F_{\gamma_2})(x),\end{align} and \begin{align}
 R_h^{\mathrm{err}}(x)&=\chi_h^{x}(x)
 e^{iP(x_h,x_h)+iA_h(x)}
 \sum_{\gamma}c_{\gamma}^he^{i\pi\gamma_1\cdot(x-x_h)/(L_0R)}\notag\\
 &\quad\times\bigl[\HH_{j,s}*(F_{\gamma_2}-\widetilde F_{\gamma_2})\bigr](x).\label{eq:Rh}
\end{align}
By \eqref{eq:mixed-fourier} and the preceding identity, a term-by-term calculation gives
\begin{equation}\label{eq:local-reduce}
 T_{j,\Omega_s}^{(h)}b_h=U_h+R_h^{\mathrm{err}}.
\end{equation}

\medskip
\noindent\textbf{The mean-correction error.}
We estimate the subtracted averages by Young's inequality.
The factor $2^{-s}$ in \eqref{eq:mode-smallmean} gives decay after
summation with the weighted Fourier coefficient bound
\eqref{eq:fourier-weight}. Indeed,
\[
 \norm{F_{\gamma_2}-\widetilde F_{\gamma_2}}_1
 \le\sum_{Q\in\mathcal Q_h}|m_{Q,\gamma_2}|
 \le C2^{-s}(|\gamma_2|+\mathfrak M_h)\sum_{Q\in\mathcal Q_h}\norm{b_Q}_1.
\]
Consequently, Young's inequality, \eqref{eq:reg-bounds}, and \eqref{eq:fourier-weight} give
\begin{align}\label{eq:reg-error}
 \norm{R_h^{\mathrm{err}}}_1
 &\le C2^{\eta_0s}\norm\Omega_1\,2^{-s}\sum_{Q\in\mathcal Q_h}\norm{b_Q}_1
 \sum_{\gamma}|c_{\gamma}^h|(|\gamma_2|+\mathfrak M_h)\notag\\
 &\le C2^{\eta_0s}\norm\Omega_1\,2^{-(1-N_0\sigma)s}\sum_{Q\in\mathcal Q_h}\norm{b_Q}_1
 \notag\\
 &=C\norm\Omega_1\,2^{-(\kappa_2-\eta_0)s}\sum_{Q\in\mathcal Q_h}\norm{b_Q}_1,
 \qquad\kappa_2=1-N_0\sigma>0.
\end{align}
By \eqref{eq:fourier-weight} and Young's inequality, the preceding Fourier series converge absolutely in $L^1$. In the large-phase region, define
\begin{equation}\label{eq:large-UR-definition}
 U_h=T_{j,\Omega_s}^{(h)}b_h,\qquad R_h^{\mathrm{err}}=0
 \qquad(\mathfrak M_h>2^{\sigma s}).
\end{equation}

At this point the mean-correction error is controlled. Define the global
corrected term and its error by
\begin{equation}\label{eq:globalUR}
 \mathcal U_{j,s}=\sum_h\one_{I_h}U_h,\qquad
 \mathcal R_{j,s}=\sum_h\one_{I_h}R_h^{\mathrm{err}}.
\end{equation}
The exact identity is
\begin{align} \label{ch1} T_{j,\Omega_s}B_{j-s}=\mathcal U_{j,s}+\mathcal R_{j,s}.\end{align}
To finish Theorem \ref{prop:lowfull}, we still need to split
$\mathcal U_{j,s}$ into the required $L^2$ main term and $L^1$ remainder.

\subsection{Local main terms, remainders, and summation}\label{subsec:GE}
In the small-phase region, the remaining term $U_h$ is a sum of
convolutions with mean-zero bad functions. We apply Seeger's kernel
decomposition(\cite{Seeger1996}) to each Fourier mode: the main term satisfies an $L^2$
decay estimate, while the remainder satisfies an $L^1$ decay estimate
using mean-zero cancellation. For $H_j=\HH_{j,s}$, denote the main kernel
defined in \eqref{eq:seeger-Gamma} by $\Gamma_{j,\Omega_s}^s$, with
$\kappa=1/2$. The construction and
the detailed estimates will be given in Section \ref{sec:microlocal}. In the small-phase region, define
\begin{align}\label{eq:Gh}
 G_h(x)&=\chi_h^{x}(x)
 e^{iP(x_h,x_h)+iA_h(x)}
 \sum_{\gamma}c_{\gamma}^he^{i\pi\gamma_1\cdot(x-x_h)/(L_0R)}
 (\Gamma_{j,\Omega_s}^s*\widetilde F_{\gamma_2})(x)\end{align} and \begin{align}
 E_h(x)&=\chi_h^{x}(x)
 e^{iP(x_h,x_h)+iA_h(x)}
 \sum_{\gamma}c_{\gamma}^he^{i\pi\gamma_1\cdot(x-x_h)/(L_0R)}\notag\\
 &\quad\times\bigl[\HH_{j,s}*\widetilde F_{\gamma_2}-\Gamma_{j,\Omega_s}^s*\widetilde F_{\gamma_2}\bigr](x).\label{eq:Eh}
\end{align}
These definitions give $U_h=G_h+E_h$ exactly. In the large-phase
region, \eqref{eq:large-application} already controls $U_h$ in $L^2$,
so set $G_h=U_h$ and $E_h=0$. In both cases, let
\begin{equation}\label{eq:globalGE}
 G_{j,s}=\sum_h\one_{I_h}G_h,\qquad E_{j,s}=\sum_h\one_{I_h}E_h.
\end{equation}
By definition,
\begin{equation}\label{eq:U-GE}
 \mathcal U_{j,s}=G_{j,s}+E_{j,s}.
\end{equation} Thus by \eqref{ch1}, we get \begin{equation}\label{ch2}
 T_{j,\Omega_s}B_{j-s}=G_{j,s}+E_{j,s}+\mathcal R_{j,s}.
\end{equation}
Thus $\mathcal R_{j,s}$ is the mean-correction error, and $E_{j,s}$ is
the remainder from Seeger's kernel decomposition.
The main term $G_{j,s}$ contains the convolutions with
$\Gamma_{j,\Omega_s}^s$ in the small-phase region and the entire
localized contribution in the large-phase region.

\begin{proof}[Proof of Theorem \ref{prop:lowfull}] To prove Theorem \ref{prop:lowfull}, we verify the decomposition and estimates in \eqref{eq:prop-reduction}--\eqref{eq:prop-reduction-error}.

\noindent{ \emph{The exact decomposition.}}
For every $h$, the preceding construction gives
\[
 T_{j,\Omega_s}^{(h)}b_h
 =U_h+R_h^{\mathrm{err}}
 =G_h+E_h+R_h^{\mathrm{err}}.
\]
Restricting to $I_h$, using \eqref{eq:parent-identity}, and summing over
the pairwise disjoint cubes $I_h$ proves \eqref{eq:prop-reduction}.

\medskip
\noindent{ \emph{The mean-correction error}.}
For large mixed coefficients this error is zero. For small mixed
coefficients use \eqref{eq:reg-error}. The bounded overlap in
\eqref{eq:parent-overlap} gives
\begin{align*}
 \norm{\mathcal R_{j,s}}_1
 &\le\sum_h\norm{\one_{I_h}R_h^{\mathrm{err}}}_1\\
 &\le C\norm\Omega_1\,2^{-(1-N_0\sigma-\eta_0)s}
       \sum_h\sum_{Q\in\mathcal Q_h}\norm{b_Q}_1\\
 &\le C\norm\Omega_1\,2^{-(\kappa_2-\eta_0)s}
       \norm{B_{j-s}}_1,
 \qquad \kappa_2=1-N_0\sigma>0.
\end{align*}
This proves \eqref{eq:prop-reduction-error}.

\medskip
\noindent{\emph{The main term and the projection remainder.}}
In the small-phase region, we claim that \begin{equation*}
 \norm{G_h}_2^2
 \le C\tau 2^{2\eta_0s}\norm\Omega_1^2
 2^{-(1/2-2N_0\sigma)s}\sum_{Q\in\mathcal Q_h}\norm{b_Q}_1.
\end{equation*} and \begin{align*}
 \norm{E_h}_1
 &\le  C2^{\eta_0s}\norm\Omega_1
 2^{-(1/4-N_0\sigma)s}\sum_{Q\in\mathcal Q_h}\norm{b_Q}_1,
 \end{align*} whose proofs will be given in Section \ref{sec:microlocal}, following directly from Seeger’s lemma\cite[Lemmas 2.1--2.2]{Seeger1996} and the weighted Fourier bound \eqref{eq:fourier-weight}, via \eqref{eq:Gh-seeger-square} and \eqref{eq:Eh-seeger}.
In the large-phase region, use \eqref{eq:large-application}
and $E_h=0$. Thus the exponents can be chosen as
\[
 \kappa_3=\min\left\{\delta_0,\frac12-2N_0\sigma\right\},
 \qquad \kappa_4=\frac14-N_0\sigma.
\]
The disjointness of the cubes $I_h$ permits summation of the squared
$L^2$ norms of the restricted main terms; the $L^1$ remainders are
summed by the triangle inequality. We claim that \begin{align*}
 \norm{G_{j,s}}_2^2
 &\le C\tau 2^{2\eta_0s}\norm\Omega_1^2\,2^{-\kappa_3s}
\norm{B_{j-s}}_1.
\end{align*} and \begin{align*}
 \norm{E_{j,s}}_1
&\le C2^{\eta_0s}\norm\Omega_1\,2^{-\kappa_4s}\norm{B_{j-s}}_1,
 \end{align*} whose proofs are given in detail in Section \ref{sec:microlocal} via \eqref{eq:prop-main-seeger} and \eqref{eq:prop-rem-seeger}.
They are exactly \eqref{eq:prop-main} and \eqref{eq:prop-rem}.
All three decay exponents are positive by \eqref{eq:sigma} and
Lemma \ref{lem:large}; the subsequent choice \eqref{eq:eta} preserves
decay after the factors from $\Omega_s$ are included. This completes
the proof of Theorem 3.2.
\end{proof}

\section{Seeger's microlocal estimates for the low positive scales}\label{sec:microlocal}
Section \ref{sec:reduction} has constructed the exact decomposition in
Theorem \ref{prop:lowfull} and estimated the mean-correction error
$\mathcal R_{j,s}$. The large-phase region has also been treated
by Lemma \ref{lem:large}. The remaining estimates used in that proof are
the $L^2$ bound for $G_{j,s}$ and the $L^1$ bound for $E_{j,s}$ in the
small-phase region. We establish them here from Seeger's lemma.

For each fixed $j,h$, the phase has already been separated into
unimodular factors and Fourier coefficients, and the functions
$\widetilde f_{Q,\gamma_2}$ satisfy ordinary mean-zero cancellation.
Thus the problem is now to estimate convolutions with
$\HH_{j,s}$. We first describe the directional projection of this
kernel, then check the hypotheses of Seeger's lemma, and finally sum
over the Fourier frequencies and the local regions.

\subsection{The kernel decomposition and the required decay}
The decomposition separates a kernel with an $L^2$ estimate based on
cube mass bounds from a remainder with an $L^1$ estimate based on
mean-zero cancellation. This is why the two parts of $U_h$ were assigned
different norms in Section \ref{subsec:GE}.
Fix $s>3$. Choose a finite set of directions on $\Sph^{n-1}$,
\(
\{e_v^s\}_v\subset\Sph^{n-1},
\)
such that
\(
|e_v^s-e_{v'}^s|>2^{-s-10}n^{-1},
v\ne v',
\)
and, for every $\theta\in\Sph^{n-1}$, there exists $v$ such that
\(
|\theta-e_v^s|\le 2^{-s-1}.
\)
Choose a corresponding pairwise disjoint measurable partition $\{E_v^s\}_v$ satisfying
\begin{equation}\label{eq:seeger-caps}
 \Sph^{n-1}=\bigcup_vE_v^s,\qquad
 E_v^s\subset\{\theta:|\theta-e_v^s|\le C2^{-s}\},\qquad
 \sigma(E_v^s)\le C2^{-(n-1)s}.
\end{equation}
The width $2^{-s}$ reflects the ratio of the bad-cube sidelength to the
annular scale. We first state the kernel construction for a general
family $H_j$ and then substitute $H_j=\HH_{j,s}$.
Let $H_j$ be a measurable function supported in the annulus
\begin{equation}\label{eq:seeger-support}
 \supp H_j\subset\{2^{j-2}\le |x|\le2^{j+2}\}
\end{equation}
where the constants are fixed, and, for each integer $N\ge0$, define the radial regularity constant
\begin{equation}\label{eq:seeger-MN}
 M_N(H):=
 \sup_{0\le a\le N}\sup_j
 \operatorname*{ess\,sup}_{\theta\in\Sph^{n-1}}
 \sup_{r>0}
 r^{n+a}\left|\left(\frac{\partial}{\partial r}\right)^aH_j(r\theta)\right|.
\end{equation}
For each $v$, set
\begin{equation}\label{eq:seeger-Hjv}
 H_{j,v}^s(x)=H_j(x)\one_{E_v^s}\!\left(\frac{x}{|x|}\right).
\end{equation}
Choose $\psi\in C_c^\infty(\R)$ supported in $[-4,4]$
with $\psi=1$ on $[-2,2]$.
Fix $0<\kappa<1$ and define the directional projection by its Fourier
multiplier
\begin{equation}\label{eq:seeger-projection}
	\widehat{P_v^s}(\xi)
	=\psi\!\left(2^{s(1-\kappa)}
	\frac{\langle\xi,e_v^s\rangle}{|\xi|}\right),
	\qquad \xi\ne0.
\end{equation}
The multiplier selects frequencies nearly perpendicular to $e_v^s$.
These frequencies are retained in the main kernel; the complementary
part is estimated using the radial regularity of $H_j$ and the
cancellation of the bad function. Define
\begin{equation}\label{eq:seeger-Gamma}
	\Gamma_j^s
	=\sum_v P_v^s*H_{j,v}^s.
\end{equation}
This gives the decomposition of $H_j$
\begin{equation}\label{eq:seeger-decomp}
	H_j=\Gamma_j^s+(H_j-\Gamma_j^s).
\end{equation}

\begin{lemma}[Seeger~{\cite[Lemmas 2.1--2.2]{Seeger1996}}]\label{lem:seeger}
With the notation of \eqref{eq:seeger-support}--\eqref{eq:seeger-decomp}, the following two estimates hold.

\medskip
\noindent
\textup{(i)} Let $\mathcal Q$ be a family of dyadic cubes with pairwise disjoint interiors, and set
\[
 \mathcal Q_m=\{Q\in\mathcal Q:\ell(Q)=2^m\},\qquad
 F_m=\sum_{Q\in\mathcal Q_m}f_Q.
\]
If
\begin{equation}\label{eq:seeger-input-main}
 \supp f_Q\subset Q,\qquad
 \norm{f_Q}_1\le\alpha|Q|,
\end{equation}
then, for $s>3$,
\begin{equation}\label{eq:seeger-L2}
 \norm{\sum_j\Gamma_j^s*F_{j-s}}_2^2
 \le CM_0(H)^2\,2^{-s(1-\kappa)}
 \alpha\sum_{Q\in\mathcal Q}\norm{f_Q}_1.
\end{equation}

\medskip
\noindent
\textup{(ii)} Let $Q$ have sidelength $2^{j-s}$, and let $b_Q$ be supported in $Q$ with
\begin{equation}\label{eq:seeger-input-rem}
 \int_Qb_Q=0.
\end{equation}
If $N\ge n+1$ and $0\le\varepsilon\le1$, then
\begin{equation}\label{eq:seeger-L1-general}
 \norm{(H_j-\Gamma_j^s)*b_Q}_1
 \le C
 \left[
 M_0(H)2^{-\varepsilon s}
 +M_N(H)2^{[\,n+(\varepsilon-\kappa)N\,]s}
 \right]\norm{b_Q}_1.
\end{equation}
\end{lemma}

\smallskip\smallskip

The two parts of this lemma have different hypotheses. Part \textup{(i)}
uses the support and mass bounds on the cubes; part \textup{(ii)} also
needs $\int_Qb_Q=0$. In particular, the mean correction in
Section \ref{subsec:mean-correction} is needed for the remainder estimate.

We choose the parameters so that both terms on the right-hand side of
\eqref{eq:seeger-L1-general} decay in $s$, while
\eqref{eq:seeger-L2} retains positive decay as well. A convenient fixed
choice is
\begin{equation}\label{eq:seeger-parameters}
 \kappa=\frac12,\qquad
 \varepsilon=\frac14,\qquad
 N=N_S:=5(n+1).
\end{equation}
Then
\begin{equation}\label{eq:seeger-exp-calc}
 1-\kappa=\frac12,
 \qquad
 n+(\varepsilon-\kappa)N_S
 =n-\frac{5(n+1)}4
 =-\frac{n+5}{4}.
\end{equation}
The first remainder term therefore decays like $2^{-s/4}$, and the
second decays faster. This explains why the derivative order $N_S$
was chosen in \eqref{eq:NS}; no optimization of these exponents is needed.
Consequently, \eqref{eq:seeger-L2} becomes
\begin{equation}\label{eq:seeger-L2-fixed}
 \norm{\sum_j\Gamma_j^s*F_{j-s}}_2^2
 \le CM_0(H)^2\,2^{-s/2}
 \alpha\sum_{Q\in\mathcal Q}\norm{f_Q}_1,
\end{equation}
and \eqref{eq:seeger-L1-general} gives
\begin{align}
 \norm{(H_j-\Gamma_j^s)*b_Q}_1
 &\le C\left[M_0(H)2^{-s/4}
 +M_{N_S}(H)2^{-(n+5)s/4}\right]\norm{b_Q}_1\notag\\
 &\le C\bigl(M_0(H)+M_{N_S}(H)\bigr)
 2^{-s/4}\norm{b_Q}_1.
 \label{eq:seeger-L1-fixed}
\end{align}

\subsection{Verification of Seeger's hypotheses}\label{subsec:seeger-verification}

Before applying these bounds, we verify separately the kernel conditions
and the conditions on the corrected bad functions. This also identifies
which constants depend on the truncation of $\Omega$ and which are
uniform in the Fourier frequency.

\medskip
\noindent{\emph{The convolution kernels}.}
Fix the scale separation $s\ge s_*$ and set
\begin{equation}\label{eq:seeger-our-kernel}
 H_j=\HH_{j,s}.
\end{equation}
By \eqref{eq:reg-bounds},
\begin{equation}\label{eq:our-seeger-support}
 \supp\HH_{j,s}\subset\{2^{j-1}\le|x|\le2^{j+1}\},
\end{equation}
so Seeger's support hypothesis \eqref{eq:seeger-support} is satisfied.

By \eqref{eq:reg-radial} and \eqref{eq:NS},
\begin{equation}\label{eq:our-MN}
 M_0(H)+M_{N_S}(H)\le C2^{\eta_0s}\norm\Omega_1.
\end{equation}
Apart from the displayed factors $2^{\eta_0s}\norm\Omega_1$, the bound
is uniform in $j,s$ and the frequency index $\gamma$.
Only radial derivatives are required, exactly as verified in
\eqref{eq:reg-radial}.

\medskip
\noindent{\emph{The corrected bad functions}.}
Fix $j,h$ and a Fourier frequency $\gamma_2\in\Z^n$ in the $y$ variable.
By \eqref{eq:mode-correct}--\eqref{eq:mode-mass}, for each $Q\in\mathcal Q_h$,
\begin{equation}\label{eq:our-seeger-input}
 \supp\widetilde f_{Q,\gamma_2}\subset Q,
 \qquad
 \int_Q\widetilde f_{Q,\gamma_2}=0,
 \qquad
 \norm{\widetilde f_{Q,\gamma_2}}_1\le C\tau|Q|,
\end{equation}
and
\begin{equation}\label{eq:our-seeger-mass}
 \sum_{Q\in\mathcal Q_h}\norm{\widetilde f_{Q,\gamma_2}}_1
 \le2\sum_{Q\in\mathcal Q_h}\norm{b_Q}_1.
\end{equation}
Moreover, each $Q\in\mathcal Q_h$ has sidelength exactly
\begin{equation}\label{eq:our-seeger-scale}
 \ell(Q)=2^{j-s}.
\end{equation}
Thus the kernel scale, the cube scale, and the support and mass bounds
match the hypotheses of part \textup{(i)}. The ordinary mean-zero
condition additionally permits part \textup{(ii)}. Although the corrected
functions depend on the phase and the Fourier frequency, their support,
cancellation, and mass bounds are uniform. We then use
\eqref{eq:fourier-weight} to sum the Fourier coefficients.

\subsection{The microlocal main term}

Fix $s\ge s_*$, $1\le j\le\varrho s$, and a local region $h$
with $\mathfrak M_h\le2^{\sigma s}$.
For each Fourier frequency $\gamma_2\in\Z^n$, apply
Lemma~\ref{lem:seeger}\textup{(i)} with
$H_\ell=\HH_{\ell,s}$, the cube family $\mathcal Q_h$,
$f_Q=\widetilde f_{Q,\gamma_2}$, and $\alpha=C\tau$.
The conditions follow from
\eqref{eq:our-seeger-input}--\eqref{eq:our-seeger-scale}.
Since all these cubes have sidelength $2^{j-s}$,
only the term $\ell=j$ remains in
\eqref{eq:seeger-L2-fixed}.
Thus \eqref{eq:our-MN} and \eqref{eq:our-seeger-mass}
give
\begin{equation}\label{eq:mode-G-seeger}
	\norm{\Gamma_{j,\Omega_s}^s*
		\widetilde F_{\gamma_2}}_2^2
	\le C\tau 2^{2\eta_0s}\norm\Omega_1^2
	2^{-s/2}
	\sum_{Q\in\mathcal Q_h}\norm{b_Q}_1.
\end{equation}
Taking square roots yields
\begin{equation}\label{eq:mode-G-seeger-root}
	\norm{\Gamma_{j,\Omega_s}^s*
		\widetilde F_{\gamma_2}}_2
	\le C\sqrt{\tau
		\sum_{Q\in\mathcal Q_h}\norm{b_Q}_1}\,
	2^{\eta_0s}\norm\Omega_1\,2^{-s/4}.
\end{equation}

In \eqref{eq:Gh}, the cutoff $\chi_h^x$ is bounded
by $1$, and all the remaining factors depending
on $x$ have modulus $1$,$\left|e^{iP(x_h,x_h)+iA_h(x)}e^{i\pi\gamma_1\cdot(x-x_h)/(L_0R)}\right|=1$.
Hence Minkowski's inequality,
\eqref{eq:mode-G-seeger-root}, and
\eqref{eq:fourier-weight} give
\begin{align}
	\norm{G_h}_2
	&\le \sum_{\gamma}|c_{\gamma}^h|
	\norm{\Gamma_{j,\Omega_s}^s*
		\widetilde F_{\gamma_2}}_2\notag\\
	&\le C\sqrt{\tau
		\sum_{Q\in\mathcal Q_h}\norm{b_Q}_1}\,
	2^{\eta_0s}\norm\Omega_1
	\,2^{-s/4}2^{N_0\sigma s}\notag\\
	&=C\sqrt{\tau
		\sum_{Q\in\mathcal Q_h}\norm{b_Q}_1}\,
	2^{\eta_0s}\norm\Omega_1
	\,2^{-(1/4-N_0\sigma)s}.
	\label{eq:Gh-seeger}
\end{align}
Consequently,
\begin{equation}\label{eq:Gh-seeger-square}
	\norm{G_h}_2^2
	\le C\tau 2^{2\eta_0s}\norm\Omega_1^2
	2^{-(1/2-2N_0\sigma)s}
	\sum_{Q\in\mathcal Q_h}\norm{b_Q}_1.
\end{equation}

For $\mathfrak M_h>2^{\sigma s}$, the definitions
give $G_h=U_h$, and
Lemma~\ref{lem:large}, through
\eqref{eq:large-application}, yields
\begin{equation}\label{eq:Gh-large-repeat}
	\norm{G_h}_2^2
	\le C\tau 2^{2\eta_0s}\norm\Omega_1^2
	2^{-\delta_0s}
	\sum_{Q\in\mathcal Q_h}\norm{b_Q}_1.
\end{equation}
Thus, setting
\begin{equation}\label{eq:kappaG}
	\kappa_3=\min\left\{
	\delta_0,\frac12-2N_0\sigma
	\right\}>0,
\end{equation}
we obtain a uniform estimate for every $h$.

Since the cubes $I_h$ are pairwise disjoint,
the restrictions $\one_{I_h}G_h$ have disjoint
supports. Summing their squared $L^2$ norms
and using \eqref{eq:parent-overlap}, we obtain
\begin{align}
	\norm{G_{j,s}}_2^2
	&=\sum_h\norm{\one_{I_h}G_h}_2^2\notag\\
	&\le C\tau 2^{2\eta_0s}\norm\Omega_1^2
	2^{-\kappa_3s}
	\sum_h\sum_{Q\in\mathcal Q_h}\norm{b_Q}_1
	\notag\\
	&\le C\tau 2^{2\eta_0s}\norm\Omega_1^2
	2^{-\kappa_3s}\norm{B_{j-s}}_1.
	\label{eq:prop-main-seeger}
\end{align}
This proves \eqref{eq:prop-main}.

\subsection{The microlocal remainder and choice of parameters}

For the remainder, we use the ordinary mean-zero
condition of the corrected bad functions.
When $\mathfrak M_h\le2^{\sigma s}$, each
$Q\in\mathcal Q_h$ satisfies
\[
\int_Q\widetilde f_{Q,\gamma_2}=0
\]
by \eqref{eq:our-seeger-input}.
Applying Lemma~\ref{lem:seeger}\textup{(ii)}
through \eqref{eq:seeger-L1-fixed}
and \eqref{eq:our-MN}, we obtain
\begin{equation}\label{eq:single-rem-seeger}
	\norm{(\HH_{j,s}-\Gamma_{j,\Omega_s}^s)*
		\widetilde f_{Q,\gamma_2}}_1
	\le C2^{\eta_0s}\norm\Omega_1
	\,2^{-s/4}\norm{\widetilde f_{Q,\gamma_2}}_1.
\end{equation}
Summing over $Q\in\mathcal Q_h$ and using
\eqref{eq:our-seeger-mass} gives
\begin{equation}\label{eq:mode-rem-seeger}
	\begin{aligned}
		&\norm{\HH_{j,s}*\widetilde F_{\gamma_2}
			-\Gamma_{j,\Omega_s}^s*
			\widetilde F_{\gamma_2}}_1\\
		&\qquad\le C2^{\eta_0s}\norm\Omega_1
		\,2^{-s/4}
		\sum_{Q\in\mathcal Q_h}\norm{b_Q}_1.
	\end{aligned}
\end{equation}

The cutoff and unimodular factors in
\eqref{eq:Eh} do not increase the $L^1$ norm.
By the triangle inequality and
\eqref{eq:fourier-weight},
\begin{align}
	\norm{E_h}_1
	&\le C2^{\eta_0s}\norm\Omega_1\,2^{-s/4}
	\sum_{Q\in\mathcal Q_h}\norm{b_Q}_1
	\sum_{\gamma}|c_{\gamma}^h|\notag\\
	&\le C2^{\eta_0s}\norm\Omega_1
	2^{-(1/4-N_0\sigma)s}
	\sum_{Q\in\mathcal Q_h}\norm{b_Q}_1.
	\label{eq:Eh-seeger}
\end{align}
For $\mathfrak M_h>2^{\sigma s}$, we have
$E_h=0$ by definition.
Set
\begin{equation}\label{eq:kappaE}
	\kappa_4=\frac14-N_0\sigma>0.
\end{equation}
Summing over $h$ and using
\eqref{eq:parent-overlap}, we obtain
\begin{align}
	\norm{E_{j,s}}_1
	&\le\sum_h\norm{\one_{I_h}E_h}_1\notag\\
	&\le C2^{\eta_0s}\norm\Omega_1
	2^{-\kappa_4s}
	\sum_h\sum_{Q\in\mathcal Q_h}\norm{b_Q}_1
	\notag\\
	&\le C2^{\eta_0s}\norm\Omega_1
	2^{-\kappa_4s}\norm{B_{j-s}}_1.
	\label{eq:prop-rem-seeger}
\end{align}
This proves \eqref{eq:prop-rem}.

It remains to check the decay exponents.
The mean-correction estimate in
Section~\ref{sec:reduction} has exponent
\begin{equation}\label{eq:kappaR-final}
	\kappa_2=1-N_0\sigma>0.
\end{equation}
The choice \eqref{eq:sigma} also gives
\begin{equation}\label{eq:low-kappas-final}
	\frac12-2N_0\sigma>0,
	\qquad
	\frac14-N_0\sigma>0.
\end{equation}
Moreover, Lemma~\ref{lem:large} gives
\[
\delta_0=\frac{c\sigma}{4(N+c)D}>0,
\]
where $c=1/(258D)$ and $N=16$ are the constants
in Lemma~\ref{lem:twocenter}.
Thus $\kappa_3$ and $\kappa_4$ are positive.

Choosing $\eta_0$ according to \eqref{eq:eta}
absorbs the growth of $\Omega_s$, and a uniform
$s_*(n,D)>3$ may be taken sufficiently large
for the fixed support and localization conditions.
Together with the mean-correction estimate from
Section~\ref{sec:reduction},
\eqref{eq:prop-main-seeger} and
\eqref{eq:prop-rem-seeger} establish the bounds
required in Theorem~\ref{prop:lowfull}
and complete the low positive scale summation
in \eqref{eq:low-main-total}--\eqref{eq:R-total}.

\section{Weak type \texorpdfstring{$(1,1)$}{(1,1)} estimates for the rough  maximal oscillatory operator}\label{sec:maximal}

We retain the notation of Sections~\ref{sec:intro}--\ref{sec:microlocal},
in particular the smooth operators $T_j$ defined by
\eqref{eq:partition}--\eqref{eq:Tj}, the quantity
$A=1+\|\Omega\|_{L\log^+L}$, and the parameters chosen in
Section~\ref{sec:main-proof}. Write $K(z)=\Omega(z)/|z|^n$.

\subsection{What remains to be proved}\label{subsec:max-roadmap}

Theorem~\ref{thm:main} controls the principal value operator.  The maximal operator in Theorem~\ref{thm1} asks for a bound uniform in the truncation radius, so one additional issue appears: after the dyadic decomposition we must control the maximal tails
\[
\mathcal M_P^+f(x)=\sup_{J\ge1}\left|\sum_{j\ge J}T_jf(x)\right|.
\]
The purpose of this section is to show that the estimates already proved for the positive scales are stable under this supremum.

The proof has four parts.  First, the local truncations $0<\varepsilon<1$ are treated by the same induction on the largest mixed bidegree $(m,l)$ used in Section~\ref{subsec:local}; the fixed far piece is then reduced to Theorem~\ref{thm:main}.  Second, a general truncation radius is compared with a dyadic tail plus the rough maximal operator $M_\Omega$.  Third, for the Calder\'on--Zygmund good function we prove a summable single-annulus $L^2$ estimate, which immediately controls the maximal tail.  Finally, for the bad function the low positive scales are controlled by the decomposition of Theorem~\ref{prop:lowfull}, while the high positive scales require one new step: the signed $L^2$ estimate implicit in Theorem~\ref{prop:high} is combined with the Rademacher--Menshov inequality to control all partial tails simultaneously.  Once these four estimates are established, the proof is completed by the same $L^1/L^2$ summation used in Section~\ref{sec:main-proof}.

\smallskip

\subsection{Reduction to local truncations and positive dyadic tails}\label{subsec:max-reduction}

The conclusion is trivial when $\Omega=0$. Henceforth assume that $\|\Omega\|_1>0$.  As in Section~\ref{subsec:local}, we induct in lexicographic order on the largest mixed bidegree $(m,l)$.

Let $T_\Omega^*$ denote the  maximal truncated
rough singular integral operator, defined by
\[
T_\Omega^*g(x)
:=\sup_{\varepsilon>0}
\left|
\int_{|x-y|>\varepsilon}
\frac{\Omega(x-y)}{|x-y|^n}g(y)\,dy
\right|.
\]For the induction basis, if $P$ has no mixed terms, write $P(x,y)=a(x)+c(y)$. Then
\[
T_P^*f(x)=T_\Omega^*(e^{ic}f)(x),
\]
since the factor $e^{ia(x)}$ has modulus $1$. The weak type $(1,1)$ estimate for $T_\Omega^*$ is due to Lai~\cite{Lai2025Maximal}, hence the desired estimate follows in this case.

Now assume $m,l\ge1$. After removing the pure phases, apply the dilation from Section~\ref{subsec:local}.  The truncation radius $\varepsilon$ is replaced by $B\varepsilon$, whose range is still $(0,\infty)$.  We may therefore assume
\begin{equation}\label{eq:max-normal}
\sum_{|\alpha|=m,\,|\beta|=l}|a_{\alpha\beta}|=1.
\end{equation}
The induction hypothesis is that Theorem~\ref{thm1} holds whenever, after removal of the pure phases, the largest mixed bidegree is strictly smaller than $(m,l)$.

\Needspace{15\baselineskip}
For a general phase, splitting the truncation radius into two ranges gives
\[
\begin{aligned}
	T_P^*f(x)
	&\le \sup_{0<\varepsilon<1}
	\left|\int_{|x-y|>\varepsilon}e^{iP(x,y)}K(x-y)f(y)\,dy\right|\\
	&\quad+\sup_{\varepsilon\ge1}
	\left|\int_{|x-y|>\varepsilon}e^{iP(x,y)}K(x-y)f(y)\,dy\right|\\
	&\le \sup_{0<\varepsilon<1}
	\left|\int_{\varepsilon<|x-y|<1}e^{iP(x,y)}K(x-y)f(y)\,dy\right|\\
	&\quad+\left|\int_{|x-y|\ge1}e^{iP(x,y)}K(x-y)f(y)\,dy\right|\\
	&\quad+\sup_{\varepsilon\ge1}
	\left|\int_{|x-y|>\varepsilon}e^{iP(x,y)}K(x-y)f(y)\,dy\right|.
\end{aligned}
\]
Denote the first and third terms in the last expression by
$T_{P,0}^*f(x)$ and $T_{P,\infty}^*f(x)$, respectively, and denote
the integral in the middle term by $F_Pf(x)$. Thus
\begin{equation}\label{eq:max-three}
	T_P^*f(x)\le T_{P,0}^*f(x)+|F_Pf(x)|+T_{P,\infty}^*f(x).
\end{equation}

\subsection{The local part $ T_{P,0}^*$ and the fixed far part $F_P$}\label{subsec:max-local}

Use the covering and localized functions from Section~\ref{subsec:local}:
\[
\mathbb R^n=\bigcup_\nu B(h_\nu,1/4),\qquad
\sum_\nu\one_{B(h_\nu,5/4)}\le C,\qquad
f_\nu=f\one_{B(h_\nu,5/4)}.
\]
If $x\in B(h_\nu,1/4)$ and $|x-y|<1$, then
$y\in B(h_\nu,5/4)$, and hence
$T_{P,0}^*f(x)=T_{P,0}^*f_\nu(x)$.
Fix $\nu$, write $h=h_\nu$, and set
\[
\begin{gathered}
	P_{m,l}(u,v)=\sum_{|\alpha|=m,\,|\beta|=l}a_{\alpha\beta}u^\alpha v^\beta,
	\qquad R_h(x,y)=P(x,y)-P_{m,l}(x-h,y-h),\\
	\Phi_h(y)=P_{m,l}(y-h,y-h),\qquad u_\nu=e^{i\Phi_h}f_\nu.
\end{gathered}
\]
By \eqref{eq:phase-expand-h}--\eqref{eq:phase-difference}, the mixed
bidegree of $R_h$ strictly decreases after removal of the pure phases, and
\[
\left|e^{iP(x,y)}-e^{iR_h(x,y)}e^{i\Phi_h(y)}\right|
\le C|x-y|
\quad\bigl(x\in B(h,1/4),\ |x-y|<1\bigr).
\]
This bound is independent of $\varepsilon$. For every
$0<\varepsilon<1$, we also have
\[
\left|\int_{\varepsilon<|x-y|<1}e^{iR_h(x,y)}K(x-y)u_\nu(y)\,dy\right|
\le 2T_{R_h}^*u_\nu(x),
\]
since the integral is the difference of two integrals with truncation
radius $\varepsilon$ and $1$. Here $T_{R_h}^*$ is obtained from
$T_P^*$ by replacing $P$ with $R_h$. Thus, on $B(h_\nu,1/4)$,
\begin{equation}\label{eq:max-local-compare}
	\begin{aligned}
		T_{P,0}^*f(x)&\le 2T_{R_h}^*u_\nu(x)+H_\nu(x),\\
		H_\nu(x)&=C\int_{|x-y|<1}
		\frac{|\Omega(x-y)|}{|x-y|^{n-1}}|f_\nu(y)|\,dy.
	\end{aligned}
\end{equation}
The error estimate following \eqref{eq:phase-difference} in
Section~\ref{subsec:local} gives
$\|H_\nu\|_1\le C\|\Omega\|_1\|f_\nu\|_1$.
The induction hypothesis, the identity
$\|u_\nu\|_1=\|f_\nu\|_1$, and Chebyshev's inequality therefore yield
\[
|\{x\in B(h_\nu,1/4):T_{P,0}^*f(x)>\lambda\}|
\le \frac{CA}{\lambda}\|f_\nu\|_1.
\]
Summing over $\nu$ and using bounded overlap, we obtain
\begin{equation}\label{eq:max-local-weak}
	\lambda|\{x:T_{P,0}^*f(x)>\lambda\}|\le CA\|f\|_1.
\end{equation}
This completes the estimate for the local part $ T_{P,0}^*$.

For the fixed far part $F_P$, use the local principal value operator from
Section~\ref{subsec:local},
\[
S_Pf(x)=\operatorname{p.v.}\int_{|x-y|<1}e^{iP(x,y)}K(x-y)f(y)\,dy.
\]
The identity $F_Pf=T_Pf-S_Pf$, Theorem~\ref{thm:main}, and
Lemma~\ref{lem:localtrunc} with truncation radius $1$ give
\begin{equation}\label{eq:max-fixed-weak}
	\begin{aligned}
		\lambda|\{x:|F_Pf(x)|>\lambda\}|
		&\le\lambda|\{x:|T_Pf(x)|>\lambda/2\}|\\
		&\quad+\lambda|\{x:|S_Pf(x)|>\lambda/2\}|
		\le CA\|f\|_1.
	\end{aligned}
\end{equation}

\subsection{Dyadic decomposition of the far part $T_{P,\infty}^*$}\label{subsec:max-dyadic}

Given $\varepsilon\ge1$, choose the unique integer $J\ge1$ such that
$2^{J-1}\le\varepsilon<2^J$. By \eqref{eq:partition},
\[
\left|\one_{(\varepsilon,\infty)}(r)
-\sum_{j\ge J}\phi(2^{-j}r)\right|
\le\one_{[2^{J-1},2^J]}(r).
\]
Consequently,
\begin{align}
	T_{P,\infty}^*f(x)
	&\le\sup_{J\ge1}\int_{2^{J-1}\le|z|<2^J}
	\frac{|\Omega(z)|}{|z|^n}|f(x-z)|\,dz
	+\mathcal M_P^+f(x)\notag\\
	&\le C M_\Omega f(x)+\mathcal M_P^+f(x),\label{eq:max-dyadic}
\end{align}
where
\[
\begin{aligned}
	M_\Omega f(x)&=\sup_{r>0}r^{-n}
	\int_{|x-y|<r}|\Omega(x-y)|\,|f(y)|\,dy,\\
	\mathcal M_P^+f(x)&=\sup_{J\ge1}\left|\sum_{j\ge J}T_jf(x)\right|.
\end{aligned}
\]
The weak type estimate for the rough maximal operator
\cite{Christ1988,CR1988} gives
$\lambda|\{M_\Omega f>\lambda\}|\le CA\|f\|_1$.
It remains to estimate $\mathcal M_P^+$.

\subsection{Estimate for the good function}\label{subsec:max-good}

Fix $\lambda>0$ and set $\tau=\lambda/A$. Apply
Lemma~\ref{lem:adapted-CZ} with $\vartheta_Q(y)=P(c_Q,y)$,
exactly as in Section~\ref{sec:main-proof}, to write
$f=g+b$, where $b=\sum_{Q\in\mathcal Q}b_Q$.
The bounds \eqref{eq:CZ} and the cancellation \eqref{eq:cancel}
remain valid. We first estimate the good function. We have
\[
\|g\|_1\le \|f\|_1,
\qquad
\|g\|_\infty\le C\lambda/A,
\]
and therefore
\begin{equation}\label{eq:good-L2}
	\|g\|_2^2
	\le \|g\|_\infty\|g\|_1
	\le C\frac{\lambda}{A}\|f\|_1.
\end{equation}
Thus the good-function estimate follows once we establish an
$L^2$ bound for the positive-scale maximal operator. The required
estimate is the following.
\begin{lemma}\label{lem:max-L2}
	Under the normalization \eqref{eq:max-normal},
	\begin{equation}\label{eq:max-L2-sum}
		\sum_{j\ge1}\|T_j\|_{2\to2}\le CA,
		\qquad
		\|\mathcal M_P^+u\|_2\le CA\|u\|_2
		\quad (u\in L^2).
	\end{equation}
\end{lemma}

\begin{proof}
	We first consider the case
	$\Omega\in L^\infty(\mathbb S^{n-1})$.
	The proof of Lemma~\ref{lem:twocenter} uses only the upper
	bound for the rough factors and does not use their mean-zero
	condition. Thus its estimates hold for any bounded $\Omega$.
	Set $R=2^j$, and let $G_j$ be the correlation kernel of
	$T_j^*T_j$, as in \eqref{eq:correlation}.
	Choose $q=q_{\alpha_*}$ as in the proof of Theorem~\ref{prop:high} in Section~\ref{sec:high}
	so that the sum of the absolute values of its degree-$l$ coefficients
	is at least $C^{-1}$. Then
	\[
	\|q(z+R\cdot)-q(z)\|_{\rm coef}\ge C^{-1}R^l,
	\]
	Moreover, \eqref{eq:topblock} gives
	\[
	\|P(z+R\cdot,y)-P(z+R\cdot,z)\|_{\rm nc}
	\ge R^m|q(y)-q(z)|.
	\]
	
	Take $\beta=c/[4(N+c)]$, where $c,N$ are the constants in
	Lemma~\ref{lem:twocenter}.
	As in the proof of Theorem~\ref{prop:high}, split the support
	$|y-z|\le CR$ into the near-diagonal region
	$|y-z|\le R^{1-\beta}$, the polynomial sublevel region
	$|q(y)-q(z)|\le R^{l-\beta}$, and their complement.
	The first region has measure at most $CR^{n-n\beta}$.
	For the second, apply Lemma~\ref{lem:sublevel} to
	$q(z+R\cdot)-q(z)$ to obtain the bound $CR^{n-\beta/D}$.
	On their complement, the equal-scale estimate
	\eqref{eq:two-equal} gives
	\[
	|G_j(y,z)|\le C\|\Omega\|_\infty^2R^{-n}
	R^{N\beta}R^{-c(m+l-\beta)}.
	\]
	Using the direct estimate \eqref{eq:high-direct} on the first
	two regions and integrating with respect to Lebesgue measure,
	we obtain
	\[
	\sup_z\int |G_j(y,z)|\,dy
	\le C\|\Omega\|_\infty^2
	\left(
	R^{-n\beta}
	+R^{-\beta/D}
	+R^{-c(m+l)+(N+c)\beta}
	\right).
	\]
	Since $m+l\ge2$, the choice of $\beta$ gives
	\[
	\sup_z\int |G_j(y,z)|\,dy
	\le C\|\Omega\|_\infty^2R^{-\beta/D}.
	\]
	By the symmetry of the correlation kernel and Lemma~\ref{lem:schur},
	\begin{equation}\label{eq:single-ring-L2}
		\|T_j\|_{2\to2}
		\le C\|\Omega\|_\infty2^{-\delta j},
		\qquad
		\delta=\frac{\beta}{2D}>0 .
	\end{equation}
	
	We next remove the boundedness assumption.
	For general $\Omega$, define
	\(
	\Omega_j^{\rm b}
	=
	\Omega
	\one_{\{|\Omega|\le
		2^{\delta j/2}\|\Omega\|_1\}} .
	\)
	Applying \eqref{eq:single-ring-L2} to $\Omega_j^{\rm b}$ and
	using Young's inequality for the remaining part on a single annulus,
	we obtain
	\[
	\|T_j\|_{2\to2}
	\le
	C\|\Omega\|_1 2^{-\delta j/2}
	+
	C\|\Omega-\Omega_j^{\rm b}\|_1 .
	\]
	The same tail estimate as in \eqref{eq:angular-tail} yields
	\[
	\sum_{j\ge1}\|T_j\|_{2\to2}\le CA .
	\]
	Finally,
	\[
	\mathcal M_P^+u
	\le
	\sum_{j\ge1}|T_ju|,
	\]
	and Minkowski's inequality gives
	\[
	\|\mathcal M_P^+u\|_2
	\le
	\sum_{j\ge1}\|T_j u\|_2
	\le CA\|u\|_2 .\qedhere
	\]
\end{proof}

Applying Lemma~\ref{lem:max-L2} to $g$ and using
\eqref{eq:good-L2}, Chebyshev's inequality gives
\begin{equation}\label{eq:max-good}
\begin{aligned}
	\left|
	\left\{
	x:\mathcal M_P^+g(x)>\frac{\lambda}{4}
	\right\}
	\right|
	&\le
	C\lambda^{-2}\|\mathcal M_P^+g\|_2^2 \\
	&\le
	CA^2\lambda^{-2}\|g\|_2^2 \\
	&\le
	C\frac{A}{\lambda}\|f\|_1.
\end{aligned}
\end{equation}
Hence the contribution of the good function satisfies the required
weak type $(1,1)$ estimate.

\subsection{Truncation by kernel size and low positive scales for the bad functions}
\label{subsec:max-low}

Let $\mathcal Q$ denote the family of bad cubes, so that
$b=\sum_{Q\in\mathcal Q}b_Q$ and
$B_k=\sum_{L(Q)=k}b_Q$.
Retain $E,\Omega_s,\nu_s$ from Section~\ref{sec:main-proof},
with $\Omega_s,\nu_s$ defined in \eqref{eq:Omega-s}.
By the annular support, on $E^c$ we have
\[
\mathcal M_P^+b(x)
\le \sum_{s\ge s_*}\sup_{J\ge1}
\left|\sum_{j\ge J}T_jB_{j-s}(x)\right|.
\]
We collect the terms with summable $L^1$ bounds in $V$, and
separate the remaining low- and high-scale terms as follows:

\begin{align*}
	V(x)
	&=\sum_{s\ge s_*}\sum_{j>0}
	|T_{j,\nu_s}B_{j-s}(x)|\\
	&\quad+
	\sum_{s\ge s_*}\sum_{1\le j\le\varrho s}
	\bigl(|\mathcal R_{j,s}(x)|+|E_{j,s}(x)|\bigr),\end{align*} and 
	\begin{align*}L_s(x)
	&=\sup_{J\ge1}\left|
	\sum_{\substack{j\ge J\\1\le j\le\varrho s}}
	G_{j,s}(x)\right|,\\
	H_s(x)
	&=\sup_{J\ge1}\left|
	\sum_{\substack{j\ge J\\j>\varrho s}}
	T_{j,\Omega_s}B_{j-s}(x)\right|.
\end{align*}
The decomposition in Theorem~\ref{prop:lowfull}, together with the kernel truncation in Section~\ref{sec:main-proof}, gives
\begin{equation}\label{eq:max-bad-split}
	\mathcal M_P^+b
	\le V+\sum_{s\ge s_*}L_s+\sum_{s\ge s_*}H_s
	\qquad\text{on }E^c.
\end{equation}
By \eqref{eq:large-angular} and \eqref{eq:R-total},
\begin{equation}\label{eq:max-V}
	\|V\|_1\le CA\|f\|_1.
\end{equation}

For each fixed $s$, there are at most $C(1+s)$ low positive scales.
Using \eqref{eq:prop-main}, Cauchy--Schwarz, and
$\sum_j\|B_{j-s}\|_1\le2\|f\|_1$, we obtain
\begin{align}
	\|L_s\|_2
	&\le\sum_{1\le j\le\varrho s}\|G_{j,s}\|_2
	\notag\\
	&\le C\sqrt{1+s}
	\left(
	\sum_{1\le j\le\varrho s}\|G_{j,s}\|_2^2
	\right)^{1/2}
	\notag\\
	&\le C\sqrt{1+s}\,\|\Omega\|_1
	2^{-(\kappa_3/2-\eta_0)s}\sqrt{\tau\|f\|_1}.
	\label{eq:max-low}
\end{align}
This is the same summation used in \eqref{eq:low-main-total}.
The bounds are summable in $s$ by \eqref{eq:eta};
it remains to estimate $H_s$.

\Needspace{8\baselineskip}
\subsection{Maximal tails at high positive scales and completion of the proof}
\label{subsec:max-high}

\begin{lemma}[Rademacher--Menshov inequality](Theorem 10.6 of \cite{DT}) \label{lem:RM}If
$v_1,\ldots,v_N\in L^2(\mathbb R^n)$ satisfy
\begin{equation}\label{eq:max-rm-hypothesis}
	\left\|\sum_{r=1}^N\epsilon_r v_r\right\|_2\le B
	\qquad
	\text{for every }
	(\epsilon_1,\ldots,\epsilon_N)\in\{-1,1\}^N,
\end{equation}
then
\begin{equation}\label{eq:max-rm-conclusion}
	\left\|
	\max_{0\le M\le N}
	\left|\sum_{r=1}^M v_r\right|
	\right\|_2
	\le CB\log(2+N).
\end{equation}\end{lemma}
We remark that Lemma \ref{lem:RM} plays a crucial role in treating a variety of maximally truncated operators (see \cite{KL2018, KL2020, IKM2025}).
The sum for $M=0$ is understood to be zero.
We verify \eqref{eq:max-rm-hypothesis} using our high-scale
estimate and construct the finite sequence to which
\eqref{eq:max-rm-conclusion} will be applied.

\begin{lemma}[Maximal tails at high positive scales]
	\label{lem:max-high}
	For every $s\ge s_*$, there exists a measurable set $F_s$ such that
	\begin{align}
		|F_s|
		&\le C2^{-s}\sum_{Q\in\mathcal Q}|Q|,
		\label{eq:max-exception}\\
		\|\one_{F_s^c}H_s\|_2
		&\le C(1+s)\|\Omega\|_1
		2^{-(\kappa_1/2-\eta_0)s}
		\left(
		\tau\sum_{Q\in\mathcal Q}\|b_Q\|_1
		\right)^{1/2}.
		\label{eq:max-high-bound}
	\end{align}
\end{lemma}

\begin{proof}
	We first assume that $\mathcal Q$ is finite and fix $s\ge s_*$.
	Set
	\[
	j(Q)=L(Q)+s,\qquad
	\mathcal Q(s)=\{Q\in\mathcal Q:j(Q)>\varrho s\},
	\]
	and write
	\[
	h_Q=T_{j(Q),\Omega_s}b_Q
	\qquad(Q\in\mathcal Q(s)).
	\]
	Then
	\[
	H_s(x)=\sup_{J\ge1}
	\left|
	\sum_{\substack{Q\in\mathcal Q(s)\\j(Q)\ge J}}
	h_Q(x)
	\right|.
	\]
	We shall use the signed estimate from the proof of
	Theorem~\ref{prop:high} and reorganize these scale tails into
	initial partial sums to which the Rademacher--Menshov inequality
	applies.
	
	For any subfamily $\mathcal F\subseteq\mathcal Q(s)$, replace
	$b_Q$ by $\epsilon_Qb_Q$ when $Q\in\mathcal F$ and by zero
	otherwise, where $|\epsilon_Q|\le1$ are constants.
	The resulting functions satisfy the same support and mass bounds
	and the same phase-adapted cancellation.
	Consequently, the calculation in
	\eqref{eq:TT-full}--\eqref{eq:high-schur} gives
	\begin{equation}\label{eq:max-signed}
		\left\|
		\sum_{Q\in\mathcal F}\epsilon_Qh_Q
		\right\|_2^2
		\le
		C\tau\|\Omega\|_1^2
		2^{-(\kappa_1-2\eta_0)s}
		\sum_{Q\in\mathcal F}\|b_Q\|_1.
	\end{equation}
	
	To apply this estimate to the maximal tails, we first use the
	support of $h_Q$ to organize the summands.
	The annular support of $T_{j(Q),\Omega_s}$ places
	$Q\cup\supp h_Q$ in a cube centered at $c_Q$ with sidelength
	$C_n2^{j(Q)}$.
	Using $3^n$ adjacent nested dyadic grids, choose a cube $I(Q)$
	for each $Q\in\mathcal Q(s)$ such that
	\begin{equation}\label{eq:max-cover}
		Q\cup\supp h_Q\subset I(Q),\qquad
		\ell(I(Q))=2^{j(Q)+c_0},\qquad
		|I(Q)|=2^{n(s+c_0)}|Q|,
	\end{equation}
	where the integer $c_0$ is fixed, independent of $Q$ and $s$.
	The exact sidelength in \eqref{eq:max-cover} preserves the order
	of the indices $j(Q)$, while the functions $h_Q$ remain unchanged.
	
	Let $\mathscr I^t$ be the collection of distinct cubes selected
	from the $t$th grid, and combine all terms associated with
	the same cube:
	\[
	h_I=
	\sum_{\substack{Q\in\mathcal Q(s)\\I(Q)=I}}h_Q,
	\qquad
	\supp h_I\subset I.
	\]
	Define
	\begin{equation}\label{eq:max-overlap-definition}
		d_s(x)=\sum_t\sum_{I\in\mathscr I^t}\one_I(x),\qquad
		N_s=\left\lceil2^{(n+1)s}\right\rceil,\qquad
		F_s=\{x:d_s(x)>N_s\}.
	\end{equation}
	By \eqref{eq:max-cover},
	\[
	\begin{aligned}
		\int d_s(x)\,dx
		&=\sum_t\sum_{I\in\mathscr I^t}|I|\\
		&\le\sum_{Q\in\mathcal Q(s)}|I(Q)|
		\le C2^{ns}\sum_{Q\in\mathcal Q}|Q|.
	\end{aligned}
	\]
	Chebyshev's inequality therefore yields
	\[
	\begin{aligned}
		|F_s|
		&\le N_s^{-1}\int d_s(x)\,dx\\
		&\le C2^{-(n+1)s}2^{ns}
		\sum_{Q\in\mathcal Q}|Q|\\
		&=C2^{-s}\sum_{Q\in\mathcal Q}|Q|.
	\end{aligned}
	\]
	This proves \eqref{eq:max-exception}.
	Moreover, $d_s(x)\le N_s$ for $x\notin F_s$, so at each such
	point at most $N_s$ selected cubes contribute.
	
	We now use the nesting of the cubes to express the scale tails
	as initial partial sums.
	For $I\in\mathscr I^t$, set
	\[
	r_t(I)=
	\operatorname{card}\{I'\in\mathscr I^t:I\subseteq I'\},
	\qquad
	\beta_r^t=
	\sum_{\substack{I\in\mathscr I^t\\r_t(I)=r}}h_I.
	\]
	Thus $r_t(I)$ counts $I$ and all its selected ancestors.
	For a fixed $x$, the cubes in $\mathscr I^t$ containing $x$
	form a nested chain
	\[
	I_1\supsetneq I_2\supsetneq\cdots\supsetneq I_q.
	\]
	Every selected ancestor of $I_a$ also contains $x$, and hence
	\[
	r_t(I_a)=a,\qquad
	\beta_a^t(x)=h_{I_a}(x)
	\qquad(1\le a\le q),
	\]
	while the other levels vanish at $x$.
	If $x\notin F_s$, then $q\le N_s$.
	
	By \eqref{eq:max-cover}, the condition $j(Q)\ge J$ is equivalent to
	\[
	\ell(I(Q))\ge2^{J+c_0}.
	\]
	It therefore selects an initial segment of the chain, starting
	with its largest cube.
	Applying the triangle inequality to the sum over the grids,
	we obtain
	\begin{equation}\label{eq:max-prefix}
		\one_{F_s^c}(x)H_s(x)
		\le
		\sum_t\max_{0\le M\le N_s}
		\left|\sum_{r=1}^M\beta_r^t(x)\right|.
	\end{equation}
	
	Fix a grid $t$ and signs
	$\epsilon_1,\ldots,\epsilon_{N_s}\in\{-1,1\}$.
	By the definition of $\beta_r^t$,
	\[
	\sum_{r=1}^{N_s}\epsilon_r\beta_r^t
	=
	\sum_{\substack{
			Q\in\mathcal Q(s),\ I(Q)\in\mathscr I^t\\
			r_t(I(Q))\le N_s}}
	\epsilon_{r_t(I(Q))}h_Q.
	\]
	The coefficients on the right-hand side are independent of $x$,
	so \eqref{eq:max-signed} applies to this subfamily.
	Taking square roots gives
	\begin{equation}\label{eq:max-rm-verification}
		\left\|
		\sum_{r=1}^{N_s}\epsilon_r\beta_r^t
		\right\|_2
		\le
		C\|\Omega\|_1
		2^{-(\kappa_1/2-\eta_0)s}
		\left(
		\tau\sum_{Q\in\mathcal Q}\|b_Q\|_1
		\right)^{1/2}.
	\end{equation}
	Thus \eqref{eq:max-rm-hypothesis} holds with
	$v_r=\beta_r^t$, $N=N_s$, and $B$ equal to the right-hand side
	of \eqref{eq:max-rm-verification}.
	The Rademacher--Menshov inequality
	\eqref{eq:max-rm-conclusion} now yields
	\begin{align}
		&\left\|
		\max_{0\le M\le N_s}
		\left|\sum_{r=1}^M\beta_r^t\right|
		\right\|_2
		\notag\\
		&\qquad\le
		C\log(2+N_s)\,\|\Omega\|_1
		2^{-(\kappa_1/2-\eta_0)s}
		\left(
		\tau\sum_{Q\in\mathcal Q}\|b_Q\|_1
		\right)^{1/2}.
		\label{eq:max-rm-applied}
	\end{align}
	Since
	\[
	\log(2+N_s)\le C(1+s),
	\]
	taking the $L^2$ norm in \eqref{eq:max-prefix} and summing over
	the finitely many grids proves \eqref{eq:max-high-bound}.
	The factor $1+s$ is precisely the logarithmic loss in the
	Rademacher--Menshov inequality.
	
	For countable $\mathcal Q$, fix all cubes $I(Q)$ and define $F_s$
	using the full family. The same argument gives
	\eqref{eq:max-exception}, and the finite-family $L^2$ estimate
	holds on $F_s^c$, since $F_s$ contains every finite-family
	exceptional set.
	Let finite subfamilies increase to $\mathcal Q$.
	By \eqref{eq:max-signed} and $\sum_Q\|b_Q\|_1<\infty$,
	each fixed-$J$ tail converges in $L^2$.
	Passing to a common almost-everywhere convergent subsequence
	for all $J\ge1$, Fatou's lemma yields
	\eqref{eq:max-high-bound} for the full family.
\end{proof}

\medskip
\noindent
\textit{Completion of the bad-function estimate.}

Set
\[
F=\bigcup_{s\ge s_*}F_s.
\]
By \eqref{eq:CZ}, \eqref{eq:Esize},
and \eqref{eq:max-exception},
\begin{equation}\label{eq:max-EF}
	|E\cup F|
	\le |E|
	+C\sum_{s\ge s_*}2^{-s}\sum_Q|Q|
	\le CA\lambda^{-1}\|f\|_1.
\end{equation}

The estimate \eqref{eq:max-low} gives
\begin{equation}\label{eq:max-low-sum}
	\begin{aligned}
		\left\|\sum_{s\ge s_*}L_s\right\|_2
		&\le
		C\|\Omega\|_1\sqrt{\tau\|f\|_1}
		\sum_{s\ge s_*}\sqrt{1+s}\,
		2^{-(\kappa_3/2-\eta_0)s}\\
		&\le C\|\Omega\|_1\sqrt{\tau\|f\|_1}.
	\end{aligned}
\end{equation}
Since $F^c\subseteq F_s^c$ for every $s$,
the high-scale estimate gives
\begin{align}
	\left\|
	\one_{F^c}\sum_{s\ge s_*}H_s
	\right\|_2
	&\le
	\sum_{s\ge s_*}\|\one_{F_s^c}H_s\|_2
	\notag\\
	&\le
	C\|\Omega\|_1\sqrt{\tau\|f\|_1}
	\sum_{s\ge s_*}
	(1+s)2^{-(\kappa_1/2-\eta_0)s}
	\notag\\
	&\le C\|\Omega\|_1\sqrt{\tau\|f\|_1}.
	\label{eq:max-high-sum}
\end{align}
Here we used $\sum_Q\|b_Q\|_1\le2\|f\|_1$.
Both numerical series converge by the choice of $\eta_0$
in \eqref{eq:eta}.
Using $\tau=\lambda/A$ and $\|\Omega\|_1\le CA$,
we conclude that
\begin{equation}\label{eq:max-L2-total}
	\left\|\sum_{s\ge s_*}L_s\right\|_2
	+
	\left\|\one_{F^c}\sum_{s\ge s_*}H_s\right\|_2
	\le C\sqrt{\lambda A\|f\|_1}.
\end{equation}

We now apply Chebyshev's inequality to the $L^1$ remainder $V$
and the two $L^2$ main terms separately.
By \eqref{eq:max-bad-split}, \eqref{eq:max-V},
\eqref{eq:max-EF}, and \eqref{eq:max-L2-total},
\begin{align}
	|\{x:\mathcal M_P^+b(x)>\lambda/2\}|
	&\le
	|E\cup F|+\frac{C}{\lambda}\|V\|_1
	\notag\\
	&\quad+
	\frac{C}{\lambda^2}
	\left(
	\left\|\sum_{s\ge s_*}L_s\right\|_2^2
	+
	\left\|\one_{F^c}\sum_{s\ge s_*}H_s\right\|_2^2
	\right)
	\notag\\
	&\le CA\lambda^{-1}\|f\|_1.
	\label{eq:max-bad-weak}
\end{align}

Together with \eqref{eq:max-good} and
\[
\mathcal M_P^+f
\le\mathcal M_P^+g+\mathcal M_P^+b,
\]
this proves
\begin{equation}\label{eq:max-tail-weak}
	\lambda|\{x:\mathcal M_P^+f(x)>\lambda\}|
	\le CA\|f\|_1.
\end{equation}
The reduction \eqref{eq:max-dyadic} and the weak type $(1,1)$
estimate for $M_\Omega$ therefore yield
\begin{equation}\label{eq:max-infty-weak}
	\lambda|\{x:T_{P,\infty}^*f(x)>\lambda\}|
	\le CA\|f\|_1.
\end{equation}
Combining \eqref{eq:max-infty-weak},
\eqref{eq:max-local-weak}, and \eqref{eq:max-fixed-weak}
with the three-part decomposition \eqref{eq:max-three},
we obtain \eqref{eq:max-target}. This completes the proof of Theorem \ref{thm1}.

\makeatletter
\let\@tocwrite\@gobbletwo
\makeatother
\section*{Statements and Declarations}

\noindent\textbf{Statement on the use of AI.}
We have worked on this paper for several years. The mathematical ideas and results in the main body of this paper are due entirely to the
authors, who developed the arguments and wrote the proofs.
We used  GPT-5.6 Sol to help check the calculations underlying our
approach to the two-center rough oscillatory estimates and
to refine the wording of the manuscript.  This was the only mathematical input from AI that we later used. 
We take full responsibility for the mathematical content of the paper.
\smallskip

\noindent\textbf{Ethical Approval}

The declaration for ethical approval is not applicable.

\noindent\textbf{Competing interests}

The authors declare no competing interests.

\noindent\textbf{Availability of data and materials}

Data sharing is not applicable to this article as no datasets were generated or analysed during the current study.

\begingroup
\makeatletter
\let\@tocwrite\@gobbletwo
\makeatother

\endgroup

\end{document}